\documentclass{amsart}
\usepackage{amsmath}
\usepackage{amsthm}
\usepackage{amssymb}
\usepackage{euscript}
\usepackage{mathrsfs}
\usepackage{bm}
\usepackage{enumitem}
\usepackage[dvipsnames]{xcolor}
\usepackage{tikz}
\usepackage{mathtools}
\usepackage{float}
\usepackage{hyperref}
\usepackage{boldline}
\usepackage{indentfirst}
\usepackage{environ}
\usetikzlibrary{positioning, arrows}

\makeatletter
\newsavebox{\measure@tikzpicture}
\NewEnviron{scaletikzpicturetowidth}[1]{%
  \def\tikz@width{#1}%
  \def\tikzscale{1}\begin{lrbox}{\measure@tikzpicture}%
  \BODY
  \end{lrbox}%
  \pgfmathparse{#1/\wd\measure@tikzpicture}%
  \edef\tikzscale{\pgfmathresult}%
  \BODY
}
\makeatother

\hypersetup{
    colorlinks=true,
    linkcolor=[RGB]{0,0,128},
    filecolor=magenta,
    urlcolor=cyan,
    citecolor = [RGB]{128,0,128}
}

\newcommand{\myref}[2]{\hyperref[#2]{#1 \ref*{#2}}}
\newcommand{\myrefT}[1]{\hyperref[#1]{Theorem \ref*{#1}}}
\newcommand{\myrefP}[1]{\hyperref[#1]{Proposition \ref*{#1}}}
\newcommand{\myrefL}[1]{\hyperref[#1]{Lemma \ref*{#1}}}
\newcommand{\myrefD}[1]{\hyperref[#1]{Definition \ref*{#1}}}
\newcommand{\myrefS}[1]{\hyperref[#1]{Section \ref*{#1}}}
\newcommand{\myrefC}[1]{\hyperref[#1]{Corollary \ref*{#1}}}
\newcommand{\myrefR}[1]{\hyperref[#1]{Remark \ref*{#1}}}
\newcommand{\myrefn}[3]{\hyperref[#2]{#1 \ref*{#2} (#3)}}

\definecolor{goldenpoppy}{rgb}{0.99, 0.76, 0.0}
\definecolor{darkblue}{rgb}{0.0, 0.0, 0.55}
\tikzstyle{close} = [inner sep = 0cm]
\tikzstyle{legendLine} = [rectangle, inner sep=0pt, minimum height=0pt, minimum width=24pt]

\tikzstyle{dynkinLabelStyle} = [scale = .9, every node/.style={transform shape}]

\newcommand{\dynkinLabel}[2] {
\begin{tikzpicture}[close, dynkinLabelStyle]
   \node (DyNode){#1};
   \node [right = 0cm of DyNode] {\Large $\euscr #2$};
\end{tikzpicture}
}

\newcommand{\dynkinLabelWeyl}[2] {
\begin{tikzpicture}[close, dynkinLabelStyle]
   \node (DyNode){#1};
   \node [right = 0cm of DyNode] {\Large $#2$};
\end{tikzpicture}
}

\newcommand{\dynkinLabelWeylHuge}[2] {
\begin{tikzpicture}[close, dynkinLabelStyle]
   \node (DyNode){#1};
   \node [right = 0cm of DyNode] {\huge $#2$};
\end{tikzpicture}
}

\newcommand{\verticalToNode}[1]{1.4cm of #1}

\newcommand{\upLine}[4]{\draw[-, thick, #1, #2] (#3.north) -- (#4.south);}

\newcommand{\DynkinDFour} {
  \draw[blue, very thick] (.707, .707) -- (0, 0);
  \draw[blue, very thick] (.707, -.707) -- (0, 0);
  \draw[blue, very thick] (-1 ,0) -- (0, 0);
  \filldraw[blue] (0, 0) circle (4pt);
  \filldraw[blue] (.707, .707) circle (4pt);
  \filldraw[blue] (.707, -.707) circle (4pt);
  \filldraw[blue] (-1, 0) circle (4pt);
}

\newcommand{\redDot}[1] {
  \ifcase#1\relax
  \or
    \filldraw[red] (.707, .707) circle (4.9pt);
  \or
    \filldraw[red] (.707, -.707) circle (4.9pt);
  \or
    \filldraw[red] (0, 0) circle (4.9pt);
  \or
    \filldraw[red] (-1, 0) circle (4.9pt);
  \fi
}

\newcommand{\otherRedDots}[1] {
  \ifcase#1\relax
  \or
    \filldraw[red] (.707, -.707) circle (4.9pt);
    \filldraw[red] (-1, 0) circle (4.9pt);
  \or
    \filldraw[red] (.707, .707) circle (4.9pt);
    \filldraw[red] (-1, 0) circle (4.9pt);
  \or
  \or
    \filldraw[red] (.707, .707) circle (4.9pt);
    \filldraw[red] (.707, -.707) circle (4.9pt);
  \fi
}

\newcommand{\TypeAPicture}[1]{
  \begin{tikzpicture}
    \DynkinDFour
    \redDot{3}
    \redDot{#1}
  \end{tikzpicture}
}

\newcommand{\TypeBPicture}[1]{
  \begin{tikzpicture}
    \DynkinDFour
    \otherRedDots{#1}
  \end{tikzpicture}
}

\newcommand{\TypeCPicture}{
  \begin{tikzpicture}
    \DynkinDFour
    \redDot{1}
    \redDot{2}
    \redDot{4}
  \end{tikzpicture}
}

\newcommand{\TypeDPicture}{
  \begin{tikzpicture}
    \DynkinDFour
    \redDot{3}
  \end{tikzpicture}
}

\newcommand{\TypeIPicture}[1]{
  \begin{tikzpicture}
    \DynkinDFour
    \redDot{#1}
  \end{tikzpicture}
}

\newcommand{\TypeFPicture}[1]{
  \begin{tikzpicture}
    \DynkinDFour
    \redDot{3}
    \otherRedDots{#1}
  \end{tikzpicture}
}

\newcommand{\doubleDist}{1.5pt}

\definecolor{indigo}{rgb}{0.29, 0.0, 0.51}

\newcommand{\DynkinATwo} {
  \draw[blue, very thick] (1, 0) -- (2, 0);
  \filldraw[blue] (1, 0) circle (4.5pt);
  \filldraw[blue] (2, 0) circle (4.5pt);
}

\newcommand{\redADot}[1] {
  \filldraw[red] (#1, 0) circle (5pt);
}

\newcommand{\IdentityPictureATwo}{
  \begin{tikzpicture}
    \DynkinATwo
  \end{tikzpicture}
}

\newcommand{\SPictureATwo}{
  \begin{tikzpicture}
    \DynkinATwo
    \redADot{1}
  \end{tikzpicture}
}

\newcommand{\TPictureATwo}{
  \begin{tikzpicture}
    \DynkinATwo
    \redADot{2}
  \end{tikzpicture}
}

\newcommand{\LongPictureATwo}{
  \begin{tikzpicture}
    \DynkinATwo
    \redADot{1}
    \redADot{2}
  \end{tikzpicture}
}

\newcommand{\xSA}{3}
\newcommand{\ySA}{2.5}

\newcommand{\upLineLabelA}[6]{\draw[-, thick, shorten <=0.2cm, shorten >=0.2cm, #1, #2] (#3 * \xSA, #4 * \ySA) -- node[#6]{\Large $#5$} (#3 * \xSA, #4 * \ySA + \ySA);}
\newcommand{\leftLineLabelA}[6]{\draw[-, thick, shorten <=0.4cm, shorten >=.4cm, #1, #2] (#3 * \xSA, #4 * \ySA) -- node[#6]{\Large $#5$} (#3 * \xSA - \xSA, #4 * \ySA + \ySA);}
\newcommand{\rightLineLabelA}[6]{\draw[-, thick, shorten <=.4cm, shorten >=0.4cm, #1, #2] (#3 * \xSA, #4 * \ySA) -- node[#6]{\Large $#5$}(#3 * \xSA + \xSA, #4 * \ySA + \ySA);}

\newcommand{\DynkinBTwo} {
  \draw[blue, very thick, double, double distance = \doubleDist] (1, 0) -- (2, 0);
  \filldraw[blue] (1, 0) circle (4.5pt);
  \filldraw[blue] (2, 0) circle (4.5pt);
}

\newcommand{\redBDot}[1] {
  \filldraw[red] (#1, 0) circle (5pt);
}

\newcommand{\IdentityPictureBTwo}{
  \begin{tikzpicture}
    \DynkinBTwo
  \end{tikzpicture}
}

\newcommand{\SPictureBTwo}{
  \begin{tikzpicture}
    \DynkinBTwo
    \redBDot{1}
  \end{tikzpicture}
}

\newcommand{\TPictureBTwo}{
  \begin{tikzpicture}
    \DynkinBTwo
    \redBDot{2}
  \end{tikzpicture}
}

\newcommand{\LongPictureBTwo}{
  \begin{tikzpicture}
    \DynkinBTwo
    \redBDot{1}
    \redBDot{2}
  \end{tikzpicture}
}

\newcommand{\xSB}{3}
\newcommand{\ySB}{2.5}

\newcommand{\dynkinLabelABTwo}[2] {
\begin{tikzpicture}[close]
   \node (DyNode){#1};
   \node [right = .1cm of DyNode] {\Large $#2$};
\end{tikzpicture}
}

\newcommand{\dynkinLabelLeftABTwo}[2] {
\begin{tikzpicture}[close]
   \node (DyNode){#1};
   \node [left = .1cm of DyNode] {\Large $#2$};
\end{tikzpicture}
}

\newcommand{\DynkinESix} {
  \draw[blue, very thick] (1, 0) -- (5, 0);
  \draw[blue, very thick] (3, 0) -- (3, 1);
  \filldraw[blue] (1, 0) circle (4pt);
  \filldraw[blue] (2, 0) circle (4pt);
  \filldraw[blue] (3, 0) circle (4pt);
  \filldraw[blue] (4, 0) circle (4pt);
  \filldraw[blue] (5, 0) circle (4pt);
  \filldraw[blue] (3, 1) circle (4pt);
}

\newcommand{\DynkinESixLabelBourbaki} {
  \draw[blue, very thick] (1, 0) -- (5, 0);
  \draw[blue, very thick] (3, 0) -- (3, 1);
  \filldraw[blue] (1, 0) circle (4pt) node [below=.2cm, black]{\small\pgfmathprintnumber{1}};
  \filldraw[blue] (2, 0) circle (4pt) node [below=.2cm, black]{\small\pgfmathprintnumber{3}};
  \filldraw[blue] (3, 0) circle (4pt) node [below=.2cm, black]{\small\pgfmathprintnumber{4}};
  \filldraw[blue] (4, 0) circle (4pt) node [below=.2cm, black]{\small\pgfmathprintnumber{5}};
  \filldraw[blue] (5, 0) circle (4pt) node [below=.2cm, black]{\small\pgfmathprintnumber{6}};
  \filldraw[blue] (3, 1) circle (4pt) node [above=.2cm, black]{\small\pgfmathprintnumber{2}};
}

\newcommand{\redEDot}[1] {
  \filldraw[red] (#1, 0) circle (7pt);
}

\newcommand{\topDotESix} {
  \filldraw[red] (3, 1) circle (7pt);
}

\newcommand{\TypeCPictureE}{
  \begin{tikzpicture}
    \DynkinESix
    \redEDot{1}
    \redEDot{2}
    \redEDot{4}
    \redEDot{5}
    \topDotESix
  \end{tikzpicture}
}

\newcommand{\TypeAOnePictureE}{
  \begin{tikzpicture}
    \DynkinESix
    \redEDot{1}
    \redEDot{2}
    \redEDot{3}
    \redEDot{5}
  \end{tikzpicture}
}

\newcommand{\TypeAFivePictureE}{
  \begin{tikzpicture}
    \DynkinESix
    \redEDot{1}
    \redEDot{3}
    \redEDot{4}
    \redEDot{5}
  \end{tikzpicture}
}

\newcommand{\TypeBOnePictureE}{
  \begin{tikzpicture}
    \DynkinESix
    \redEDot{1}
    \redEDot{2}
    \redEDot{4}
    \topDotESix
  \end{tikzpicture}
}

\newcommand{\TypeBFivePictureE}{
  \begin{tikzpicture}
    \DynkinESix
    \redEDot{2}
    \redEDot{4}
    \redEDot{5}
    \topDotESix
  \end{tikzpicture}
}

\newcommand{\TypeDOnePictureE}{
  \begin{tikzpicture}
    \DynkinESix
    \redEDot{2}
    \redEDot{3}
    \redEDot{5}
  \end{tikzpicture}
}

\newcommand{\TypeDFivePictureE}{
  \begin{tikzpicture}
    \DynkinESix
    \redEDot{1}
    \redEDot{3}
    \redEDot{4}
  \end{tikzpicture}
}

\newcommand{\ESixPicture}[4]{
  \begin{tikzpicture}
    \DynkinESix
    \redEDot{#1}
    \redEDot{#2}
    \redEDot{#3}
    \redEDot{#4}
  \end{tikzpicture}
}

\newcommand{\ESixTopPicture}[3]{
  \begin{tikzpicture}
    \DynkinESix
    \redEDot{#1}
    \redEDot{#2}
    \redEDot{#3}
    \topDotESix
  \end{tikzpicture}
}

\newcommand{\xSE}{4.5}
\newcommand{\ySE}{2.7}
\newcommand{\upLineE}[4]{\draw[-, thick, shorten <=0.4cm, shorten >=0.3cm, #1, #2] (#3 * \xSE, #4 * \ySE) -- (#3 * \xSE, #4 * \ySE + \ySE);}
\newcommand{\leftLineE}[4]{\draw[-, thick, shorten <=0.3cm, shorten >=.6cm, #1, #2] (#3 * \xSE, #4 * \ySE) -- (#3 * \xSE - \xSE, #4 * \ySE + \ySE);}
\newcommand{\rightLineE}[4]{\draw[-, thick, shorten <=.3cm, shorten >=0.6cm, #1, #2] (#3 * \xSE, #4 * \ySE) -- (#3 * \xSE + \xSE, #4 * \ySE + \ySE);}

\newcommand{\upLineLabel}[4]{\draw[-, thick, #1] (#2.north) -- node[right]{$#4$} (#3.south);}

\newcommand{\xSLT}{3.7}
\newcommand{\ySLT}{3}
\newcommand{\upLineT}[4]{\draw[-, thick, shorten <=0.4cm, shorten >=0.4cm, #1, #2] (#3 * \xSLT, #4 * \ySLT) -- (#3 * \xSLT, #4 * \ySLT + \ySLT);}
\newcommand{\leftLineT}[4]{\draw[-, thick, shorten <=0.5cm, shorten >=1cm, #1, #2] (#3 * \xSLT, #4 * \ySLT) -- (#3 * \xSLT - \xSLT, #4 * \ySLT + \ySLT);}
\newcommand{\rightLineT}[4]{\draw[-, thick, shorten <=1cm, shorten >=0.5cm, #1, #2] (#3 * \xSLT, #4 * \ySLT) -- (#3 * \xSLT + \xSLT, #4 * \ySLT + \ySLT);}

\newcommand{\xSLF}{3.5}
\newcommand{\ySLF}{3}
\newcommand{\upLineF}[4]{\draw[-, thick, shorten <=0.4cm, shorten >=0.4cm, #1, #2] (#3 * \xSLF, #4 * \ySLF) -- (#3 * \xSLF, #4 * \ySLF + \ySLF);}
\newcommand{\leftLineF}[4]{\draw[-, thick, shorten <=0.5cm, shorten >=1cm, #1, #2] (#3 * \xSLF, #4 * \ySLF) -- (#3 * \xSLF - \xSLF, #4 * \ySLF + \ySLF);}
\newcommand{\rightLineF}[4]{\draw[-, thick, shorten <=1cm, shorten >=0.5cm, #1, #2] (#3 * \xSLF, #4 * \ySLF) -- (#3 * \xSLF + \xSLF, #4 * \ySLF + \ySLF);}

\tikzstyle{dynkinLabelWeylStyle} = [scale = .8, every node/.style={transform shape}]

\newcommand{\lineArrowCellT}[4]{\draw[->, thick, Plum, shorten <=.7cm, shorten >=0.9cm] (#1 * \xSLT, #2 * \ySLT) -- (#3 * \xSLT, #4 * \ySLT);}

\newcommand{\lineArrowCellTMiddle}[4]{\draw[->, thick, Plum, shorten <=.4cm, shorten >=0.4cm] (#1 * \xSLT, #2 * \ySLT) -- (#3 * \xSLT, #4 * \ySLT);}

\newcommand{\bendCellBig}[3]{\draw[bend #3 = 60, ->, thick, Plum] (#1) to (#2);}

\newcommand{\bendCellSmall}[3]{\draw[bend #3 = 30, ->, thick, Plum] (#1) to (#2);}

\newcommand{\bendCell}[3]{\draw[bend #3 = 15, ->, thick, Plum] (#1) to (#2);}
\newcommand{\bendCellSmallest}[3]{\draw[bend #3 = 5, ->, thick, Plum] (#1) to (#2);}

\tikzstyle{impliesArrow} = [->, double distance = .13cm, shorten <=.5cm, shorten >=.5cm, >=implies]
\newcommand{\diamondX}{5}
\newcommand{\diamondY}{3}

\DeclarePairedDelimiter\abs{\lvert}{\rvert}
\DeclareMathOperator{\pair}{pair}
\newcommand{\euscr}{\EuScript}

\newcommand{\ti}{$\tau\textnormal{-invariant}$}

\newcommand{\dsim}{\underset{d}{\sim}}
\newcommand{\mutilde}{\tilde \mu}
\newcommand{\refSet}{\{1,2,4\}}
\newcommand{\bigRefSet}{\{s_1, s_2, s_4\}}
\newcommand{\DFourSet}{\{s_1, s_2, s_3, s_4\}}
\newcommand{\SDFour}{{J_0}}
\newcommand{\WDFour}{W_0}
\newcommand{\tDFour}{\tau_0}

\newcommand{\scA}{\EuScript A}
\newcommand{\scB}{\EuScript B}
\newcommand{\scC}{\EuScript C}
\newcommand{\scD}{\EuScript D}

\newcommand{\equivR}{\underset{R}{\sim}}
\newcommand{\equivL}{\underset{L}{\sim}}
\newcommand{\leqL}{\underset{L}{\leq}}
\newcommand{\leqR}{\underset{R}{\leq}}
\newcommand{\equivF}{\underset{\euscr F}{\sim}}
\newcommand{\equivGTF}{\underset{GT\euscr F}{\sim}}
\newcommand{\equivGTFPrime}{\underset{GT\euscr F'}{\sim}}
\newcommand{\equivUnder}[1]{\underset{#1}{\sim}}

\newcommand{\tauL}{\tau_L}
\newcommand{\tauR}{\tau_R}
\newcommand{\Dst}{D_{s,t}(W)}
\newcommand{\DstL}{D^L_{s,t}(W)}
\newcommand{\DstR}{D^R_{s,t}(W)}
\newcommand{\Tst}{T_{s,t}}
\newcommand{\TstL}{T^L_{s,t}}
\newcommand{\TstR}{T^R_{s,t}}

\newcommand{\citeKL}{\cite{kazhdan_lusztig_1979}}
\newcommand{\refDFourFigures}{\autoref{fig:type10Full}, \autoref{fig:type14aFull}, and \autoref{fig:type14bFull}}

\numberwithin{equation}{section}

\newtheorem{theorem}{Theorem}[section]
\newtheorem{corollary}[theorem]{Corollary}
\newtheorem{proposition}[theorem]{Proposition}
\newtheorem{lemma}[theorem]{Lemma}
\theoremstyle{definition}
\newtheorem{definition}[theorem]{Definition}
\newtheorem{notation}[theorem]{Notation}
\theoremstyle{remark}
\newtheorem*{remark}{Remark}
\newtheorem{remarkNumbered}[theorem]{Remark}

\begin{document}
  \title{Edge Transport from Parabolic Subgroups of Type $D_4$}
  \author{Devra Garfinkle Johnson}
  \email{devrajoh@sas.upenn.edu}
  \begin{abstract}
    This paper is part of the program to classify Kazhdan-Lusztig cells for Weyl groups of type $D_n$.
    We prove analogous results to those of section 4 of Kazhdan-Lusztig's original paper, this time related to a parabolic subgroup of type $D_4$.
    We also show how this is used in the definition of the generalized $\tau$-invariant.
  \end{abstract}

  \maketitle
  \begingroup
    \hypersetup{linkcolor=black}
    \setcounter{tocdepth}{1}
    \tableofcontents
  \endgroup
  \section*{Introduction}

  This paper is the next paper in the series \cite{garfinkle_1990,garfinkle_1992,garfinkle_1993}.
  The goal achieved in those three papers was the classification of primitive ideals in the universal enveloping algebra of a complex simple Lie algebra of type $B_n$ or $C_n$.
  This problem was solved by classifying such primitive ideals by domino tableaux and by their generalized $\tau$-invariant.

  However, since the proof of the Kazhdan-Lusztig conjectures, the problem of classifying such primitive ideals is known to be equivalent to that of classifying left cells in the corresponding Weyl group.
  Moreover, by now there are many people studying Kazhdan-Lusztig cells for various other motivations.
  So, if possible, it's desirable to have a proof of the classification of left cells in Weyl groups which does not depend on the proof of the Kazhdan-Lusztig conjectures.
  Kazhdan-Lusztig provided that in their original paper, \citeKL, for the Weyl group of type $A_n$.
  Their main ingredient was Theorem 4.2 of \citeKL, which we're calling an ``edge transport theorem''.

  To carry this program out for other Weyl groups, what's needed first is the appropriate version of Theorem 4.2 of \citeKL.
  For types $B_n$ and $C_n$, the necessary edge transport theorem is already known, and appears in \cite{lusztig_1985}.
  This paper proves the version needed for the Weyl group of type $D_n$ (see \myrefS{sec:main}).

  \subsection*{Overall Context}

  Let $(W, S)$ be a Coxeter system.
  In \citeKL, Kazhdan-Lusztig defined the left (and right) cell equivalence relation on $W$.
  This equivalence relation is defined as coming from certain edges of the $W$ graph, where the edges are defined in terms of the Kazhdan-Lusztig polynomials.
  Though these polynomials are in principle computable, by recursion, the computation quickly becomes too large to manage.
  So, in practice, one would like to identify a smaller and easier-to-compute set of edges which yield the same equivalence classes.
  Once having found such a set, there remains the task of
  showing that they generate the same equivalence classes.

  For the latter task, we need some way of showing that two elements of the Coxeter group are not in the same left cell.
  A starting point is \cite[Proposition 2.4]{kazhdan_lusztig_1979}.
  This says that the right descent set, or \ti, is constant on left cells.
  This is a fairly weak requirement, but it can be made stronger.
  Basically, what we want is to have a lot of maps which take left cells to left cells.
  Then, if we have two members of the same left cell, we can apply one of these maps.
  The results will still have the same right \ti.
  Or, conversely, if we have two elements which are not in the same left cell, we can hope to find a sequence of such maps such that the results of applying the sequence to both elements do not have the same right \ti.
  This is the idea behind the generalized $\tau$-invariant.

  Edge transport theorems are a crucial ingredient in this program.
  Basically, an edge transport theorem says that, under certain conditions, if we're given an edge connecting two elements of the Coxeter group, then there is also an edge connecting two other elements.
  We can then apply the theorem to some of the edges used to define the left cell equivalence relation.

  This paper has two focuses.  The first is to prove the $D_4$ edge transport theorem, \myrefT{thm:mainA}.
  The second is to prove applications of edge transport theorems, most notably the generalized $\tau$-invariant.
  Here the main theorems are \myrefT{thm:genTau} and \myrefT{thm:genTauConclusion}, which say that the generalized $\tau$-invariant, when defined with respect to the edge transport functions which we are studying, is a weaker equivalence relation than that of being in the same (left or right) cell.
  Along the way, we'll also prove the analogue of (a stronger version of) \cite[Corollary 4.3]{kazhdan_lusztig_1979}, namely \myrefP{prop:edgeTransportLeqB}.
  (See also \myrefP{prop:edgeTransportKLOrderA} for the stronger version of \cite[Corollary 4.3]{kazhdan_lusztig_1979}.)

  \subsection*{Classification of Left Cells for Classical Weyl Groups}

  \begin{figure}
    \centering
    \begin{tikzpicture}
      \node[draw, inner sep=0.5cm, rectangle] (Left) at (-\diamondX, 0) {$y \equivUnder{\euscr F L} w$};
      \node[draw, inner sep=0.5cm, rectangle] (Top) at (0, \diamondY) {$y \equivL w$};
      \node[draw, inner sep=0.5cm, rectangle] (Right) at (\diamondX, 0) {$y \equivUnder{GTR} w$};
      \node[draw, inner sep=0.5cm, rectangle] (Bottom) at (0, -\diamondY) {$y \equivUnder{Tab R} w$};
      \draw[impliesArrow] (Left) -- (Top) node[midway, above=.13cm, xshift=-.2cm] {\large I};
      \draw[impliesArrow] (Top) -- (Right) node[midway, above=.15cm, xshift=.15cm] {\large II};
      \draw[impliesArrow] (Right) -- (Bottom) node[midway, above=.15cm, xshift=-.2cm] {\large III};
      \draw[impliesArrow] (Bottom) -- (Left) node[midway, above=.15cm, xshift=.3cm] {\large IV};
    \end{tikzpicture}
    \caption{Classification of Left Cells for Classical Weyl Groups}
    \label{fig:classification}
  \end{figure}
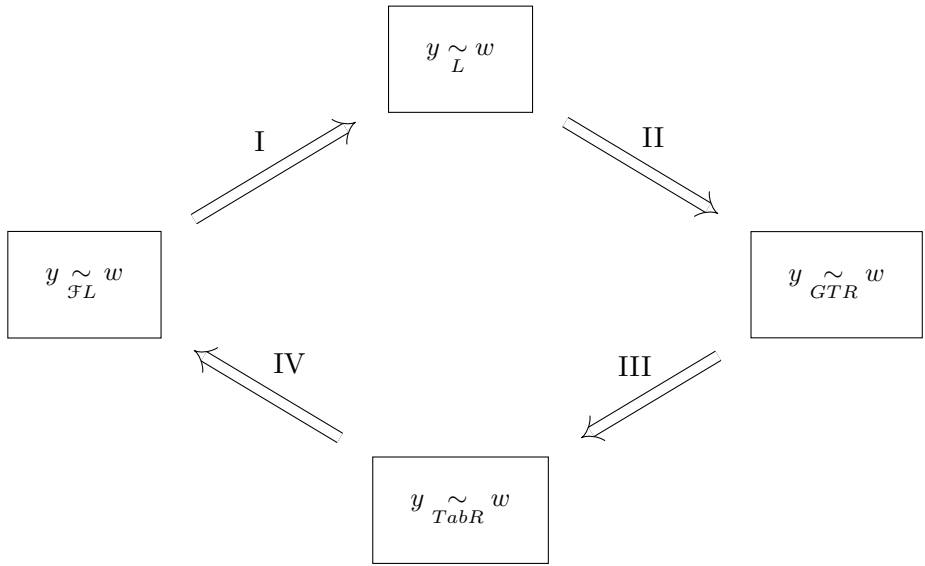

  Since classification of left cells is the author's main motivation for writing this paper, we'll describe here how this paper fits in to that result.
  The proof of the classification of left cells in classical Weyl groups follows in each case the pattern shown in \autoref{fig:classification}.
  This figure shows that, in addition to the left cell equivalence shown in the top box, there are three other equivalence relations on the Weyl group which are shown to coincide with it.
  We'll describe here each of the boxes and each of the arrows.

  The box on the left shows an equivalence relation generated by a family of functions.  (See \myrefD{def:equivalenceRelationF} and \myrefD{def:equivalenceRelationF2}.)
  For type $A_n$ these are the $*$ operations of \citeKL, which we'll call Knuth maps, acting on the left (see \myrefS{sec:KnuthMaps}).
  For type $BC_n$ we add to that family the $B_2$ maps (see \myrefS{sec:B2Maps}).
  For type $D_n$ we add the $D_4$ maps (see \myrefS{sec:D4Maps}).
  The box on the right is a generalized \ti\ equivalence relation (see \myrefS{sec:genTauA} and \myrefS{sec:genTauAB}).
  For a generalized $\tau$-invariant, we need to specify a family of maps.
  These will be the same families as listed above, this time acting on the right.

  The box on the bottom is an equivalence relation related to tableaux.
  That is, in each case we associate to an element of the Weyl group a pair of tableaux.
  For type $A_n$, this association is performed by the well-known Robinson-Schensted algorithm, or RSK.
  For the other classical Weyl groups, the association starts with the domino Robinson-Schensted algorithm (introduced by the author in \cite{garfinkle_1990}).
  After that, though, we need to apply another procedure to both tableaux to bring them to a special shape.
  (See \cite{garfinkle_1990} for both of these procedures.)
  The equivalence relation $y \equivUnder{Tab R} w$ specifies that for the two elements $y, w \in W$, the right-hand tableaux of the two pairs coincide.

  This paper is about the two top arrows of \autoref{fig:classification}.
  The material about type $D_n$ is new, but we will also recall the material for types $A_n$ and $BC_n$.
  The arrow labeled I is relatively straightforward.
  It says that certain maps defined on subsets of $W$ stay within cells.
  For type $D_n$, see \myrefP{prop:D4KLCell}.
  The arrow labeled II is more complicated.
  It requires an edge transport theorem and then some.
  For type $D_n$, see \myrefT{thm:genTau} and \myrefS{sec:D4Maps}.

  We'll outline here the material contained in the bottom two arrows.
  We are working with the various maps as listed above, that is, the Knuth maps, the $B_2$ maps, and the $D_4$ maps.
  The first step is to define such maps on pairs of tableaux.
  The next step is to show that the algorithm to associate a pair of tableaux to a Weyl group element commutes with the maps.
  For type $A_n$ this is well-known.
  For type $BC_n$, see \cite[Theorem~2.1.19 and Theorem~2.3.8]{garfinkle_1992}.
  For type $D_n$, this will be proved in the next paper in this series.
  After that, the two bottom arrows can be proved on pairs of tableaux.
  In the case of type $BC_n$, for implication IV, see
  \cite[Theorem~3.2.2]{garfinkle_1993}.
  For implication III, see
  \cite[Theorem~3.4.17]{garfinkle_1993}.
  For type $A_n$, this is also known. (For example, the $A_n$ versions of the type $BC_n$ proofs, which are much easier, can be used here.) For type $D_n$, this will be proved in the future papers in this series.  See for example \cite{ariki_2000} for an exposition of the overall $A_n$ situation.

  The net result of this is that there are three other equivalence relations which are the same as the left cell equivalence.
  Each are in principle useful.
  The one on the left shows that we can use fewer and easier-to-understand edges in place of the full set required to define the left cell equivalence.
  The one on the right, the generalized $\tau$-invariant, is used in type $A_n$ in the context of primitive ideals to show that we can use a tableau algorithm to compute annihilators of irreducible Harish-Chandra modules.
  (\cite{garfinkle_1993_a}.)
  The one on the bottom gives you a tableau which represents the cell.
  Some properties of the left cell are easier to read from the tableau associated to it.
  One example is the descent set.
  Further, the tableau makes it easier to visualize the generalized $\tau$-invariant.

  \subsection*{Other Applications}
  The generalized \ti\ and the more direct use of the edge transport theorems called the ``technique of strings'' have both been used in the classification of left cells for low-rank affine Weyl groups.
  See \cite{lusztig_1985}, \cite{bedard_1986}, and \cite{du_1988} for example.
  We discuss this some in \myrefS{sec:techniques}.

  \subsection*{Organization of the Paper}
  The paper is organized as follows.
  \myrefS{sec:Preliminaries} recalls and/or proves the results which we'll need about Coxeter groups and Kazhdan-Lusztig polynomials.
  \myrefS{sec:KnuthMaps} recalls the definitions and theorems about Knuth maps which we'll need.
  As part of that, we state the first edge transport theorem, which is the model for the two which follow.
  However, the second edge transport theorem, the one coming from a parabolic subgroup of type $B_2$, is a better model for the edge transport theorem which is the subject of this paper.
  So, we'll present that next, in \myrefS{sec:B2Maps}.

  \myrefS{sec:D4cells} describes the left cells in $D_4$ which are of interest to us.
  \myrefS{sec:main} proves the edge transport theorem coming from a parabolic subgroup of type $D_4$, \myrefT{thm:main}.

  To go from an edge transport theorem to a useful generalized $\tau$-invariant, we need some additional properties of the maps which we are using to define the generalized $\tau$-invariant.
  Since we'll be doing this three times, we'll formalize this with some definitions which we can reuse.
  We'll begin that formalism in \myrefS{sec:edgeTransportFunctions}.
  Along the way, we'll prove \myrefP{prop:edgeTransportKLOrderA}.
  We'll define the first version of the generalized \ti\ in \myrefS{sec:genTauA} and discuss how it is used.

  \myrefS{sec:edgeTransportPairs} continues the formalism of \myrefS{sec:edgeTransportFunctions}, this time to encompass the more complicated functions associated with the $B_2$ and $D_4$ edge transport theorems.
  \myrefS{sec:genTauAB} defines the more complicated generalized $\tau$-invariant which uses these functions, and proves \myrefT{thm:genTau}, which says that, in our circumstances, the generalized \ti\ is a weaker equivalence relation than that of being in the same (left or right) cell.

  \myrefS{sec:D4Maps} defines the $D_4$ maps and shows that they satisfy the conditions defined in \myrefS{sec:edgeTransportPairs}, and thus that the theorems of that section and the next also apply to the $D_4$ maps.  \myrefT{thm:genTauConclusion} summarizes the results on the generalized \ti\ as it applies to the maps which we have been considering.

  \myrefS{sec:otherMaps} introduces in our context some related maps, derived from those described in \myrefS{sec:edgeTransportPairs}, to which  the definitions, and thus the conclusions, of \myrefS{sec:edgeTransportFunctions} apply.
  Finally, \myrefS{sec:techniques} briefly discusses another application of the edge transport theorems previously proved, namely Lusztig's ``technique of strings''.
  We show how this can be extended using the current edge transport theorem.

  Note:  The $D_4$ maps, in the context of primitive ideals, were studied in \cite{garfinkle_vogan_1992}.
  \myrefT{thm:genTau}, in the context of the $B_2$ maps and affine Weyl groups, appears in \cite[Proposition 1.13]{du_1988}.

  \section{Preliminaries about Kazhdan-Lusztig Polynomials and Parabolic Subgroups}
  \label{sec:Preliminaries}

  In this section we'll first recall the facts about Kazhdan-Lusztig polynomials which we'll use in this paper.
  We'll next recall some basic facts about parabolic subgroups.
  Mostly, we need to know that, for a parabolic subgroup of a Coxeter group, every coset has a unique representative of minimal length.
  We'll use that in \myrefP{prop:KLPolysParabolic} to show that the Kazhdan-Lustig polynomial relating two elements in the same coset is the same as that coming from the parabolic subgroup.

  Let $(W, S)$ be a Coxeter system.
  For $y, w \in W$ with $y \leq w$ (Bruhat order), Kazhdan-Lusztig in \citeKL\ defined polynomials, $P_{y, w}(q)$.
  We have $P_{w, w} = 1$ for any $w \in W$.
  For $y < w$, the degree of $P_{y, w}$ is less than or equal to $d(y, w) = (l(w) - l(y) - 1) / 2$.
  If the degree of $P_{y, w}$ is equal to $d(y, w)$, write $y \prec w$.
  If $y, w \in W$ with $y \nleq w$ set $P_{y, w} = 0$.
  Similarly, for $y, w \in W$ if $y \leq w$ let $\mu(y, w)$ be the coefficient of $q^{d(y, w)}$ in $P_{y,w}$, otherwise set $\mu(y, w) = 0$.

  We'll also use the $\mutilde$ notation of \cite{lusztig_1985}, that is, $\mutilde(y, w)$ is defined by:
  \begin{enumerate}
    \item If $y \leq w$ then $\mutilde(y, w) = \mu(y, w)$.
    \item If $w < y$ then $\mutilde(y, w) = \mu(w, y)$.
    \item $\mutilde(y, w) = 0$ otherwise.
  \end{enumerate}

  Using the polynomials $P_{y, w}$, Kazhdan-Lusztig defined a $W$-graph, where there is an edge between $y$ and $w$ whenever $y \neq w$ and $\mutilde(y, w) \neq 0$.
  To define the left and right preorders, they also need the left and right descent sets, or \ti s, of an element $w$ of $W$.

  \begin{definition}
    Let $\tau_L(w) = \{s \in S \mid l(sw) < l(w)\}$.  Let $\tau_R(w) = \{s \in S \mid l(ws) < l(w)\}$.
  \end{definition}

  \begin{definition}
    We say $x \leqL y$ if there is a sequence $w_1,\dots,w_n$ of elements of $W$
    with $w_1 = x$ and $w_n = y$  such that $\mutilde(w_i,w_{i + 1}) > 0$ and $\tau_L(w_i) \not\subset \tau_L(w_{i + 1})$ for $1 \le i \le n-1$.
    The corresponding equivalence relation is denoted $\equivL$.

    We say $x \leqR y$ if there is a sequence $w_1,\dots,w_n$ of elements of $W$
    with $w_1 = x$ and $w_n = y$  such that $\mutilde(w_i,w_{i + 1}) > 0$ and $\tau_R(w_i) \not\subset \tau_R(w_{i + 1})$ for $1 \le i \le n - 1$.
    The corresponding equivalence relation is denoted $\equivR$.
  \end{definition}

  Up through the end of \myrefS{sec:main}, we'll be working on theorems, etc., which have left and right versions.
  To avoid having to subscript (or superscript) everything with $L$ and $R$, we'll work on the left.
  That is, we'll write the left version of everything, without subscript, and leave it to the reader to formulate the right version.
  So, for example, we'll write $\tau$ for $\tau_L$.
  In later sections, when we have to work with both sides at once, we'll put the subscripts back in.

  There are a few propositions from \citeKL\ which we'll be using frequently, so we'll recall them here.

  \begin{proposition}[Equation (2.2.c) of \citeKL]
    \label{prop:KL22c}
    Let $y, w \in W$, $s \in S$, with $sw < w$.  Then
    \begin{equation}
      \label{eq:KL22c}
      P_{y,w} = q^{1-c}P_{sy,sw} + q^cP_{y,sw} - \sum_{\substack{z\\y \le z \prec sw\\sz < z}} \mu(z,sw)q^{\frac {l(w) - l(z)} 2}P_{y,z}
    \end{equation}
    with $c=1$ if $sy < y$, $c=0$ if $sy > y$.
  \end{proposition}

  \begin{remark}
    In proofs involving \autoref{eq:KL22c}, I'll refer the last part of the equation, the part with the summation sign, as the sum portion of the equation.
  \end{remark}

  \begin{remark}
    With our conventions, in the sum portion of  \autoref{eq:KL22c}, we can omit the requirement that $y \leq z$, since if $y \nleq z$ then $P_{y, w} = 0$.
    We can omit the requirement that $z \prec sw$, since if $z \nprec sw$ then $\mu(z, sw) = 0$.
    We can also replace the requirement $z \prec sw$ with $z \leq sw$.
    In what follows, we will use whichever version of the formula is most convenient.
  \end{remark}

  \begin{proposition}[(2.3.g) and (2.3.e) of \citeKL]
    \label{prop:KL23g}
    Suppose $x, x' \in W$, $s \in S$, and suppose $sx > x$, $sx' < x'$.
    Then $P_{x, x'} = P_{sx, x'}$.
    In particular, if $x' \neq sx$ then $\mu(x, x') = 0$.
  \end{proposition}

  \begin{proposition}[Proposition 2.4 of \citeKL]
    \label{prop:tauCell}
    Let $x, y \in W$.
    If $x \leqR y$ then $\tau_L(y) \subseteq \tau_L(x)$.
    If $x \equivR y$ then $\tau_L(x) = \tau_(y)$.
    Similarly, with left and right interchanged.
  \end{proposition}

  We'll also make use of this result from \cite{elias_williamson_2014}:
  \begin{theorem}[Corollary 1.2 of \cite{elias_williamson_2014}]
    \label{thm:nonNegative}
    For $y, w \in W$, the coefficients of $P_{y, w}$, in particular $\mu(y, w)$, are non-negative.
  \end{theorem}

  Next, we'll need some facts about the Bruhat order and about parabolic subgroups. See for example \cite{bjorner_brenti_2005}.
  \myrefP{prop:subword} and \myrefP{prop:parabolic} can be found there.

  \begin{proposition}
    \label{prop:subword}
    Let $y,w \in W$.  The following are equivalent:
    \begin{enumerate}
      \item $y \leq w$.
      \item Every reduced expression for $w$ has a subword which is a reduced expression for $y$.
      \item Some reduced expression for $w$ has a subword which is a reduced expression for $y$.
    \end{enumerate}
  \end{proposition}

  \begin{definition}
    Let $J \subseteq S$.
    \begin{enumerate}
      \item Let $W_J$ be the subgroup of $W$ generated by the set $J$.
      \item Let $W^J = \{w \in W \mid sw > w\ \text{for all}\ s \in J$\}.
    \end{enumerate}
  \end{definition}

  \begin{proposition}
    \label{prop:parabolic}
    Let $J \subseteq S$.  We have the following:
    \begin{enumerate}
      \item $(W_J,J)$ is a Coxeter system.
      \item For all $w \in W_J$, we have $l_J(w) = l(w)$, where $l_J(w)$ is the length of $w$ in the Coxeter system $(W_J,J)$.
      \item Every $w \in W$ has a unique factorization $w = w_Jw^J$ such that $w_J \in W_J$ and $w^J \in W^J$.
      \item For this factorization, $l(w) = l(w_J) + l(w^J) = l_J(w_J) + l(w^J)$.
      \item Each right coset $W_Jw$ has a unique representative of minimal length.  The system of such minimal coset representatives is $W^J$.
    \end{enumerate}
  \end{proposition}

  \begin{remark}
    By \myrefP{prop:parabolic}--2, for $w_J \in W_J$, its length is the same whether computed in $W_J$ or in $W$.
    So, we can use $l(w_J)$ to refer to this common value.
  \end{remark}

  \begin{definition}
    With $J$, $W^J$, and $w = w_Jw^J$ as in \myrefP{prop:parabolic}, define $p_J:W \longrightarrow W_J$ by $p_J(w) = w_J$.
    For $a \in W^J$, define $i_a^J:W_J \longrightarrow W$ by $i_a^J(w_J) = w_Ja$ for $w_J \in W_J$.
  \end{definition}

  We'll need this easy consequence of the above:

  \begin{proposition}
    \label{prop:parabolic2}
    Let $J \subseteq S$.  Let $w \in W$ and write $w = w_Jw^J$ with $w_J \in W_J$ and $w^J \in W^J$.
    \begin{enumerate}
      \item Let $y_J \in W_J$. We have $p_J(y_Jw) = y_Jw_J$ and $(y_Jw)^J = w^J$.
      \item For $s \in J$, we have $s \in \tau(w)$ if and only if $s \in \tau(p_J(w))$.
    \end{enumerate}
  \end{proposition}

  \begin{proof}
    Statement 1 is clear from statement 3 of  \myrefP{prop:parabolic}.
    From that and statements 2 and 4 of \myrefP{prop:parabolic}, we have $l(w) = l_J(w_J) + l(w^J)$ and $l(sw) = l_J(sw_J) + l(w^J)$.
    Statement 2 follows easily from that.
  \end{proof}

  We'll also need this later.

  \begin{proposition}
    \label{prop:tauInterval}
    Let $J \subseteq S$.
    Suppose $x \leqR w \leqR y$, and suppose $\tauL(x) \cap J = \tauL(y) \cap J$.
    Then $\tauL(w) \cap J = \tauL(x) \cap J$.
  \end{proposition}

  \begin{proof}
    This follows easily from \myrefP{prop:tauCell}.
  \end{proof}

  The last part of this section is a proposition relating parabolic subgroups and Kazhdan-Lusztig polynomials, \myrefP{prop:KLPolysParabolic}, which we'll need for what follows.
  We'll write $S(s,y,x) = \{z \in W \mid sz < z \text{ and } y \le z \le x\}$.

  \begin{proposition}
    \label{prop:intermediate}
    Let $J \subseteq S$.  Let $y, w \in W$ with $y \le w$, and
    suppose $y$ and $w$ are in the same right coset of $W_J$.
    Let $a^J$ be the minimal length representative of the coset,
    and write $w = w_Ja^J$  and
    $y = y_Ja^J$ with $y_J, w_J \in W_J$.
    Let $z \in W$ with $y \le z \le w$.
    Then we have $z \in W_Ja^J$, and, writing $z = z_Ja^J$ with $z_J \in W_J$, we have $y_J \le z_J \le w_J$.

    In particular (setting $z = y$, say) we have $y_J \le w_J$.
  \end{proposition}

  \begin{proof}
    Let $s_1\dots s_j$ with $s_i \in S$ be a reduced expression for $a^J$ and let $t_1\dots t_k$ with $t_i \in J$ be a reduced expression for $w_J$.
    Since $l(w) = l(w_J) + l(a^J)$, we have that $t_1\dots t_ks_1\dots s_j$ is a reduced expression for $w$.
    By \myrefP{prop:subword}, we can obtain a reduced expression for $z$ by removing some of the $s_i$ and $t_i$ elements from this reduced expression for $w$, and then we can obtain a reduced expression for $y$ from that reduced expression for $z$ by removing more of the $s_i$ and $t_i$ elements.
    If any of the $s_i$ elements are removed at either step, then the product of the remaining $s_i$ elements will form an element of shorter length than $a^J$ in the coset $W_Jy$, contradicting \myrefP{prop:parabolic}.
    So, only $t_i$ elements are removed at each stage, which gives the desired conclusion.
  \end{proof}

  \begin{corollary}
    \label{cor:iaJ}
    Let $J \subseteq S$.  Let $a \in W^J$.
    For $s \in J$ and $x,y \in W_J$, we have $i_a^J(S(s,y,x)) = S(s,i_a^J(y),i_a^J(x))$.
  \end{corollary}

  \begin{proof}
    By \myrefP{prop:parabolic2}, for $z \in W_J$, we have $sz < z$ if and only if $si_a^J(z) < i_a^J(z)$.
    It's clear that $y \le z \le x$ implies that $i_a^J(y) \le i_a^J(z) \le i_a^J(x)$.
    On the other hand, if $i_a^J(y) \le z' \le i_a^J(x)$ for some $z' \in W$, then \myrefP{prop:intermediate} says that $z' = i_a^J(z)$ for some $z \in W_J$ with $y \le z \le x$.
  \end{proof}

  \begin{remark}
    As a consequence of \myrefP{prop:KL22c}, if $y_J, w_J \in W_J$, then $P_{y_J,w_J}$ is the same whether computed in $W_J$ or in $W$.
    So, we can use $P_{y_J,w_J}$ to refer to this common polynomial.
  \end{remark}

  \begin{proposition}
    \label{prop:KLPolysParabolic}
    Let $J \in S$ and let $y_J, w_J \in W_J$.
    Let $a \in W^J$ and let $y = i_a^J(y_J)$ and $w = i_a^J(w_J)$.
    Then $P_{y,w} = P_{y_J, w_J}$, where the latter polynomial is taken with respect to the Coxeter system $(W_J, J)$.
  \end{proposition}

  \begin{proof}
    The proof is by induction on $l(w_J)$, using \myrefP{prop:KL22c} and \myrefP{prop:intermediate}.
    If $l(w_J) = 0$, then $l(y_J) = 0$ as well, so $y = w$ and the proposition holds.
    So assume $l(w_J) > 0$ and choose $s \in J$ with $l(sw_J) < l(w_J)$.
    By \myrefP{prop:parabolic2}, $sw < w$.
    Applying \myrefP{prop:KL22c}, first in $W$ and secondly in $W_J$, we obtain the following two equations:
    \begin{align*}
      P_{y,w} &= q^{1-c}P_{sy,sw} + q^cP_{y,sw} - \sum_{z \in S(s, y, sw)} \mu(z,sw)q^{\frac {l(w) - l(z)} 2}P_{y,z}\\
      P_{y_J,w_J} &= q^{1-c}P_{sy_J,sw_J} + q^cP_{y_J,sw_J} - \sum_{z' \in S(s,y_J,sw_J)} \mu(z',sw_J)q^{\frac {l(w_J) - l(z_J)} 2}P_{y_J,z'}
    \end{align*}
    with $c=1$ if $sy < y$, $c=0$ if $sy > y$.
    Again using \myrefP{prop:parabolic2}, we have $sy < y$ if and only if $sy_J < y_J$, so $c$ is the same in both equations.

    Now we need to match up the terms on the right-hand sides of the two equations and show that they are equal, by induction.
    The first two terms, that's clear.
    We need to show that the sum portions of the equations are the same.
    Since we've fixed $J$, we'll write $i_a$ for $i_a^J$.
    We have $sw = i_a(sw_J)$ (see \myrefP{prop:parabolic2}--1), so
    \begin{align*}
      &\sum_{z \in S(s, y, sw)} \mu(z,sw)q^{\frac {l(w) - l(z)} 2}P_{y,z}
      = \sum_{z \in i_a(S(s, y_J, sw_J))} \mu(z,sw)q^{\frac {l(w) - l(z)} 2}P_{y,z}\\
      &= \sum_{z' \in S(s, y_J, sw_J)} \mu(i_a(z'),i_a(sw_J))q^{\frac {l(i_a(w_J)) - l(i_a(z'))} 2}P_{i_a(y_J), i_a(z')}\\
      &=\sum_{z' \in S(s,y_J,sw_J)} \mu(z',sw_J)q^{\frac {l(w_J) - l(z_J)} 2}P_{y_J,z'}
    \end{align*}

    Here the first equality is \myrefC{cor:iaJ}, the second is just substitution, and the third is induction, together with the fact that $l(w_J) - l(z') = l(w_Ja) - l(z'a)$, by \myrefP{prop:parabolic}--4.
  \end{proof}

  \section{Knuth Maps and the \texorpdfstring{$A_2$}{A2} Edge Transport Theorem}
  \label{sec:KnuthMaps}

  In this section we'll recall the first edge transport theorem, which concerns parabolic subgroups of type $A_2$.
  We'll be emulating it and using it extensively.
  We'll also recall the maps associated with this edge transport theorem, the Knuth maps.

  First, we'll describe the overall pattern of the three edge transport theorems.
  They each concern a parabolic subgroup $(W_J, J)$, isomorphic to a Weyl group.
  More precisely, they concern the middle two-sided cell (call it $C$ for now) in the parabolic subgroup, where middle means that it is preserved by multiplication by the long element of $W_J$.
  We'll group the elements of this cell into types.
  We'll extend the grouping into types to the subset of $W$ of elements whose image under $p_J$ lie in $C$ by saying that such elements have the same type as their image.
  An edge transport theorem concerns edges connecting elements in $W$ of the same type.
  It transports edges connecting elements of one type to edges connecting elements of another type.

  Let's see this pattern in the first edge transport theorem.
  Here $W_J$ is of type $A_2$.
  Write $J = \{s, t\}$ (with $st$ of order 3).
  The middle cell has four elements, in two left cells.
  See \autoref{fig:A2}.

  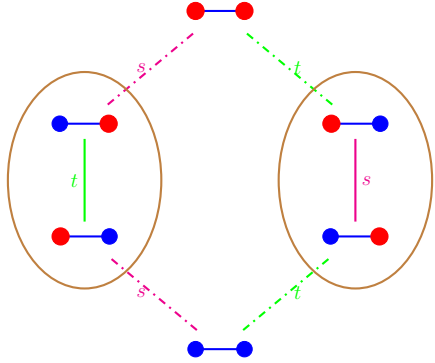
\begin{figure}[!ht]
    \begin{center}
      \begin{tikzpicture}[scale=0.6, every node/.style={scale=.65}]
        \node (ENode) at (0, 0) {\IdentityPictureATwo};
        \node at (-1 * \xSA, 1 * \ySA) {\SPictureATwo};
        \node at (1 * \xSA, 1 * \ySA) {\TPictureATwo};
        \node at (-1 * \xSA, 2 * \ySA) {\TPictureATwo};
        \node at (1 * \xSA, 2 * \ySA) {\SPictureATwo};
        \node at (0 * \xSA, 3 * \ySA) {\LongPictureATwo};

        \leftLineLabelA{magenta}{dashdotted}{0}{0}{s}{left}
        \rightLineLabelA{green}{dashdotted}{0}{0}{t}{right}
        \upLineLabelA{green}{solid}{-1}{1}{t}{left}
        \upLineLabelA{magenta}{solid}{1}{1}{s}{right}
        \rightLineLabelA{magenta}{dashdotted}{-1}{2}{s}{left}
        \leftLineLabelA{green}{dashdotted}{1}{2}{t}{right}

        \draw[brown, thick] (-1 * \xSA, 1.5 * \ySA) ellipse (1.7cm and 2.4cm);
        \draw[brown, thick] (1 * \xSA, 1.5 * \ySA) ellipse (1.7cm and 2.4cm);

      \end{tikzpicture}
    \end{center}
    \caption{Weyl group of type $A_2$}
    \label{fig:A2}
  \end{figure}

  The two left cells of interest are circled.
  They are $\{s, ts\}$ and $\{t, st\}$.
  In this illustration, each mini Dynkin diagram represents an element of $W_J$.
  Each Dynkin diagram is marked with the left \ti\ of the element which it represents.
  Elements of $J$ not in the \ti\ of the Coxeter group element are colored blue, whereas elements of $J$  in the \ti\ are colored red and are a little larger.
  In this parabolic subgroup, an element's type is determined by its \ti.
  So, there are two types.

  Now, let's look at the edge transport theorem.
  \begin{theorem}[\citeKL, Theorem 4.2]
    \label{thm:stStringsA}
    Let $s, t \in S$ with $st$ of order 3.
    Let $J = \{s, t\}$.
    \begin{enumerate}
      \item Let $L, U, L', U' \in W$ with $p_J(L) = p_J(L') = s$, $U = tL$, $U' = tL'$, and suppose $L \leq L'$.
      Then $\mu(U, U') = \mu(L, L')$.
      \item Let $L, U, L', U' \in W$ with $p_J(L) = s$, $p_J(L') = t$, $U = tL$, $U' = sL'$, and suppose $U \leq L'$.  Assume further that $tL \neq sL'$.
      Then $\mu(L, U') = \mu(U, L')$.
      \item Let $L, U, L', U' \in W$ with $p_J(L) = s$, $p_J(L') = t$, $U = tL$, and $U' = sL'$, and assume that $tL = sL'$.
      Then $\mu(L, U') = \mu(L', U) = 1$.
    \end{enumerate}
  \end{theorem}

  \myrefT{thm:stStringsA} is pictured in \autoref{fig:stStringsA}.
  ($L$ is for lower, $U$ is for upper;  later we'll have $M$ for middle.)
  The light blue dotted lines are the edges which are the subject of the theorem.
  Their arrowheads indicate the direction of the Bruhat order comparison.  ($W$-graph edges are undirected.)
  The theorem says that if one of the blue edges is present in the $W$-graph, then so is the other.

  \begin{figure}[!ht]
    \begin{center}
      \begin{tikzpicture}[scale=0.48, every node/.style={scale=.65}]
        \node at (-4.5, 0) {\begin{tikzpicture}[scale=0.6, every node/.style={scale=.9}]
          \node at (-1.03 * \xSA, 0 * \ySA) {\dynkinLabelLeftABTwo{\IdentityPictureATwo}{\phantom{L}}};
          \node (LNode) at (-1.03 * \xSA, 1 * \ySA)  {\dynkinLabelLeftABTwo{\SPictureATwo}{L}};
          \node (UNode) at (-1.03 * \xSA, 2 * \ySA) {\dynkinLabelLeftABTwo{\TPictureATwo}{U}};
          \node at (-1.03 * \xSA, 3 * \ySA) {\dynkinLabelLeftABTwo{\LongPictureATwo}{\phantom{U}}};

          \node at (1.03 * \xSA, 0 * \ySA) {\dynkinLabelABTwo{\IdentityPictureATwo}{\phantom{L'}}};
          \node (LPNode) at (1.03 * \xSA, 1 * \ySA) {\dynkinLabelABTwo{\SPictureATwo}{L'}};
          \node (UPNode) at (1.03 * \xSA, 2 * \ySA) {\dynkinLabelABTwo{\TPictureATwo}{U'}};
          \node at (1.03 * \xSA, 3 * \ySA) {\dynkinLabelABTwo{\LongPictureATwo}{\phantom{U'}}};

          \node at (0 * \xSA, -.5 * \xSA) {\Large Case 1};

          \upLineLabelA{magenta}{dashdotted}{-.9}{0}{s}{left}
          \upLineLabelA{green}{solid}{-.9}{1}{t}{left}
          \upLineLabelA{magenta}{dashdotted}{-.9}{2}{s}{left}

          \upLineLabelA{magenta}{dashdotted}{.9}{0}{s}{right}
          \upLineLabelA{green}{solid}{.9}{1}{t}{right}
          \upLineLabelA{magenta}{dashdotted}{.9}{2}{s}{right}

          \draw[->, very thick, cyan, dotted] (LNode.east) -- (LPNode.west);
          \draw[->, very thick, cyan, dotted] (UNode.east) -- (UPNode.west);
        \end{tikzpicture}};
        \node at (4.5, 0) {\begin{tikzpicture}[scale=0.6, every node/.style={scale=.9}]
          \node at (-1.03 * \xSA, 0 * \ySA) {\dynkinLabelLeftABTwo{\IdentityPictureATwo}{\phantom{L}}};
          \node (LNode) at (-1.03 * \xSA, 1 * \ySA)  {\dynkinLabelLeftABTwo{\SPictureATwo}{L}};
          \node (UNode) at (-1.03 * \xSA, 2 * \ySA) {\dynkinLabelLeftABTwo{\TPictureATwo}{U}};
          \node at (-1.03 * \xSA, 3 * \ySA) {\dynkinLabelLeftABTwo{\LongPictureATwo}{\phantom{U}}};

          \node at (1.03 * \xSA, 0 * \ySA) {\dynkinLabelABTwo{\IdentityPictureATwo}{\phantom{L'}}};
          \node (LPNode) at (1.03 * \xSA, 1 * \ySA) {\dynkinLabelABTwo{\TPictureATwo}{L'}};
          \node (UPNode) at (1.03 * \xSA, 2 * \ySA) {\dynkinLabelABTwo{\SPictureATwo}{U'}};
          \node at (1.03 * \xSA, 3 * \ySA) {\dynkinLabelABTwo{\LongPictureATwo}{\phantom{U'}}};

          \node at (0 * \xSA, -.5 * \xSA) {\Large Case 2};

          \upLineLabelA{magenta}{dashdotted}{-.9}{0}{s}{left}
          \upLineLabelA{green}{solid}{-.9}{1}{t}{left}
          \upLineLabelA{magenta}{dashdotted}{-.9}{2}{s}{left}

          \upLineLabelA{green}{dashdotted}{.9}{0}{t}{right}
          \upLineLabelA{magenta}{solid}{.9}{1}{s}{right}
          \upLineLabelA{green}{dashdotted}{.9}{2}{t}{right}

          \draw[->, very thick, cyan, dotted] (LNode.east) -- (UPNode.west);
          \draw[->, very thick, cyan, dotted] (UNode.east) -- (LPNode.west);
        \end{tikzpicture}};
        \node at (13.5,0)
        {\begin{tikzpicture}[scale=0.6, every node/.style={scale=.9}]
          \node at (-.13 * \xSA, 0 * \ySA) {\dynkinLabelLeftABTwo{\IdentityPictureATwo}{\phantom{L}}};
          \node (LNode) at (-1.03 * \xSA, 1 * \ySA)  {\dynkinLabelLeftABTwo{\SPictureATwo}{L}};
          \node (UNode) at (-1.03 * \xSA, 2 * \ySA) {\dynkinLabelLeftABTwo{\TPictureATwo}{U}};
          \node at (-.13 * \xSA, 3 * \ySA) {\dynkinLabelLeftABTwo{\LongPictureATwo}{\phantom{U}}};

          \node (LPNode) at (1.03 * \xSA, 1 * \ySA) {\dynkinLabelABTwo{\TPictureATwo}{L'}};
          \node (UPNode) at (1.03 * \xSA, 2 * \ySA) {\dynkinLabelABTwo{\SPictureATwo}{U'}};

          \node at (0 * \xSA, -.5 * \xSA) {\Large Case 3};

          \leftLineLabelA{magenta}{dashdotted}{-0}{0}{s}{left}
          \upLineLabelA{green}{solid}{-.9}{1}{t}{left}
          \rightLineLabelA{magenta}{dashdotted}{-.9}{2}{s}{left}

          \rightLineLabelA{green}{dashdotted}{0}{0}{t}{right}
          \upLineLabelA{magenta}{solid}{.9}{1}{s}{right}
          \leftLineLabelA{green}{dashdotted}{.9}{2}{t}{right}

          \draw[->, very thick, cyan, dotted] (LNode.east) -- (UPNode.west);
          \draw[<-, very thick, cyan, dotted] (UNode.east) -- (LPNode.west);
        \end{tikzpicture}};
      \end{tikzpicture}
    \end{center}
    \caption{\myrefT{thm:stStringsA}}
    \label{fig:stStringsA}
  \end{figure}

  Note, except for the Case 3 picture, the pictures in \autoref{fig:stStringsA} don't accurately compare the lengths of the elements on the left to those of the elements on the right.
  In the Case 1 picture, if the light blue lines represent edges (that is, $\mu$ is non-zero), then $l(L')$ must be at least one greater than $l(L)$.
  In the Case 2 picture, if the light blue lines represent edges, then $l(L')$ must be at least one greater than $l(U)$.

  For some applications, we only care about $\mutilde$ values, in which case we can condense the theorem into two cases, as follows.

  \begin{theorem}[\citeKL, Theorem 4.2]
    \label{thm:stStrings}
    Let $s, t \in S$ with $st$ of order 3.
    Let $J = \{s, t\}$.
    \begin{enumerate}
      \item Let $L, U, L', U' \in W$ with $p_J(L) = p_J(L') = s$, $U = tL$, and $U' = tL'$.
      Then $\mutilde(U, U') = \mutilde(L, L')$.
      \item Let $L, U, L', U' \in W$ with $p_J(L) = s$, $p_J(L') = t$, $U = tL$, and $U' = sL'$.
      Then $\mutilde(L, U') = \mutilde(U, L')$.
    \end{enumerate}
  \end{theorem}

  Now let's recall the definition of the Knuth maps.
  Each Knuth map is defined on a subset of $W$, as follows.

  \begin{definition}
    \label{def:domainKnuthA}
    Let $s,t \in S$ with $st$ of order 3.
    We set
    \begin{equation*}
      D_{s,t}(W) = \{w \in W \mid \tau(w)\cap \{s,t\} = t \}.
    \end{equation*}

    Equivalently, by \myrefP{prop:parabolic2}, we can write
    \begin{equation*}
      D_{s,t}(W) = \{w \in W \mid p_J(w) = t \text{ or } p_J(w) = ts \}.
    \end{equation*}
  \end{definition}
  (Note, this notation differs from that of \citeKL.
  Their $D_L(s,t)$ is our $D_{s,t}(W) \cup D_{t,s}(W)$.)
  \begin{definition}
    \label{def:KnuthA}
    Let $s,t \in S$ with $st$ of order 3.
    Let $J = \{s, t\}$.
    We define the Knuth map
    \begin{equation*}
      T_{s,t}: D_{s,t}(W) \longrightarrow D_{t,s}(W)
    \end{equation*}
    as follows:
    if $p_J(w) = t$ then $T_{s,t}(w) = sw$, else $T_{s,t}(w) = tw$.
  \end{definition}

  \begin{proposition}
    \label{prop:talbAltDef}
    We have
    \begin{equation*}
      T_{s,t}(w) = w'\ \text{where}\ \{w'\} = D_{t,s}(W) \cap \{sw, tw\}.
    \end{equation*}
  \end{proposition}

  \begin{proof}
    Write $w = p_J(w)a$ with $a \in W^J$.
    Suppose first $p_J(w) = t$, so $T_{s, t}(w) = sw$.
    We have $sw = sta$ and $tw = a$.
    By \myrefP{prop:parabolic2}, $sta \in D_{t,s}(W)$ and $a \notin D_{t,s}(W)$, as desired.

    Suppose instead that $p_J(w) = ts$, so $T_{s, t}(w) = tw$.
    We have $sw = stsa$ and $tw = sa$.
    By \myrefP{prop:parabolic2}, $sa \in D_{t,s}(W)$ and $stsa \notin D_{t,s}(W)$, as desired.
  \end{proof}

  \begin{remark}
    With $s, t$ as in \myrefD{def:KnuthA}, we have $T_{t,s} = T_{s, t}^{-1}$.
  \end{remark}

  We'll also use the following:

  \begin{proposition}
    \label{prop:talbParabolic}
    Let $s,t \in S$ with $st$ of order 3.
    Let $J \subseteq S$ with $s, t \in J$.
    Let $w \in W$.
    Then $w \in D_{s,t}(W)$ if and only if $p_J(w) \in D_{s,t}(W_J)$.
    If $w \in D_{s,t}(W)$ then $T_{s,t}(w) = T_{s,t}(p_J(w))w^J$, where $w^J \in W^J$ is such that $w = p_J(w)w^J$.
  \end{proposition}

  \begin{proof}
    This follows from \myrefP{prop:talbAltDef} and \myrefP{prop:parabolic2}.
    The first statement is clear from \myrefP{prop:parabolic2}.
    Similarly, since by \myrefP{prop:parabolic2}--2, we have $p_J(sw) = sp_J(w)$ and $p_J(tw) = tp_J(w)$, then $sw \in D_{t,s}(W)$ if and only if $sp_J(w) \in D_{t,s}(W_J)$, and similarly for $tw$.
  \end{proof}

  With the above notation, we can rephrase cases 1 and 2 of \myrefT{thm:stStringsA} as follows:
  \begin{theorem}
    \label{thm:talb}
    Let $s, t \in S$ with $st$ of order 3.
    Let $y, w \in D_{s,t}(W)$, and suppose that $yw^{-1}$ is not in the subgroup generated by $s$ and $t$.
    Then $\mu(y,w) = \mu(T_{s,t}(y), T_{s,t}(w))$.
  \end{theorem}

  Similarly, we can rephrase \myrefT{thm:stStrings} as follows:
  \begin{theorem}
    \label{thm:talbTilde}
    Let $s, t \in S$ with $st$ of order 3.
    Let $y, w \in D_{s,t}(W)$.
    Then $\mutilde(y,w) = \mutilde(T_{s,t}(y), T_{s,t}(w))$.
  \end{theorem}

  Suppose $w \in \Dst$, and let $y = \Tst(w)$.
  We know that $s \in \tau(y)$ and $t \notin \tau(y)$.
  Let's record here the other possible changes to $\tau(y)$.
  We'll need this:

  \begin{proposition}
    \label{prop:tauCommuting}
    Let $s, t \in S$ with $st = ts$.
    Let $w \in W$.
    We have $t \in \tau(sw)$ if and only if $t \in \tau(w)$.
  \end{proposition}

  \begin{proof}
    If we let $J = \{s, t\}$, then the statement is clearly true for $w \in W_J$.
    So then the proposition follows from \myrefP{prop:parabolic2}.
  \end{proof}

  \begin{proposition}
    \label{prop:tauNonCommuting}
    Let $s, t \in S$.
    Let $w \in W$.
    If $s, t \notin \tau(w)$ then $t \notin \tau(sw)$.
    If $s, t \in \tau(w)$ then $t \in \tau(sw)$.
  \end{proposition}

  \begin{proof}
    If we let $J = \{s, t\}$, then the statement is true for $w$ in the dihedral group $W_J$.
    Using \myrefP{prop:parabolic2}, we reduce to this case.
  \end{proof}

  Finally, we have this:

  \begin{proposition}
    \label{prop:tauTalb}
    Suppose $s, t \in S$ with $st$ of order 3, and suppose $w \in \Dst$.
    Let $r \in S \smallsetminus \{s, t\}$.
    \begin{enumerate}
      \item Suppose $r \in \tau(w)$ and $r \notin \tau(\Tst(w))$.
      Then $\Tst(w) = sw$ and $rs \neq sr$.
      If $rs$ is of order 3 then $T_{s,r}(w) = \Tst(w)$.
      \item Suppose $r \notin \tau(w)$ and $r \in \tau(\Tst(w))$.
      Then $\Tst(w) = tw$ and $rt \neq tr$.
      If $rt$ is of order 3 then $T_{r,t}(w) = \Tst(w)$.
    \end{enumerate}
  \end{proposition}

  \begin{proof}
    This follows from the previous two propositions, and \myrefP{prop:talbAltDef}.
  \end{proof}

  For convenience in finite Coxeter groups, we have the following.

  \begin{proposition}
    \label{prop:longElement}
    Let $(W, S)$ be a Coxeter system with $\abs{W}$ finite.
    Let $w_0 \in W$ be the long element.
    Let $s, t \in S$ with $st$ of order 3.
    Let $w \in W$.
    Then $w \in D_{s,t}(W)$ if and only if $ww_0 \in D_{t,s}(W)$.
    If $w \in D_{s,t}(W)$ then $T_{t,s}(ww_0) = T_{s,t}(w)w_0$.
  \end{proposition}

  \begin{proof}
    For $w \in W$ we have $l(ww_0) = l(w_0) - l(w)$ (see for example Proposition 2.3.2 of \cite{bjorner_brenti_2005}.)
    It follows easily that $\tau(ww_0) = S \smallsetminus \tau(w)$.
    The proposition follows easily from that.
  \end{proof}

  \section{\texorpdfstring{$B_2$}{B2} maps and the \texorpdfstring{$B_2$}{B2} Edge Transport Theorem}
  \label{sec:B2Maps}

  In this section we'll recall the second edge transport theorem, which concerns parabolic subgroups of type $B_2$.
  We'll also recall the definition of the maps associated to the theorem.
  This situation has many features in common with the $D_4$ edge transport theorem, which is the subject of this paper.

  Write $J = \{s, t\}$ with $st$ of order 4.
  The middle cell in $W_J$ has six elements, in two left cells.
  See \autoref{fig:B2}

  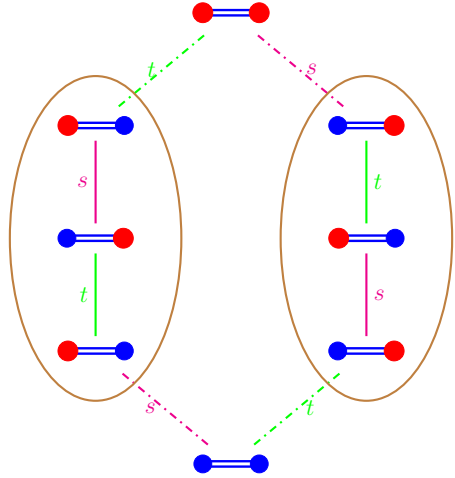
\begin{figure}[!ht]
    \begin{center}
      \begin{tikzpicture}[scale=0.6, every node/.style={scale=.75}]
        \node at (0, 0) {\IdentityPictureBTwo};
        \node at (-1 * \xSB, 1 * \ySB) {\SPictureBTwo};
        \node at (1 * \xSB, 1 * \ySB) {\TPictureBTwo};
        \node at (-1 * \xSB, 2 * \ySB) {\TPictureBTwo};
        \node at (1 * \xSB, 2 * \ySB) {\SPictureBTwo};
        \node at (-1 * \xSB, 3 * \ySB) {\SPictureBTwo};
        \node at (1 * \xSB, 3 * \ySB) {\TPictureBTwo};
        \node at (0 * \xSB, 4 * \ySB) {\LongPictureBTwo};

        \leftLineLabelA{magenta}{dashdotted}{0}{0}{s}{left}
        \rightLineLabelA{green}{dashdotted}{0}{0}{t}{right}
        \upLineLabelA{green}{solid}{-1}{1}{t}{left}
        \upLineLabelA{magenta}{solid}{1}{1}{s}{right}
        \upLineLabelA{magenta}{solid}{-1}{2}{s}{left}
        \upLineLabelA{green}{solid}{1}{2}{t}{right}
        \rightLineLabelA{green}{dashdotted}{-1}{3}{t}{left}
        \leftLineLabelA{magenta}{dashdotted}{1}{3}{s}{right}
        \draw[brown, thick] (-1 * \xSB, 2 * \ySB) ellipse (1.9cm and 3.6cm);
        \draw[brown, thick] (1 * \xSB, 2 * \ySB) ellipse (1.9cm and 3.6cm);
      \end{tikzpicture}
    \end{center}
    \caption{Weyl group of type $B_2$}
    \label{fig:B2}
  \end{figure}

  The two left cells of interest are circled.
  They are $\{s, ts, sts\}$ and $\{t, st, tst\}$.
  Lusztig (\cite{lusztig_1985}, Section 10.2) calls these left cells, and the corresponding left cells in $A_2$, strings.

  In this parabolic subgroup, again, an element's type is determined by its \ti.
  So, there are two types.

  Now, let's look at the edge transport theorem.  Note, lacking \cite{elias_williamson_2014}, this is not stated in full generality in \cite{lusztig_1985}.

  \begin{theorem}[\cite{lusztig_1985}, 10.4.2]
    \label{thm:stsStringsA}
     Let $s, t \in S$ with $st$ of order 4.
     Let $J$ be the parabolic subgroup generated by $\{s, t\}$.
    \begin{enumerate}
      \item Let $L, M, U, L', M', U' \in W$ with $p_J(L) = p_J(L') = s$, $M = tL$, $U = sM$, $M' = tL'$, $U' = sM'$.
      Then
      \begin{align*}
        &\mu(M, M') = \mu(L, L') + \mu(U, L')\\
        &\mu(U, U') = \mu(L, L')\\
        &\mu(U, L') = \mu(L, U')
      \end{align*}
      \item Let $L, M, U, L', M', U' \in W$ with $p_J(L) = s$, $p_J(L') = t$, $M = tL$, $U = sM$, $M' = sL'$, $U' = tM'$, and suppose $tL \neq sL'$.
      Then
      \begin{equation*}
        \mu(L, M') = \mu(U, M') = \mu(M, L')  = \mu(M, U')
      \end{equation*}
      \item Let $L, M, U, L', M', U' \in W$ with $p_J(L) = s$, $p_J(L') = t$, $M = tL$, $U = sM$, $M' = sL'$, $U' = tM'$, and assume that $tL = sL'$.
      Then
      \begin{equation*}
         \mu(L, M') = \mu(M', U) = \mu(L', M) = \mu(M, U') = 1
      \end{equation*}
    \end{enumerate}
  \end{theorem}

  \begin{figure}[!ht]
    \begin{center}
      \begin{tikzpicture}[scale=0.48, every node/.style={scale=.65}]
        \node at (-4.5, 0) {\begin{tikzpicture}[scale=0.6, every node/.style={scale=.9}]
          \node at (-1.03 * \xSB, 0 * \ySB) {\dynkinLabelLeftABTwo{\IdentityPictureBTwo}{\phantom{L}}};
          \node (LNode) at (-1.03 * \xSB, 1 * \ySB)  {\dynkinLabelLeftABTwo{\SPictureBTwo}{L}};
          \node (MNode) at (-1.03 * \xSB, 2 * \ySB) {\dynkinLabelLeftABTwo{\TPictureBTwo}{M}};
          \node (UNode) at (-1.03 * \xSB, 3 * \ySB) {\dynkinLabelLeftABTwo{\SPictureBTwo}{U}};
          \node at (-1.03 * \xSB, 4 * \ySB) {\dynkinLabelLeftABTwo{\LongPictureBTwo}{\phantom{U}}};

          \node at (1.03 * \xSB, 0 * \ySB) {\dynkinLabelABTwo{\IdentityPictureBTwo}{\phantom{L'}}};
          \node (LPNode) at (1.03 * \xSB, 1 * \ySB) {\dynkinLabelABTwo{\SPictureBTwo}{L'}};
          \node (MPNode) at (1.03 * \xSB, 2 * \ySB) {\dynkinLabelABTwo{\TPictureBTwo}{M'}};
          \node (UPNode) at (1.03 * \xSB, 3 * \ySB) {\dynkinLabelABTwo{\SPictureBTwo}{U'}};
          \node at (1.03 * \xSB, 4 * \ySB) {\dynkinLabelABTwo{\LongPictureBTwo}{\phantom{U'}}};

          \node at (0 * \xSB, -.5 * \xSB) {\Large Case 1};

          \upLineLabelA{magenta}{dashdotted}{-.9}{0}{s}{left}
          \upLineLabelA{green}{solid}{-.9}{1}{t}{left}
          \upLineLabelA{magenta}{solid}{-.9}{2}{s}{left}
          \upLineLabelA{green}{dashdotted}{-.9}{3}{t}{left}

          \upLineLabelA{magenta}{dashdotted}{.9}{0}{s}{right}
          \upLineLabelA{green}{solid}{.9}{1}{t}{right}
          \upLineLabelA{magenta}{solid}{.9}{2}{s}{right}
          \upLineLabelA{green}{dashdotted}{.9}{3}{t}{right}

          \draw[->, very thick, cyan, dotted] (LNode.east) -- (LPNode.west);
          \draw[->, very thick, orange, dotted] (MNode.east) -- (MPNode.west);
          \draw[->, very thick, cyan, dotted] (UNode.east) -- (UPNode.west);
          \draw[->, very thick, indigo, dotted, shorten >=0.15cm] (LNode.east) -- (UPNode.west);
          \draw[->, very thick, indigo, dotted, shorten >=0.15cm] (UNode.east) -- (LPNode.west);
        \end{tikzpicture}};
        \node at (4.5, 0) {\begin{tikzpicture}[scale=0.6, every node/.style={scale=.9}]
          \node at (-1.03 * \xSB, 0 * \ySB) {\dynkinLabelLeftABTwo{\IdentityPictureBTwo}{\phantom{L}}};
          \node (LNode) at (-1.03 * \xSB, 1 * \ySB)  {\dynkinLabelLeftABTwo{\SPictureBTwo}{L}};
          \node (MNode) at (-1.03 * \xSB, 2 * \ySB) {\dynkinLabelLeftABTwo{\TPictureBTwo}{M}};
          \node (UNode) at (-1.03 * \xSB, 3 * \ySB) {\dynkinLabelLeftABTwo{\SPictureBTwo}{U}};
          \node at (-1.03 * \xSB, 4 * \ySB) {\dynkinLabelLeftABTwo{\LongPictureBTwo}{\phantom{U}}};

          \node at (1.03 * \xSB, 0 * \ySB) {\dynkinLabelABTwo{\IdentityPictureBTwo}{\phantom{L'}}};
          \node (LPNode) at (1.03 * \xSB, 1 * \ySB) {\dynkinLabelABTwo{\TPictureBTwo}{L'}};
          \node (MPNode) at (1.03 * \xSB, 2 * \ySB) {\dynkinLabelABTwo{\SPictureBTwo}{M'}};
          \node (UPNode) at (1.03 * \xSB, 3 * \ySB) {\dynkinLabelABTwo{\TPictureBTwo}{U'}};
          \node at (1.03 * \xSB, 4 * \ySB) {\dynkinLabelABTwo{\LongPictureBTwo}{\phantom{U'}}};

          \node at (0 * \xSB, -.5 * \xSB) {\Large Case 2};

          \upLineLabelA{magenta}{dashdotted}{-.9}{0}{s}{left}
          \upLineLabelA{green}{solid}{-.9}{1}{t}{left}
          \upLineLabelA{magenta}{solid}{-.9}{2}{s}{left}
          \upLineLabelA{green}{dashdotted}{-.9}{3}{t}{left}

          \upLineLabelA{green}{dashdotted}{.9}{0}{t}{right}
          \upLineLabelA{magenta}{solid}{.9}{1}{s}{right}
          \upLineLabelA{green}{solid}{.9}{2}{t}{right}
          \upLineLabelA{magenta}{dashdotted}{.9}{3}{s}{right}

          \draw[->, very thick, cyan, dotted, shorten >=0.1cm] (LNode.east) -- (MPNode.west);
          \draw[->, very thick, cyan, dotted] (MNode.east) -- (LPNode.west);
          \draw[->, very thick, cyan, dotted] (MNode.east) -- (UPNode.west);
          \draw[->, very thick, cyan, dotted] (UNode.east) -- (MPNode.west);
        \end{tikzpicture}};
        \node at (13.5, 0) {\begin{tikzpicture}[scale=0.6, every node/.style={scale=.9}]
          \node at (-.13 * \xSB, 0 * \ySB) {\dynkinLabelLeftABTwo{\IdentityPictureBTwo}{\phantom{L}}};
          \node (LNode) at (-1.03 * \xSB, 1 * \ySB)  {\dynkinLabelLeftABTwo{\SPictureBTwo}{L}};
          \node (MNode) at (-1.03 * \xSB, 2 * \ySB) {\dynkinLabelLeftABTwo{\TPictureBTwo}{M}};
          \node (UNode) at (-1.03 * \xSB, 3 * \ySB) {\dynkinLabelLeftABTwo{\SPictureBTwo}{U}};
          \node at (-.13 * \xSB, 4 * \ySB) {\dynkinLabelLeftABTwo{\LongPictureBTwo}{\phantom{U}}};

          \node (LPNode) at (1.03 * \xSB, 1 * \ySB) {\dynkinLabelABTwo{\TPictureBTwo}{L'}};
          \node (MPNode) at (1.03 * \xSB, 2 * \ySB) {\dynkinLabelABTwo{\SPictureBTwo}{M'}};
          \node (UPNode) at (1.03 * \xSB, 3 * \ySB) {\dynkinLabelABTwo{\TPictureBTwo}{U'}};

          \node at (0 * \xSB, -.5 * \xSB) {\Large Case 3};

          \leftLineLabelA{magenta}{dashdotted}{0}{0}{s}{left}
          \upLineLabelA{green}{solid}{-.9}{1}{t}{left}
          \upLineLabelA{magenta}{solid}{-.9}{2}{s}{left}
          \rightLineLabelA{green}{dashdotted}{-.9}{3}{t}{left}

          \rightLineLabelA{green}{dashdotted}{0}{0}{t}{right}
          \upLineLabelA{magenta}{solid}{.9}{1}{s}{right}
          \upLineLabelA{green}{solid}{.9}{2}{t}{right}
          \leftLineLabelA{magenta}{dashdotted}{.9}{3}{s}{right}

          \draw[->, very thick, cyan, dotted, shorten >=0.1cm] (LNode.east) -- (MPNode.west);
          \draw[<-, very thick, cyan, dotted, shorten <=0.1cm] (MNode.east) -- (LPNode.west);
          \draw[->, very thick, cyan, dotted] (MNode.east) -- (UPNode.west);
          \draw[<-, very thick, cyan, dotted] (UNode.east) -- (MPNode.west);
        \end{tikzpicture}};
      \end{tikzpicture}
    \end{center}
    \caption{\myrefT{thm:stsStringsA}}
    \label{fig:stsStringsA}
  \end{figure}

  \myrefT{thm:stsStringsA} is pictured in \autoref{fig:stsStringsA}.
  The dotted lines in the center of each diagram are the edges which are the subject of the theorem.
  Their arrowheads indicate the direction of the Bruhat order comparison.  ($W$-graph edges are undirected.)
  In Cases 2 and 3, the theorem says that if one of the blue edges is present in the $W$-graph, then so are the others.

  Case 1 is more complicated.
  It says that if one of the blue edges is present in the $W$-graph, then so is the other, and similarly for the purple edges.
  If one of either the blue edges or the purple edges is present, then so is the orange edge.
  (We need \myrefT{thm:nonNegative} as well as \myrefT{thm:stStringsA} to conclude this.)
  If the orange edge is present, then either the blue edges or the purple edges are present, or possibly all are.

  Note again, except for the Case 3 picture, the pictures in \autoref{fig:stsStringsA} don't accurately compare the lengths of the elements on the left to those of the elements on the right.
  In the Case 1 picture, if the light blue lines represent edges (that is, $\mu$ is non-zero), then $l(L')$ must be at least one greater than $l(L)$.
  If the purple lines represent edges , then $l(L')$ must be at least one greater than $l(U)$.
  In the Case 2 picture, if the light blue lines represent edges, then $l(M')$ must be at least one greater than $l(U)$.

  \begin{remark}
    To prove statement 3 of \myrefT{thm:stStringsA}, we can start by looking within $W_J$.
    Note that there, the blue edges are just the edges which we are already familiar with, but on the right.
    That is, the blue edges connect $s$ with $st$, $st$ with $sts$, $t$ with $ts$, and $ts$ with $tst$.
    We can go from $W_J$ to the general case using \myrefP{prop:KLPolysParabolic}.

    This is different than the argument in the proof of Theorem 4.2 of \citeKL, which relies on their Lemma 2.6(iii).
    We will, however, use this line of argument in the proof of Case 3 of \myrefT{thm:mainA}, where we won't always have Lemma 2.6(iii) available.
  \end{remark}

  Again, for some applications, we only care about $\mutilde$ values.
  The $\mutilde$ version of \myrefT{thm:stsStringsA} is as follows.

  \begin{theorem}[\cite{lusztig_1985}, 10.4.2]
    \label{thm:stsStrings}
     Let $s, t \in S$ with $st$ of order 4.
     Let $J$ be the parabolic subgroup generated by $\{s, t\}$.
    \begin{enumerate}
      \item Let $L, M, U, L', M', U' \in W$ with $p_J(L) = p_J(L') = s$, $M = tL$, $U = sM$, $M' = tL'$, and $U' = sM'$.  Then
      \begin{align*}
        &\mutilde(M, M') = \mutilde(L, L') + \mutilde(U, L')\\
        &\mutilde(U, U') = \mutilde(L, L')\\
        &\mutilde(U, L') = \mutilde(L, U')
      \end{align*}
      \item Let $L, M, U, L', M', U' \in W$ with $p_J(L) = s$, $p_J(L') = t$, $M = tL$, $U = sM$, $M' = sL'$, and $U' = tM'$.  Then
      \begin{equation*}
        \mutilde(L, M') = \mutilde(U, M') = \mutilde(M, L')  = \mutilde(M, U')
      \end{equation*}
    \end{enumerate}
  \end{theorem}

  We'll note the proof that \myrefT{thm:stsStringsA} implies \myrefT{thm:stsStrings} here, since we'll need the same argument in \autoref{sec:main}.

  \begin{proof}[Proof of \myrefT{thm:stsStrings} given \myrefT{thm:stsStringsA}]
    Note that, in general, $\mutilde(x, y) = \mu(x, y)$ unless $y < x$.
    Given \myrefT{thm:stsStringsA}, we just need to show that all the $\mutilde$ terms in \myrefT{thm:stsStrings} are equal to their corresponding $\mu$ terms.
    We can assume that at least one of the $\mutilde$ terms in \myrefT{thm:stStrings} is non-zero.
    So then one of the $\mu$ terms must be non-zero. Without loss of generality, we can assume this term is of the form $\mu(x, y)$ with $x \in \{L, M, U\}$.
    Taking first the first statement of \myrefT{thm:stsStrings},
    since one of the $\mu$ terms is non-zero, then we must have $\mu(M, M') \neq 0$.
    Then $M < M'$, and so $L < M < M' < U'$, that is, $L < U'$, and thus $\mu(U', L) = 0$.
    Applying \myrefT{thm:stStringsA}, with the sides reversed, we conclude that $\mu(L', U) = 0$.
    Thus $\mutilde(L, U') = \mu(L, U')$ and $\mutilde(U, L') = \mu(U, L')$.
    The equality of the other $\mutilde$ terms to their corresponding $\mu$ terms is clear.

    For statement 2 the argument is easier.
    If one of the four $\mu$ terms in one direction is non-zero, then they all are, and thus are equal to their corresponding $\mutilde$ terms.
  \end{proof}

  It's a little more complicated to define the maps coming from this edge transport theorem than to define the Knuth (or $A_2$) maps, since here sometimes an element of one type corresponds to two elements of the other type.
  So for this situation, the image of an element in the domain of one of these maps will be a set of one or two elements.
  Otherwise, the definitions are like \myrefD{def:domainKnuthA} and \myrefD{def:KnuthA}.
  We'll call these maps $B_2$ maps, since they come from a parabolic subgroup of type $B_2$.

  \begin{definition}
    \label{def:domainKnuthB}
    Let $s,t \in S$ with $st$ of order 4.
    Let $J = \{s, t\}$.
    We set
    \begin{equation*}
      D_{s,t}(W) = \{w \in W \mid \tau(w)\cap \{s,t\} = t \}.
    \end{equation*}
    Equivalently, by \myrefP{prop:parabolic2}, we can write
    \begin{equation*}
      D_{s,t}(W) = \{w \in W \mid p_J(w) \in \{t, ts, tst\} \}.
    \end{equation*}
  \end{definition}

  \begin{definition}
    \label{def:KnuthB}
    Let $s,t \in S$ with $st$ of order 4.
    Let $J = \{s, t\}$.
    We define the map
    \begin{equation*}
      T_{s,t}: D_{s,t}(W) \longrightarrow \mathcal P \left({D_{t,s}(W)}\right)
    \end{equation*}
    as follows:
    \begin{enumerate}
      \item If $p_J(w) = t$ then $T_{s,t}(w) = \{sw\}$.
      \item If $p_J(w) = ts$ then $T_{s,t}(w) = \{sw, tw\}$.
      \item If $p_J(w) = tst$ then $T_{s,t}(w) = \{tw\}$.
    \end{enumerate}
  \end{definition}

  Again, we have the alternate characterization:

  \begin{proposition}
    \label{prop:KnuthBAltDef}
    We have
    \begin{equation*}
      T_{s,t}(w) = D_{t,s}(W) \cap \{sw, tw\}.
    \end{equation*}
  \end{proposition}

  \begin{remark}
    Our definition is the Coxeter group version of the map $T_{\alpha\beta}$ of \cite[Definition 3.4 b]{vogan_1979}.
    Lusztig in Section 10.6 of \cite{lusztig_1985} defines a $*$ map and a map $w \mapsto \tilde w$, but this is neither.
    We'll see the map $w \mapsto \tilde w$ later, in \myrefS{sec:otherMaps}.
  \end{remark}

  \section{\texorpdfstring{$D_4$}{D4} Cells}
  \label{sec:D4cells}

  In this section we describe the left cells in the middle two-sided cell of a Weyl group of type $D_4$.
  For this section and the next, we assume that we have elements $s_1, s_2, s_3, s_4 \in S$ as shown below.

  \begin{figure}[H]
    \centering
    \begin{tikzpicture}
      \draw[blue, very thick] (-.293, .707) -- (-1, 0);
      \draw[blue, very thick] (-.293, -.707) -- (-1, 0);
      \filldraw[blue] (-.293, .707) circle (4pt) node [below=.15cm, black]{$s_1$};
      \filldraw[blue] (-.293, -.707) circle (4pt) node [below=.15cm, black]{$s_2$};
      \draw[blue, very thick] (-2, 0) -- (-1, 0);
      \filldraw[blue] (-1, 0) circle (4pt) node [below=.15cm, black]{$s_3$};
      \filldraw[blue] (-2, 0) circle (4pt) node [below=.15cm, black]{$s_4$};
    \end{tikzpicture}
    \caption{$D_4$}
    \label{fig:D4}
  \end{figure}

  That is, $s_1s_3$, $s_2s_3$, and $s_4s_3$ are of order 3, whereas $s_1s_2$, $s_1s_4$, and $s_2s_4$ are of order 2.

  We'll write $\WDFour$ for the parabolic subgroup of $W$ generated by $\SDFour = \DFourSet$.
  For $w \in W$, we'll write $\tDFour(w) = \tau(w) \cap \SDFour$.
  We will write $T_{i,j}$ as a shorthand for $T_{s_i,s_j}$.

  Note that the (nonstandard) numberical labels are chosen for compatibility with later papers in this series, which study the Weyl group of type $D_n$ using domino tableaux.
  In this paper, there is no significance to the choice of which of the three outer elements are labeled $s_1$, $s_2$, and $s_4$.

  The two-sided cell in the middle of $\WDFour$ is a union of eight left cells.
  In this section, we'll look at those cells and describe how they are divided into types.
  We'll state and prove some of the results which we'll need for what follows.

  The first two left cells are shown in \autoref{fig:type10Full}.

  \begin{figure}[!ht]
    \begin{center}
      \begin{tikzpicture}[scale=0.7,  every node/.style={transform shape}]
        \node (CNode) at (0, 0) {\dynkinLabel{\TypeCPicture}{C}};
        \node (DNode) at (0, \ySLT) {\dynkinLabel{\TypeDPicture}{D}};
        \node (A2Node) at (-1 * \xSLT, 2 * \ySLT) {\dynkinLabel{\TypeAPicture{2}}{A_2}};
        \node (A4Node) at (0, 2 * \ySLT) {\dynkinLabel{\TypeAPicture{4}}{A_4}};
        \node (A1Node) at (1 * \xSLT, 2 * \ySLT) {\dynkinLabel{\TypeAPicture{1}}{A_1}};
        \node (B1Node) at (-1 * \xSLT, 3 * \ySLT) {\dynkinLabel{\TypeBPicture{1}}{B_1}};
        \node (B4Node) at (0, 3 * \ySLT) {\dynkinLabel{\TypeBPicture{4}}{B_4}};
        \node (B2Node) at (1 * \xSLT, 3 * \ySLT) {\dynkinLabel{\TypeBPicture{2}}{B_2}};
        \node (CPNode) at (0, 4 * \ySLT) {\dynkinLabel{\TypeCPicture}{C}};
        \node (DPNode) at (0, 5 * \ySLT) {\dynkinLabel{\TypeDPicture}{D}};

        \upLineT{darkblue}{solid}{0}{0}
        \upLineT{goldenpoppy}{dashed}{0}{1}
        \rightLineT{magenta}{dashed}{0}{1}
        \leftLineT{green}{dashed}{0}{1}
        \rightLineT{magenta}{solid}{-1}{2}
        \upLineT{goldenpoppy}{solid}{-1}{2}
        \leftLineT{green}{solid}{0}{2}
        \leftLineT{green}{solid}{1}{2}
        \upLineT{goldenpoppy}{solid}{1}{2}
        \rightLineT{magenta}{solid}{0}{2}
        \leftLineT{green}{dashed}{1}{3}
        \rightLineT{magenta}{dashed}{-1}{3}
        \upLineT{goldenpoppy}{dashed}{0}{3}
        \upLineT{darkblue}{solid}{0}{4}

        \draw[bend left = 60, gray, dashed] (DPNode) to (A1Node);
        \draw[bend right = 60, gray, dashed] (DPNode) to (A2Node);
        \draw[bend left = 30, gray, dashed] (DPNode) to (A4Node);
        \draw[bend left = 60, gray, dashed] (B2Node) to (CNode);
        \draw[bend right = 60, gray, dashed] (B1Node) to (CNode);
        \draw[bend right = 30, gray, dashed] (B4Node) to (CNode);

        \matrix [draw, below left, every node/.style={scale=.7}]
        at (current bounding box.north east)
        {
          \node [legendLine, draw=magenta,label=right:{$s_1$}] {}; \\
          \node [legendLine, draw=green,label=right:{$s_2$}] {}; \\
          \node [legendLine, draw=darkblue,label=right:{$s_3$}] {}; \\
          \node [legendLine, draw=goldenpoppy,label=right:{$s_4$}] {}; \\
        };
      \end{tikzpicture}
    \end{center}
    \caption{$C(10, a)$ or $C(10, b)$}
    \label{fig:type10Full}
  \end{figure}
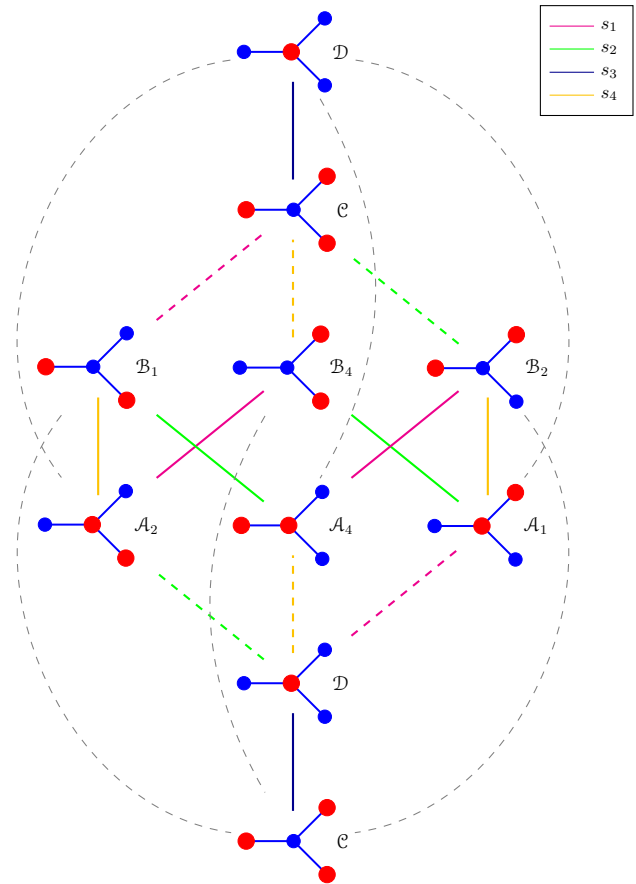

  In any such illustration, each mini Dynkin diagram represents an element of $\WDFour$ or of $W$.
  Elements of $\SDFour$ not in the \ti\ of the Coxeter group element are colored blue, whereas elements of $\SDFour$  in the \ti\ are colored red and are a little larger.
  For the lines in the diagrams which appear in this section, and in subsequent sections, we'll use the following conventions:
  \begin{itemize}
    \item Solid lines are Knuth maps.
    \item Dashed lines which are not gray are $D_4$ maps. (To be defined later.)
    \item Gray lines are connections where $\mu(y, w) = 1$, but which don't come from multiplication by an element of $\SDFour$ on the left.
    \item Dash dotted lines are lines which are neither Knuth moves nor $D_4$ maps.
    \item Dotted lines are lines where we don't have enough information to determine which of the above situations apply.
    \item With the exception of the gray lines, lines come from multiplying on the left by elements of $\SDFour$.
    Colors are the same in all diagrams.
  \end{itemize}

  The first two left cells of interest each have the structure shown in \autoref{fig:type10Full}.
  They are the cells whose bottom elements are $s_1s_2s_4$ and $s_1s_2s_4s_3$.  We'll call them $C(10, a)$ and $C(10, b)$, respectively.

  The next three left cells are illustrated by \autoref{fig:type14aFull}.
  The cell shown in that figure has bottom element $s_4s_3s_4$.  We'll call this cell $C(14, a, 4)$.
  The other two cells of this type have bottom element $s_1s_3s_1$ or $s_2s_3s_2$.  We'll call them $C(14, a, 1)$ and $C(14, a, 2)$, respectively.

  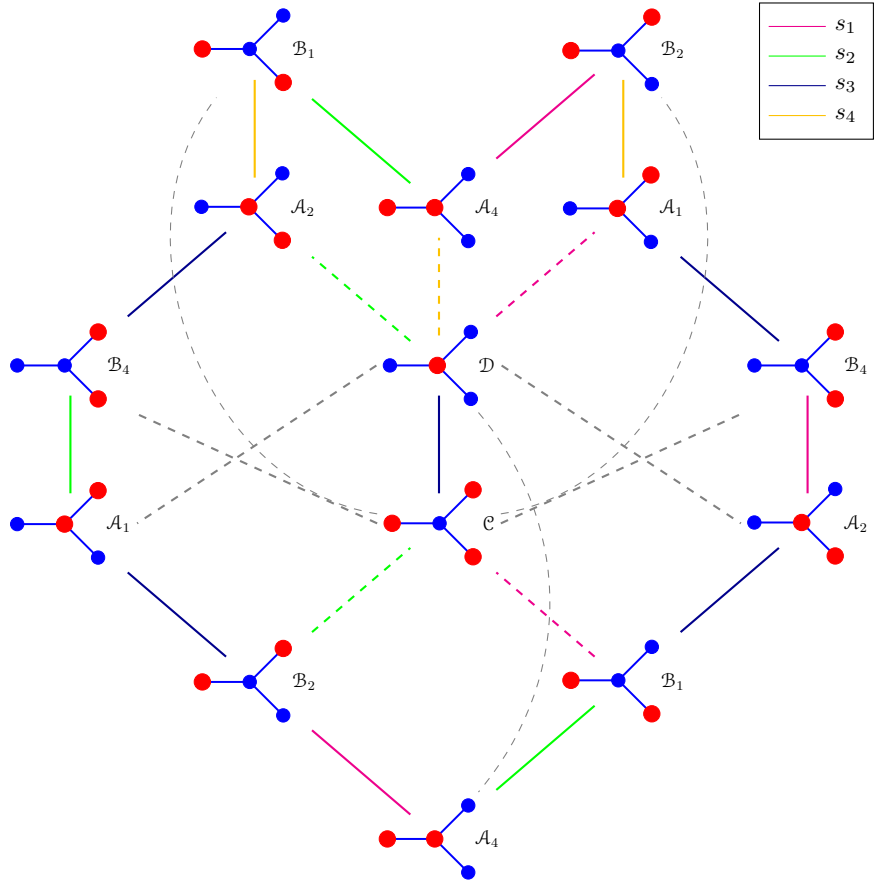
\begin{figure}[!ht]
    \begin{center}
      \begin{tikzpicture}[scale=0.7,  every node/.style={transform shape}]
        \node (A4Node) at (0,0) {\dynkinLabel{\TypeAPicture{4}}{A_4}};
        \node (B2Node) at (-1 * \xSLF, 1 * \ySLF) {\dynkinLabel{\TypeBPicture{2}}{B_2}};
        \node (B1Node) at (1 * \xSLF, 1 * \ySLF) {\dynkinLabel{\TypeBPicture{1}}{B_1}};
        \node (A1Node) at (-2 * \xSLF, 2 * \ySLF) {\dynkinLabel{\TypeAPicture{1}}{A_1}};
        \node (CNode) at (0, 2 * \ySLF) {\dynkinLabel{\TypeCPicture}{C}};
        \node (A2Node) at (2 * \xSLF, 2 * \ySLF) {\dynkinLabel{\TypeAPicture{2}}{A_2}};
        \node (B4Node) at (-2 * \xSLF, 3 * \ySLF) {\dynkinLabel{\TypeBPicture{4}}{B_4}};
        \node (DNode) at (0, 3 * \ySLF) {\dynkinLabel{\TypeDPicture}{D}};
        \node (B4PNode) at (2 * \xSLF, 3 * \ySLF) {\dynkinLabel{\TypeBPicture{4}}{B_4}};
        \node (A2PNode) at (-1 * \xSLF, 4 * \ySLF) {\dynkinLabel{\TypeAPicture{2}}{A_2}};
        \node (A4PNode) at (0, 4 * \ySLF) {\dynkinLabel{\TypeAPicture{4}}{A_4}};
        \node (A1PNode) at (1 * \xSLF, 4 * \ySLF) {\dynkinLabel{\TypeAPicture{1}}{A_1}};
        \node (B1PNode) at (-1 * \xSLF, 5 * \ySLF) {\dynkinLabel{\TypeBPicture{1}}{B_1}};
        \node (B2PNode) at (1 * \xSLF, 5 * \ySLF) {\dynkinLabel{\TypeBPicture{2}}{B_2}};

        \upLineF{darkblue}{solid}{0}{2}
        \upLineF{goldenpoppy}{dashed}{0}{3}
        \rightLineF{magenta}{dashed}{0}{3}
        \leftLineF{green}{dashed}{0}{3}
        \upLineF{goldenpoppy}{solid}{-1}{4}
        \leftLineF{green}{solid}{0}{4}
        \upLineF{goldenpoppy}{solid}{1}{4}
        \rightLineF{magenta}{solid}{0}{4}
        \rightLineF{darkblue}{solid}{-2}{3}
        \leftLineF{darkblue}{solid}{2}{3}
        \upLineF{green}{solid}{-2}{2}
        \upLineF{magenta}{solid}{2}{2}
        \leftLineF{darkblue}{solid}{-1}{1}
        \rightLineF{green}{dashed}{-1}{1}
        \leftLineF{magenta}{dashed}{1}{1}
        \rightLineF{darkblue}{solid}{1}{1}
        \leftLineF{magenta}{solid}{0}{0}
        \rightLineF{green}{solid}{0}{0}

        \draw[-, thick, gray, dashed] (CNode.west) -- (B4Node.south east);
        \draw[-, thick, gray, dashed] (DNode.west) -- (A1Node.east);
        \draw[-, thick, gray, dashed] (CNode.east) -- (B4PNode.south west);
        \draw[-, thick, gray, dashed] (DNode.east) -- (A2Node.west);

        \draw[bend left = 40, gray, dashed] (DNode) to (A4Node);
        \draw[bend right = 60, gray, dashed] (CNode) to (B2PNode);
        \draw[bend left = 60, gray, dashed] (CNode) to (B1PNode);

        \matrix [draw, below left, every node/.style={scale=.9}]
        at (current bounding box.north east)
        {
          \node [legendLine, draw=magenta,label=right:{$s_1$}] {}; \\
          \node [legendLine, draw=green,label=right:{$s_2$}] {}; \\
          \node [legendLine, draw=darkblue,label=right:{$s_3$}] {}; \\
          \node [legendLine, draw=goldenpoppy,label=right:{$s_4$}] {}; \\
        };
      \end{tikzpicture}
    \end{center}
    \caption{$C(14, a, 4)$}
    \label{fig:type14aFull}
  \end{figure}

  Finally, the last three cells of interest are illustrated by \autoref{fig:type14bFull}.
  The cell shown in that figure has bottom elements $s_1s_3s_1s_2$ and $s_2s_3s_2s_1$.
  We'll call this cell $C(14, b, 4)$.
  There is another cell of this type with bottom elements
  $s_2s_3s_2s_4$ and $s_4s_3s_4s_2$.
  We'll call this cell $C(14, b, 1)$.
  The last cell of this type has bottom elements $s_1s_3s_1s_4$ and $s_4s_3s_4s_1$.
  We'll call this cell $C(14, b, 2)$.

  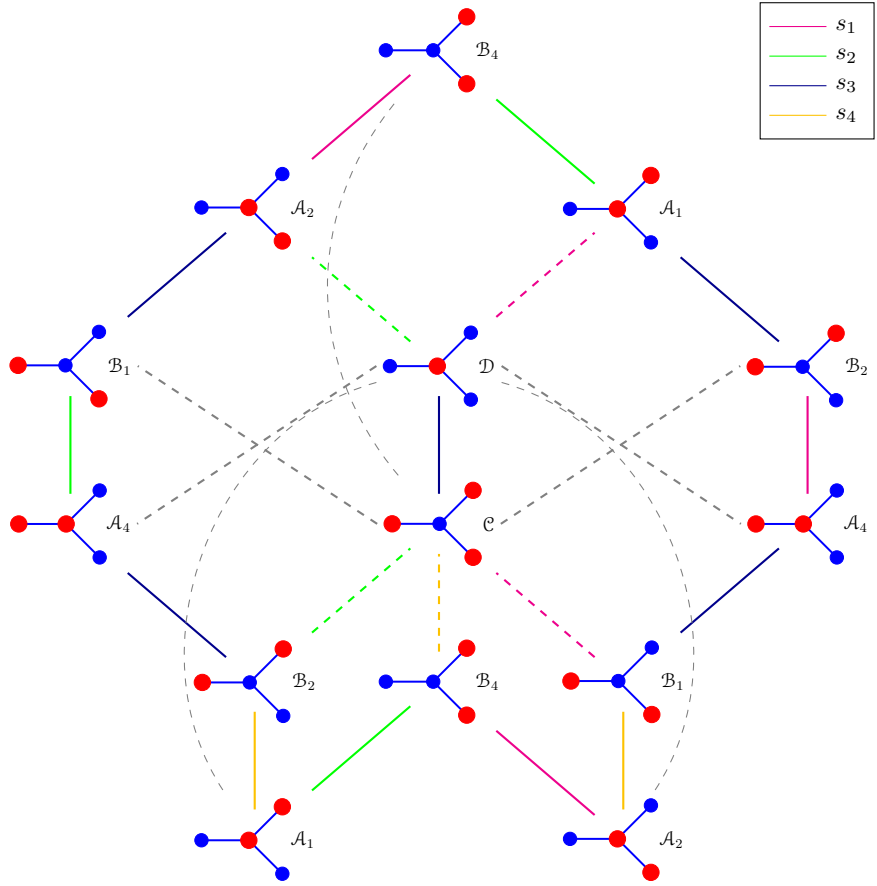
\begin{figure}[!ht]
    \begin{center}
      \begin{tikzpicture}[scale=0.7,  every node/.style={transform shape}]
        \node (A1Node) at (-1 * \xSLF, 0 * \ySLF)
        {\dynkinLabel{\TypeAPicture{1}}{A_1}};
        \node (A2Node) at (1 * \xSLF, 0 * \ySLF)
        {\dynkinLabel{\TypeAPicture{2}}{A_2}};
        \node (B2Node) at (-1 * \xSLF, 1 * \ySLF) {\dynkinLabel{\TypeBPicture{2}}{B_2}};
        \node (B4Node) at (0 * \xSLF, 1 * \ySLF) {\dynkinLabel{\TypeBPicture{4}}{B_4}};
        \node (B1Node) at (1 * \xSLF, 1 * \ySLF) {\dynkinLabel{\TypeBPicture{1}}{B_1}};
        \node (A4Node) at (-2 * \xSLF, 2 * \ySLF) {\dynkinLabel{\TypeAPicture{4}}{A_4}};
        \node (CNode) at (0 * \xSLF, 2 * \ySLF) {\dynkinLabel{\TypeCPicture}{C}};
        \node (A4PNode) at (2 * \xSLF, 2 * \ySLF) {\dynkinLabel{\TypeAPicture{4}}{A_4}};
        \node (B1PNode) at (-2 * \xSLF, 3 * \ySLF) {\dynkinLabel{\TypeBPicture{1}}{B_1}};
        \node (DNode)  at (0 * \xSLF, 3 * \ySLF) {\dynkinLabel{\TypeDPicture}{D}};
        \node (B2PNode)  at (2 * \xSLF, 3 * \ySLF) {\dynkinLabel{\TypeBPicture{2}}{B_2}};
        \node (A2PNode) at (-1 * \xSLF, 4 * \ySLF) {\dynkinLabel{\TypeAPicture{2}}{A_2}};
        \node (A1PNode) at (1 * \xSLF, 4 * \ySLF) {\dynkinLabel{\TypeAPicture{1}}{A_1}};
        \node (B4PNode) at (0 * \xSLF, 5 * \ySLF) {\dynkinLabel{\TypeBPicture{4}}{B_4}};

        \upLineF{darkblue}{solid}{0}{2}
        \rightLineF{magenta}{dashed}{0}{3}
        \leftLineF{green}{dashed}{0}{3}
        \rightLineF{magenta}{solid}{-1}{4}
        \leftLineF{green}{solid}{1}{4}
        \upLineF{goldenpoppy}{dashed}{0}{1}
        \rightLineF{green}{solid}{-1}{0}
        \leftLineF{magenta}{solid}{1}{0}
        \leftLineF{magenta}{dashed}{1}{1}
        \rightLineF{green}{dashed}{-1}{1}
        \upLineF{goldenpoppy}{solid}{-1}{0}
        \upLineF{goldenpoppy}{solid}{1}{0}
        \leftLineF{darkblue}{solid}{-1}{1}
        \upLineF{green}{solid}{-2}{2}
        \rightLineF{darkblue}{solid}{-2}{3}
        \rightLineF{darkblue}{solid}{1}{1}
        \upLineF{magenta}{solid}{2}{2}
        \leftLineF{darkblue}{solid}{2}{3}

        \draw[-, thick, gray, dashed] (CNode.west) -- (B1PNode.east);
        \draw[-, thick, gray, dashed] (DNode.west) -- (A4Node.east);
        \draw[-, thick, gray, dashed] (CNode.east) -- (B2PNode.west);
        \draw[-, thick, gray, dashed] (DNode.east) -- (A4PNode.west);

        \draw[bend left = 40, gray, dashed] (CNode) to (B4PNode);
        \draw[bend right = 54, gray, dashed] (DNode) to (A1Node);
        \draw[bend left = 54, gray, dashed] (DNode) to (A2Node);

        \matrix [draw, below left, every node/.style={scale=.9}]
        at (current bounding box.north east)
        {
          \node [legendLine, draw=magenta,label=right:{$s_1$}] {}; \\
          \node [legendLine, draw=green,label=right:{$s_2$}] {}; \\
          \node [legendLine, draw=darkblue,label=right:{$s_3$}] {}; \\
          \node [legendLine, draw=goldenpoppy,label=right:{$s_4$}] {}; \\
        };
      \end{tikzpicture}
    \end{center}
    \caption{$C(14, b, 4)$}
    \label{fig:type14bFull}
  \end{figure}

  \begin{proposition}
    \label{prop:WGraph}
    The diagrams in \autoref{fig:type10Full}, \autoref{fig:type14aFull}, and \autoref{fig:type14bFull} show all the edges of the $W$-graph connecting pairs of elements shown in the diagram.
  \end{proposition}

  \begin{proof}
    The colored lines are given by multiplication on the left by an element of $\SDFour$, so of course are edges.
    To see that the gray dashed lines and curves are edges, we could appeal to \myrefP{prop:KLPolysParabolic}, since these are known edges, where $\mu(y, w) = 1$, for Weyl groups of type $A_2$ (for the edges connecting elements whose lengths differ by 1) or type $A_3$ (for the edges connecting elements whose lengths differ by 3).
    Alternatively, we can appeal to \myrefT{thm:talbTilde}, since each gray edge can be obtained, using that theorem, from a colored edge.

    Finally, we need to know that there are no other edges.
    We first note that edges connect elements whose lengths differ by an odd number.
    After that, we can use \myrefP{prop:KL23g} to rule out the presence of any other edges.
    That is, \myrefP{prop:KL23g} says that if $\mu(x, y) \neq 0$ then either $y = sx$ for some $s \in S$ or $\tau(x) \subseteq \tau(y)$.
  \end{proof}

  \begin{proposition}
    \label{prop:D4LeftCells}
    The sets $C(10, a)$, $C(10, b)$, $C(14, a, i)$, and $C(14, b, i)$, for $i \in \refSet$, are left cells in $\WDFour$.
  \end{proposition}

  \begin{proof}
    We first need to show that the elements of each set are in the same cell.
    We'll start with one of the sets shown in \autoref{fig:type10Full}.
    Let's recall the description of the edges which generate the left preorder.
    An unordered edge connecting two elements $y, w \in W$ contributes to the left preorder if $\tau(y) \neq \tau(w)$.
    In this case we have a directed edge pointing away from the element which has something in its \ti\ which the other element lacks.
    (That is, the smaller element in the left preorder has the larger \ti, or at least, has something in its \ti\ which the other doesn't.)
    So, one unordered edge might contribute two ordered edges.

    \autoref{fig:CellPreorderD4} shows the ordered edges coming from \autoref{fig:type10Full}.
    Arrows point from the smaller element to the larger.
    By the way, this picture is typical of the general situation for left cells.
    An unordered edge connecting an element $w$ to $sw$ for $s \in S$ will contribute a downward-pointing edge and possibly one pointing up.
    Any other edge will point up.
    As noted at the end of the proof ot the previous proposition, this is from \myrefP{prop:KL23g}.

    We see from \autoref{fig:CellPreorderD4} that the ten elements pictured there are in the same left cell.  We can see similarly that the elements pictured in \autoref{fig:type14aFull} and \autoref{fig:type14bFull} are in the same cell.

    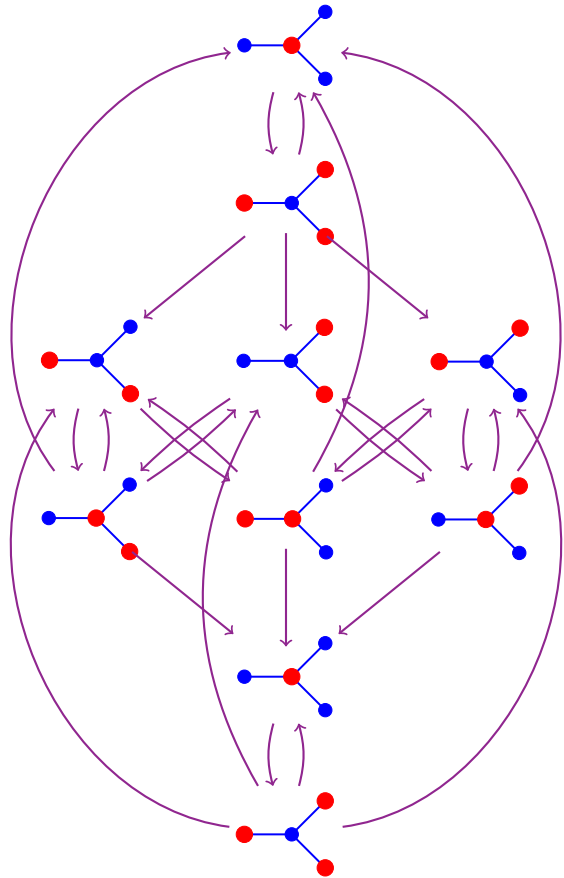
\begin{figure}[!ht]
      \begin{center}
        \begin{tikzpicture}[scale=0.7,  every node/.style={transform shape}]
          \node (CNode) at (0, 0) {\dynkinLabelWeyl{\TypeCPicture}{ }};
          \node (DNode) at (0, \ySLT) {\dynkinLabelWeyl{\TypeDPicture}{}};
          \node (A2Node) at (-1 * \xSLT, 2 * \ySLT) {\dynkinLabelWeyl{\TypeAPicture{2}}{}};
          \node (A4Node) at (0, 2 * \ySLT) {\dynkinLabelWeyl{\TypeAPicture{4}}{}};
          \node (A1Node) at (1 * \xSLT, 2 * \ySLT) {\dynkinLabelWeyl{\TypeAPicture{1}}{}};
          \node (B1Node) at (-1 * \xSLT, 3 * \ySLT) {\dynkinLabelWeyl{\TypeBPicture{1}}{}};
          \node (B4Node) at (0, 3 * \ySLT) {\dynkinLabelWeyl{\TypeBPicture{4}}{}};
          \node (B2Node) at (1 * \xSLT, 3 * \ySLT) {\dynkinLabelWeyl{\TypeBPicture{2}}{}};
          \node (CPNode) at (0, 4 * \ySLT) {\dynkinLabelWeyl{\TypeCPicture}{}};
          \node (DPNode) at (0, 5 * \ySLT) {\dynkinLabelWeyl{\TypeDPicture}{}};
          \lineArrowCellT{0}{4}{-1}{3}
          \lineArrowCellTMiddle{0}{4}{0}{3}
          \lineArrowCellT{0}{4}{1}{3}

          \lineArrowCellT{-1}{2}{0}{1}
          \lineArrowCellTMiddle{0}{2}{0}{1}
          \lineArrowCellT{1}{2}{0}{1}

          \bendCellSmallest{A2Node}{B4Node}{right}
          \bendCellSmallest{B4Node}{A2Node}{right}
          \bendCellSmallest{A4Node}{B1Node}{right}
          \bendCellSmallest{B1Node}{A4Node}{right}
          \bendCell{A2Node}{B1Node}{right}
          \bendCell{B1Node}{A2Node}{right}
          \bendCellSmallest{A4Node}{B2Node}{right}
          \bendCellSmallest{B2Node}{A4Node}{right}
          \bendCellSmallest{A1Node}{B4Node}{right}
          \bendCellSmallest{B4Node}{A1Node}{right}
          \bendCell{A1Node}{B2Node}{right}
          \bendCell{B2Node}{A1Node}{right}

          \bendCell{CPNode}{DPNode}{right}
          \bendCell{DPNode}{CPNode}{right}
          \bendCell{CNode}{DNode}{right}
          \bendCell{DNode}{CNode}{right}
          \bendCellBig{A1Node}{DPNode}{right}
          \bendCellBig{A2Node}{DPNode}{left}
          \bendCellSmall{A4Node}{DPNode}{right}

          \bendCellBig{CNode}{B2Node}{right}
          \bendCellBig{CNode}{B1Node}{left}
          \bendCellSmall{CNode}{B4Node}{left}

        \end{tikzpicture}
      \end{center}
      \caption{Left Preorder in a Cell in the Weyl Group of Type $D_4$}
      \label{fig:CellPreorderD4}
    \end{figure}

    We now need to see that there are no other elements in the same cell as any of those elements.
    We can show this using the first form of the generalized $\tau$-invariant, as found in \myrefS{sec:genTauA}.
    So, we'll defer the rest of this proof until then, since we can use it to illustrate the generalized $\tau$-invariant.
    Note, \myrefS{sec:genTauA} just relies on material from \ref{sec:KnuthMaps} and earlier.
  \end{proof}

  The elements in \autoref{fig:type10Full}, \autoref{fig:type14aFull}, and \autoref{fig:type14bFull} are labeled with their types, cf.\ \cite{garfinkle_vogan_1992}.
  We can define the types as follows:
  \begin{definition}
    We say $w \in W$ is of type $\scA_1$ if its projection $p_{\SDFour}(w)$ onto $\WDFour$ is an element labeled $\scA_1$ in one of these eight cells.
    Similarly for the other types.
  \end{definition}

  For the results of \myrefS{sec:D4Maps}, we'll need another characterization of these types.

  \begin{proposition}
    \label{prop:D4OpsAltDef}
    Let $w \in W$.  Let $i \in \refSet$.
    Choose $j, k$ such that $\{i, j, k\} = \refSet$.
    \begin{enumerate}
      \item We have that $w$ is of type $\scA_i$ if and only if $\tDFour(w) = \{s_i, s_3\}$ and
      $\tDFour(T_{j, 3}(w)) = \{s_i, s_j\}$.

      \item We have that $w$ is of type $\scB_i$ if and only if $\tDFour(w) = \{s_j, s_k\}$ and
      $\tDFour(T_{3, j}(w)) = \{s_k, s_3\}$.

      \item We have that $w$ is of type $\scC$ if and only if $\tDFour(w) = \bigRefSet$ and
      $\tDFour(T_{3, i}(w)) = \{s_3\}$.

      \item We have that $w$ is of type $\scD$ if and only if $\tDFour(w) = \{s_3\}$ and
      $\tDFour(T_{3, i}(w)) = \bigRefSet$.
    \end{enumerate}
  \end{proposition}

  \begin{proof}
    For $w \in \WDFour$, this is by inspection. One can check from \autoref{fig:type10Full}, \autoref{fig:type14aFull}, and \autoref{fig:type14bFull} that the elements in question satisfy these properties.
    Then, one has to check all the other elements of $\WDFour$, to see that none of them satisfy any of the listed conditions.

    To go from $\WDFour$ to $W$, we can use \myrefP{prop:parabolic2}--2 and \myrefP{prop:talbParabolic}.
  \end{proof}

  \begin{proposition}
    \label{prop:D4OpRightCells}
    Suppose
    $x, y, w \in W$ with $x$ and $y$ of right type $\scA_1$ and $x \leqL w \leqL y$.
    Then $w$ is of right type $\scA_1$.
    In particular, if $x \equivL y$ and $x$ is of right type $\scA_1$, then so is $y$.
    Similarly for the other types, and similarly interchanging left and right.
  \end{proposition}

  \begin{proof}
    We'll use the characterization of the types in \myrefP{prop:D4OpsAltDef}.
    There are two conditions for $x$ to be of type $\scA_1$.
    The first is that $\tau_0(x) = \{s_1, s_3\}$.
    The second is that $\tau_0(T_{j, 3}(x)) = \{s_1, s_j\}$ for $j \in \{2, 3\}$.
    Now suppose that $x, y \in D$ and $w \in W$ with $x \leqR w \leqR y$.
    That $\tau_0(w) = \{s_1, s_3\}$ follows from \myrefP{prop:tauInterval}.

    Now, by \myrefP{prop:KnuthAKLSet}, we have
    $T_{j, 3}(x) \leqR T_{j, 3}(w) \leqR T_{j, 3}(y)$.
    So, again by \myrefP{prop:tauInterval}, we have $\tau_0(T_{j, 3}(w)) = \tau_0(T_{j, 3}(x)) = \{s_1, s_j\}$.

    The proofs for the other types are similar.
  \end{proof}

  In what follows, we'll be working with elements of $W$ whose projection onto $\WDFour$ sits in one of the eight cells described in the previous section.

  \begin{definition}
    We'll write $X(10, a)$ for elements of $W$ whose projection onto $\WDFour$ sit inside $C(10, a)$, and similarly $X(10, b)$, $X(14, a, 1)$, etc.
    If $w^{\SDFour} \in W^{\SDFour}$, we'll call $C(10, a)w^{\SDFour}$ a $C(10, a)$ clump, or simply a clump, and similarly for the other $C$s.
  \end{definition}

  \begin{proposition}
    \label{prop:clumpDiagram}
    If $C$ is a clump contained in $X(10, a)$, then the relative lengths of its elements, their $\tau$-invariants, and the $T_{i, j}$ maps connecting them, are as shown in \autoref{fig:type10Full}.
    Similarly for the other types of clumps.
  \end{proposition}

  \begin{proof}
    This follows from \myrefP{prop:parabolic2} and \myrefP{prop:talbParabolic}.
  \end{proof}

  \begin{proposition}
    \label{prop:clumpSameCell}
    Let $C$ be a clump, and let $y, w \in C$.
    Then $y \equivL w$.
  \end{proposition}

  \begin{proof}
    We first need to know that the edges in \refDFourFigures\ are also edges in the clumps.
    This follows from \myrefP{prop:KLPolysParabolic}.
    Once we have that, we can argue as in the proof of \myrefP{prop:D4LeftCells}, above.
  \end{proof}

  \section{The \texorpdfstring{$D_4$}{D4} Edge Transport Theorem}
  \label{sec:main}

  In this section we prove the main theorem of the paper, the edge transport theorem coming from a parabolic subgroup of type $D_4$.

  For readers familiar with the proof of Theorem 4.2 of \citeKL, in broad outline this proof follows the same pattern.
  It starts with \autoref{eq:KL22c}.
  The differences are, first that it is more difficult to resolve the $P_{y, sw}$ term in \autoref{eq:KL22c}.
  This takes two steps, and results in two known terms, as well as (potentially) some unknown terms.
  In dealing with unknown terms, we rely on \myrefT{thm:nonNegative}, and accept inequalities in place of equalities.
  Secondly, we do not have the same ability to restrict a priori the terms coming from the sum portion of \autoref{eq:KL22c}.
  Again, after taking the terms which we need from it, we rely on \myrefT{thm:nonNegative} and obtain inequalities.
  Thirdly, we have many more cases.
  This is partly because the edge transport theorem involves elements of two essentially different types, and partly because there are eight left cells within the middle two-sided cell.
  Also, since we initially have inequalities, we need more inequalities so that we can solve them into equalities.
  In the end, after all this, we obtain the same equations as those in the $B_2$ edge transport theorem, \myrefT{thm:stsStrings}.

  As before we have two versions of the theorem.
  For the first theorem, case 3 is more complicated than in the previous edge transport theorems.
  We'll defer its more detailed statement until after we've proved the first two cases.

  \begin{theorem}
    \label{thm:mainA}
    Let $C$ and $C'$ be clumps.
    Fix $i \in \refSet$.
    We choose elements $L, M, U \in C$ as follows:
    if $\abs C = 10$ then $L$ and $U$ are the two elements of type $\scC$ in $C$, and $M$ is the one element of type $\scA_i$ in $C$.
    If instead $\abs C = 14$, then $L$ and $U$ are the two elements of type $\scA_i$ in $C$,
    and $M$ is the one element of type $\scC$ in $C$.
    We choose similarly $L', M', U' \in C'$.
    \begin{enumerate}
      \item Suppose $\abs C = \abs{C'}$, and suppose $C' \not\subset \WDFour C$.  Then
      \begin{align}
        \label{eqs:mainA1}
        \begin{split}
          &\mu(M, M') = \mu(L, L') + \mu(U, L')\\
          &\mu(U, U') = \mu(L, L')\\
          &\mu(U, L') = \mu(L, U')
        \end{split}
      \end{align}
      \item Suppose $\abs C \neq \abs{C'}$, and suppose $C' \not\subset \WDFour C$.  Then
      \begin{equation}
        \label{eqs:mainA2}
        \mu(L, M') = \mu(U, M') = \mu(M, L') = \mu(M, U')
      \end{equation}
      \item \myrefT{thm:main} Suppose $C' \subset \WDFour C$.  Then
      \begin{equation*}
        \mutilde(L, M') = \mutilde(U, M') = \mutilde(M, L') = \mutilde(M, U')
      \end{equation*}
      More precisely, \myrefP{prop:mainCase3} holds.
    \end{enumerate}
  \end{theorem}

  The version of this theorem using $\mutilde$ is as follows.

  \begin{theorem}
    \label{thm:main}
    Let $C$ and $C'$ be clumps.
    Fix $i \in \refSet$.
    We choose elements $L, M, U \in C$ as follows:
    if $\abs C = 10$ then $L$ and $U$ are the two elements of type $\scC$ in $C$, and $M$ is the one element of type $\scA_i$ in $C$.
    If instead $\abs C = 14$, then $L$ and $U$ are the two elements of type $\scA_i$ in $C$,
    and $M$ is the one element of type $\scC$ in $C$.
    We choose similarly $L', M', U' \in C'$.
    \begin{enumerate}
      \item Suppose $\abs C = \abs{C'}$.  Then
      \begin{align*}
        &\mutilde(M, M') = \mutilde(L, L') + \mutilde(U, L')\\
        &\mutilde(U, U') = \mutilde(L, L')\\
        &\mutilde(U, L') = \mutilde(L, U')
      \end{align*}
      \item Suppose $\abs C \neq \abs{C'}$.  Then
      \begin{equation*}
        \mutilde(L, M') = \mutilde(U, M') = \mutilde(M, L') = \mutilde(M, U')
      \end{equation*}
    \end{enumerate}
  \end{theorem}

  \begin{remark}
    The proof that \myrefT{thm:mainA} implies \myrefT{thm:main} is the same as the proof that  \myrefT{thm:stStringsA} implies \myrefT{thm:stsStrings}.  See \myrefS{sec:B2Maps}.
  \end{remark}

  \begin{remarkNumbered}
    \label{rem:chooseLU}
    It is enough to prove \myrefT{thm:mainA} for one choice of $L, U$
    and one choice of $L', U'$.
    To see this, first note that if we interchange $L$ and $U$ in \myref{Equations}{eqs:mainA1}, we get the same family of equations.
    If we interchange $L'$ and $U'$ in \myref{Equations}{eqs:mainA1}, we get an equivalent family of equations:
    \begin{align*}
      \begin{split}
        &\mu(M, M') = \mu(L, U') + \mu(U, U')\\
        &\mu(U, L') = \mu(L, U')\\
        &\mu(U, U') = \mu(L, L')
      \end{split}
    \end{align*}
    If we interchange $L$ and $U$ in \myref{Equations}{eqs:mainA2}, we get the same family of equations, and similarly for $L'$ and $U'$.
    In the cases which we study in detail, we will choose $L$ and $U$ with $l(L) < l(U)$.
  \end{remarkNumbered}

  \begin{remarkNumbered}
    \label{rem:relabel}
    As remarked before, there is no significance in this paper to the choice of which of the elements of $\SDFour$ are labeled $s_1$, $s_2$, and $s_4$.
    So, we will prove the theorems in some cases, and then deduce from those that it holds in the rest of the cases by renaming the elements of $\bigRefSet$.
    Specifically, we can and will do the following.
    Clumps of size 10 are symmetric in $s_1$, $s_2$, and $s_4$.
    So, to prove \myrefT{thm:mainA} when $\abs C = 10$ and $\abs{C'} = 14$, it suffices to prove it for one choice of $j$ with $C' \subset X(14, a, j)$ and one choice of $k$ with $C' \subset X(14, b, k)$.
    For convenience, we will choose $C' \subset X(14, a, 4)$ and $C' \subset X(14, b, 2)$.
    Similarly, when $\abs C = 14$ and $\abs{C'} = 10$, we will choose $C \subset X(14, a, 4)$ and $C \subset X(14, b, 2)$.
    Now suppose both $\abs C = 14$ and $\abs{C'} = 14$.  Again, we can choose that $C \subset X(14, a, 4)$ or $C \subset X(14, b, 2)$.
    Suppose that $C \subset X(14, a, 4)$.  Then, we see that it suffices to consider the cases of $C' \subset X(14, a, j)$ and $C' \subset X(14, b, j)$, with $j \neq 1$.
    We do not need to consider $C' \subset X(14, a, 1)$ since we can get to that case by interchanging the labels $s_1$ and $s_2$ when $C' \subset X(14, b, 2)$.
    Similarly, when $C \subset X(14, b, 2)$, it suffices to consider $C' \subset X(14, a, j)$ and $C' \subset X(14, b, j)$ with $j \neq 1$.
  \end{remarkNumbered}

  \begin{remark}
    Throughout this section, we will be using \myrefP{prop:clumpDiagram}.
  \end{remark}

  \begin{remark}
    I'd like to draw all the pictures analogous to those in \autoref{fig:stsStringsA}, but there are too many.
    Here is one, though, in \autoref{fig:mainACase1}.
    It shows one of the parts of Case 1 of \myrefT{thm:mainA}.
    As you can see, though there are different elements labeled $L, M, U$, etc., the blue, orange, and purple lines connecting them are in the same places, and have the same meaning, as in \autoref{fig:stsStringsA}.
  \end{remark}

  \begin{figure}[!ht]
    \centering
    \begin{tikzpicture}[scale=0.36,  every node/.style={transform shape}]
      \node at (0 * \xSLT, 0 * \ySLT) {
        \begin{tikzpicture}[opacity=0.8]
          \upLineF{darkblue}{solid}{0}{2}
          \upLineF{magenta}{dashed}{0}{3}
          \rightLineF{green}{dashed}{0}{3}
          \leftLineF{goldenpoppy}{dashed}{0}{3}
          \upLineF{magenta}{solid}{-1}{4}
          \leftLineF{goldenpoppy}{solid}{0}{4}
          \upLineF{magenta}{solid}{1}{4}
          \rightLineF{green}{solid}{0}{4}
          \rightLineF{darkblue}{solid}{-2}{3}
          \leftLineF{darkblue}{solid}{2}{3}
          \upLineF{goldenpoppy}{solid}{-2}{2}
          \upLineF{green}{solid}{2}{2}
          \leftLineF{darkblue}{solid}{-1}{1}
          \rightLineF{goldenpoppy}{dashed}{-1}{1}
          \leftLineF{green}{dashed}{1}{1}
          \rightLineF{darkblue}{solid}{1}{1}
          \leftLineF{green}{solid}{0}{0}
          \rightLineF{goldenpoppy}{solid}{0}{0}

          \node (A1Node) at (0,0) {\dynkinLabelWeylHuge{\TypeAPicture{1}}{L}};
          \node (B4Node) at (-1 * \xSLF, 1 * \ySLF) {\dynkinLabelWeylHuge{\TypeBPicture{4}}{ }};
          \node (B2Node) at (1 * \xSLF, 1 * \ySLF) {\dynkinLabelWeylHuge{\TypeBPicture{2}}{ }};
          \node (A2Node) at (-2 * \xSLF, 2 * \ySLF) {\dynkinLabelWeylHuge{\TypeAPicture{2}}{ }};
          \node (CNode) at (0, 2 * \ySLF) {\dynkinLabelWeylHuge{\TypeCPicture}{M}};
          \node (A4Node) at (2 * \xSLF, 2 * \ySLF) {\dynkinLabelWeylHuge{\TypeAPicture{4}}{ }};
          \node (B1Node) at (-2 * \xSLF, 3 * \ySLF) {\dynkinLabelWeylHuge{\TypeBPicture{1}}{ }};
          \node (DNode) at (0, 3 * \ySLF) {\dynkinLabelWeylHuge{\TypeDPicture}{ }};
          \node (B1PNode) at (2 * \xSLF, 3 * \ySLF) {\dynkinLabelWeylHuge{\TypeBPicture{1}}{ }};
          \node (A4PNode) at (-1 * \xSLF, 4 * \ySLF) {\dynkinLabelWeylHuge{\TypeAPicture{4}}{ }};
          \node (A1PNode) at (0, 4 * \ySLF) {\dynkinLabelWeylHuge{\TypeAPicture{1}}{U}};
          \node (A2PNode) at (1 * \xSLF, 4 * \ySLF) {\dynkinLabelWeylHuge{\TypeAPicture{2}}{ }};
          \node (B2PNode) at (-1 * \xSLF, 5 * \ySLF) {\dynkinLabelWeylHuge{\TypeBPicture{2}}{ }};
          \node (B4PNode) at (1 * \xSLF, 5 * \ySLF) {\dynkinLabelWeylHuge{\TypeBPicture{4}}{ }};

        \end{tikzpicture}
      };
      \node at (5 * \xSLT, 0 * \ySLT) {
        \begin{tikzpicture}[opacity=0.8]
          \upLineF{darkblue}{solid}{0}{2}
          \rightLineF{magenta}{dashed}{0}{3}
          \leftLineF{goldenpoppy}{dashed}{0}{3}
          \rightLineF{magenta}{solid}{-1}{4}
          \leftLineF{goldenpoppy}{solid}{1}{4}
          \upLineF{green}{dashed}{0}{1}
          \rightLineF{goldenpoppy}{solid}{-1}{0}
          \leftLineF{magenta}{solid}{1}{0}
          \leftLineF{magenta}{dashed}{1}{1}
          \rightLineF{goldenpoppy}{dashed}{-1}{1}
          \upLineF{green}{solid}{-1}{0}
          \upLineF{green}{solid}{1}{0}
          \leftLineF{darkblue}{solid}{-1}{1}
          \upLineF{goldenpoppy}{solid}{-2}{2}
          \rightLineF{darkblue}{solid}{-2}{3}
          \rightLineF{darkblue}{solid}{1}{1}
          \upLineF{magenta}{solid}{2}{2}
          \leftLineF{darkblue}{solid}{2}{3}

          \node (A1Node) at (-1 * \xSLF, 0 * \ySLF)
          {\dynkinLabelWeylHuge{\TypeAPicture{1}}{L'}};
          \node (A4Node) at (1 * \xSLF, 0 * \ySLF)
          {\dynkinLabelWeylHuge{\TypeAPicture{4}}{ }};
          \node (B4Node) at (-1 * \xSLF, 1 * \ySLF) {\dynkinLabelWeylHuge{\TypeBPicture{4}}{ }};
          \node (B2Node) at (0 * \xSLF, 1 * \ySLF) {\dynkinLabelWeylHuge{\TypeBPicture{2}}{ }};
          \node (B1Node) at (1 * \xSLF, 1 * \ySLF) {\dynkinLabelWeylHuge{\TypeBPicture{1}}{ }};
          \node (A2Node) at (-2 * \xSLF, 2 * \ySLF) {\dynkinLabelWeylHuge{\TypeAPicture{2}}{ }};
          \node (CNode) at (0 * \xSLF, 2 * \ySLF) {\dynkinLabelWeylHuge{\TypeCPicture}{M'}};
          \node (A2PNode) at (2 * \xSLF, 2 * \ySLF) {\dynkinLabelWeylHuge{\TypeAPicture{2}}{ }};
          \node (B1PNode) at (-2 * \xSLF, 3 * \ySLF) {\dynkinLabelWeylHuge{\TypeBPicture{1}}{ }};
          \node (DNode)  at (0 * \xSLF, 3 * \ySLF) {\dynkinLabelWeylHuge{\TypeDPicture}{ }};
          \node (B4PNode)  at (2 * \xSLF, 3 * \ySLF) {\dynkinLabelWeylHuge{\TypeBPicture{4}}{ }};
          \node (A4PNode) at (-1 * \xSLF, 4 * \ySLF) {\dynkinLabelWeylHuge{\TypeAPicture{4}}{ }};
          \node (A1PNode) at (1 * \xSLF, 4 * \ySLF) {\dynkinLabelWeylHuge{\TypeAPicture{1}}{U'}};
          \node (B2PNode) at (0 * \xSLF, 5 * \ySLF) {\dynkinLabelWeylHuge{\TypeBPicture{2}}{ }};

        \end{tikzpicture}
      };

      \coordinate (L) at (0 * \xSLT, -2.5 * \ySLT);
      \coordinate (LP) at (4 * \xSLT, -2.5 * \ySLT);
      \coordinate (M) at (0 * \xSLT, -.5 * \ySLT);
      \coordinate (MP) at (5 * \xSLT, -.5 * \ySLT);
      \coordinate (U) at (0 * \xSLT, 1.5 * \ySLT);
      \coordinate (UP) at (6 * \xSLT, 1.5 * \ySLT);

      \draw[->, bend left = 10, cyan, dotted, line width=0.45mm, shorten <=0.6cm, shorten >=0.7cm] (L) to (LP);

      \draw[->, bend left = 10, cyan, dotted, line width=0.45mm, shorten <=0.6cm, shorten >=0.7cm] (U) to (UP);

      \draw[->, bend left = 10, orange, dotted, line width=0.45mm, shorten <=0.6cm, shorten >=0.7cm] (M) to (MP);

      \draw[->, bend left = 8, indigo, dotted, line width=0.45mm, shorten <=0.5cm, shorten >=0.4cm] (L) to (UP);

      \draw[->, bend left = 10, indigo, dotted, line width=0.45mm, shorten <=0.5cm, shorten >=0.4cm] (U) to (LP);

    \end{tikzpicture}
    \caption{Part of Case 1 of \myrefT{thm:mainA}}
    \label{fig:mainACase1}
  \end{figure}

  We now move to the proof of \myrefT{thm:mainA}, Cases 1 and 2.
  We will be using \myrefP{prop:KL22c}, in situations where $sy < y$.
  Our biggest difficulty will be the resolution of the $qP_{y,sw}$ term in the expression for $P_{y,w}$ given by \myrefP{prop:KL22c}.
  Unlike the situation in the proof of Theorem 4.2 of \citeKL\ and 10.4 of \cite{lusztig_1985}, here multiplying $w$ by $s$ takes $s$ out of the \ti, but doesn't put anything in.
  So, we can't use \myrefP{prop:KL23g} at first.

  There will be three stages to the resolution of the $qP_{y,sw}$ term.
  First, a general proposition.
  The $x$ and $x'$ in this next proposition will later be $y$ and $sw$ in the main theorem.
  So, the purpose of this next proposition is to begin the resolution of the $qP_{y,sw}$ term by breaking it into two pieces, minus a residual.
  After that, we'll need to look at some different cases, to resolve the two terms which result from this following proposition.

  We'll use the notation $\sim$ from \citeKL, but with a little more data, as follows.
  \begin{definition}
    If $P$ and $P'$ are two polynomials, we say $P \dsim P'$ if $P$ and $P'$ are of degree at most $d$, and if $P - P'$ has degree less than $d$.
  \end{definition}

  \begin{proposition}
    \label{prop:lowerStar}
    Suppose $a, b \in S$.
    Suppose $x, x' \in W$ with $x < x'$.
    Suppose $ax < x$, $bx > x$, $ax' < x'$, and $bax' < ax'$.
    Suppose $l(x') - l(x)$ is even and suppose $bx \ne ax'$.
    Let $d = (l(x') - l(x) - 2)/2$
    Then
    \begin{equation*}
      P_{x,x'} \dsim P_{ax, ax'} + qP_{bx, ax'} - Cq^d
    \end{equation*}
    where $C$ is a non-negative integer.
  \end{proposition}

  \begin{figure}[!ht]
    \begin{center}
      \begin{tikzpicture}
        \node (r1y) at (0, 0) {$ax$};
        \node (y) at (0, 1.4) {$x$};
        \node (r2y) at (0, 2.8) {$bx$};

        \upLineLabel{blue}{r1y}{y}{a}
        \upLineLabel{brown}{y}{r2y}{b}

        \node (r2r1w) at (3, 0) {$bax'$};
        \node (r1w) at (3, 1.4) {$ax'$};
        \node (w) at (3, 2.8) {$x'$};

        \upLineLabel{brown}{r2r1w}{r1w}{b}
        \upLineLabel{blue}{r1w}{w}{a}
      \end{tikzpicture}
      \caption{\myrefP{prop:lowerStar}}
      \label{fig:lowerStar}
    \end{center}
  \end{figure}

  \begin{proof}
    Refer to \autoref{fig:lowerStar}.
    By \myrefP{prop:KL22c}, we have
    \begin{equation*}
      P_{x,x'} = P_{ax,ax'} + qP_{x,ax'} - \sum_{\substack{x\le z\prec ax'\\az < z}} \mu(z,ax')q^{(l(x') - l(z)) / 2}P_{x,z}.
    \end{equation*}

    We have $P_{x,ax'} = P_{bx, ax'}$ by \myrefP{prop:KL23g}.  So, we can put this information into the equation.
    Since by hypothesis $bx \ne ax'$, then $P_{x,ax'} = P_{bx, ax'}$ also implies that $\mu(x, ax') = 0$.
    So, we can remove $z = x$ from the sum portion of the equation.  Now we have

    \begin{equation*}
      P_{x,x'} = P_{ax,ax'} + qP_{bx,ax'} - \sum_{\substack{x < z\prec ax'\\az < z}} \mu(z,ax')q^{(l(x') - l(z)) / 2}P_{x,z}
    \end{equation*}

    Let
    \begin{equation*}
      Q_{x,x'} = \sum_{\substack{x < z\prec ax'\\az < z}} \mu(z,ax')q^{(l(x') - l(z)) / 2}P_{x,z}
    \end{equation*}
    We know that $Q_{x,x'}$ is a polynomial with non-negative coefficients by \myrefT{thm:nonNegative}.
    For each $z$ which contributes to the sum portion of the equation, we have $\mu(z,ax') \ne 0$, and thus $l(ax') - l(z)$ is odd, and thus $l(x') - l(z)$ is even.
    Since $l(x') - l(x)$ is even, we conclude that $l(x) - l(z)$ is also even, and thus (since $z \ne x$) that
    $P_{x,z}$ is of degree at most $(l(z) - l(x) - 2) / 2$.
    So, $q^{(l(x') - l(z)) / 2}P_{x,z}$ is of degree at most $(l(x') - l(x) - 2) / 2 = d$.
    So, the highest order term of $Q_{x,x'}$ is of the form $Cq^d$, where $C$ is a non-negative integer.
  \end{proof}

  In the next stage, we split into cases, and derive in each case an inequality involving $\mu$ terms.
  These inequalities will be our resolution of \autoref{eq:KL22c} as it applies to our situation.

  \begin{proposition}
    \label{prop:CTop}
    Let $y, w \in W$ with $l(w) - l(y)$ odd and $y \notin \WDFour w$.
    Suppose $y$ is of type $\scC$ and suppose $\tDFour(w) = \bigRefSet$, $\tDFour(s_1w) = \{s_2, s_4\}$, and $\tDFour(s_2s_1w) = \{s_3, s_4\}$.
    Then

    \begin{equation*}
      \begin{split}
        &\mu(y,w) + \sum_{\substack{z \in \WDFour y\\ s_1z < z}}\mu(y,z)\mu(z,s_1w) \leq\\ \mu(s_1y,s_1w) +  &\mu(s_3s_2y, s_2s_1w)
        + \mu(s_4s_3y, s_2s_1w)
         - \sum_{\substack{z \in \WDFour w\\ s_1z < z}} \mu(y,z)\mu(z,s_1w)
      \end{split}
    \end{equation*}
  \end{proposition}

  \begin{figure}[!ht]
    \begin{center}
      \begin{tikzpicture}[scale=0.7,  every node/.style={transform shape}]
        \node (Left) at (0,0) {
        \begin{tikzpicture}[scale=0.8,  every node/.style={transform shape}]
          \node (AFourNode)  {\dynkinLabelWeyl{\TypeAPicture{4}}{ }};
          \node (DNode) [below = \verticalToNode{AFourNode}] {\dynkinLabelWeyl{\TypeDPicture}{ }};
          \node (CNode) [below = \verticalToNode{DNode}] {\dynkinLabelWeyl{\TypeCPicture}{y}};
          \node (BTwoNode) [below = \verticalToNode{CNode}] {\dynkinLabelWeyl{\TypeBPicture{2}}{ }};

          \draw[-, ultra thick, goldenpoppy, dotted] (DNode.north) -- (AFourNode.south);
          \draw[-, very thick, green, dashed] (BTwoNode.north) -- (CNode.south);
          \upLine{darkblue}{solid}{CNode}{DNode}
        \end{tikzpicture}
        };
        \node (Right) at (8.5,-4) {
        \begin{tikzpicture}[scale=0.8,  every node/.style={transform shape}]
          \node (CNode) {\dynkinLabelWeyl{\TypeCPicture}{w}};
          \node (BOneNode) [below = \verticalToNode{CNode}] {\dynkinLabelWeyl{\TypeBPicture{1}}{ }};
          \node (AFourNode) [below = \verticalToNode{BOneNode}] {\dynkinLabelWeyl{\TypeAPicture{4}}{ }};

          \draw[-, ultra thick, magenta, dotted] (BOneNode.north) -- (CNode.south);
          \upLine{green}{solid}{AFourNode}{BOneNode}
        \end{tikzpicture}
        };

        \matrix [draw, below left, every node/.style={scale=.7}]
        at (current bounding box.north east)
        {
          \node [legendLine, draw=magenta,label=right:{$s_1$}] {}; \\
          \node [legendLine, draw=green,label=right:{$s_2$}] {}; \\
          \node [legendLine, draw=darkblue,label=right:{$s_3$}] {}; \\
          \node [legendLine, draw=goldenpoppy,label=right:{$s_4$}] {}; \\
        };
      \end{tikzpicture}
    \end{center}
    \caption{\myrefP{prop:CTop}}
    \label{fig:CTop}
  \end{figure}

  \begin{proof}
    Refer to \autoref{fig:CTop}.
    Let $d = d(y, w) = (l(w) - l(y) - 1) / 2$.
    In this proof, we'll be using \myrefP{prop:KL22c} with $s=s_1$, so $c = 1$.
    Let's first examine the term $qP_{y,s_1w}$ from that equation, using \myrefP{prop:lowerStar}.
    We will show that
    \begin{equation*}
      qP_{y, s_1w} \dsim (\mu(s_3s_2y, s_2s_1w) + \mu(s_4s_3y, s_2s_1w) - C)q^d
      \tag{*}
    \end{equation*}
    where $C$ is a non-negative integer.
    To see this, let $x = y$, $x' = s_1w$, $s = s_2$, and $t = s_3$.
    So, $(l(x') - l(x) - 2)/2 = (l(s_1w) - l(y) - 2)/2 = (l(w) - l(y) - 3)/2 = d(y, w) - 1$.
    Since $s_4 \notin \tau(tx)$ and $s_4 \in \tau(sx')$, we have $tx \ne sx'$.
    If we apply \myrefP{prop:lowerStar} (the $d$ of that proposition is then one less than the $d$ of this proposition)
    the result, after multiplying both sides by $q$, is
    \begin{equation*}
      qP_{y,s_1w} \dsim qP_{s_2y, s_2s_1w} + q^2P_{s_3y, s_2s_1w} - Cq^d
      \tag{**}
    \end{equation*}
    for some non-negative integer $C$.
    Since $s_3 \in \tau(s_2s_1w)$ and $s_3 \notin \tau(s_2y)$,
    we can apply \myrefP{prop:KL23g} to obtain $P_{s_2y, s_2s_1w} = P_{s_3s_2y, s_2s_1w}$.
    Since $s_4 \in \tau(s_2s_1w)$ and $s_4 \notin \tau(s_3y)$,
    we can apply \myrefP{prop:KL23g} to obtain $P_{s_3y, s_2s_1w} = P_{s_4s_3y, s_2s_1w}$.
    Now, $l(s_3s_2y) = l(y)$ and $l(s_2s_1w) - l(w) - 2$, so $d(s_3s_2y, s_2s_1w) = d(y, w) - 1$.
    Thus
    \begin{equation*}
      qP_{s_2y, s_2s_1w} = qP_{s_3s_2y, s_2s_1w} \dsim \mu(s_3s_2y, s_2s_1w)q^d.
    \end{equation*}
    Also, $l(s_4s_3y) = l(y) - 1$, so $d(s_4s_3y, s_2s_1w) = d(y, w) - 2$.
    Thus
    \begin{equation*}
      q^2P_{s_3y, s_2s_1w} = q^2P_{s_4s_3y, s_2s_1w} \dsim \mu(s_4s_3y, s_2s_1w)q^d.
    \end{equation*}
    If we put these last two formulas into (**), we get (*), as desired.

    So, now, with this preparation in hand, let's use \myrefP{prop:KL22c}.  From that, we have
    \begin{equation*}
      P_{y,w} = P_{s_1y,s_1w} + qP_{y,s_1w} - \sum_{\substack{y\le z\prec s_1w\\s_1z < z}} \mu(z,s_1w)q^{d(z,w)/2}P_{y,z}
    \end{equation*}

    From this we obtain directly
    \begin{equation*}
      \mu(y,w)q^d \dsim \mu(s_1y,s_1w)q^d + qP_{y,s_1w} - \sum_{s_1z < z} \mu(y,z)\mu(z,s_1w)q^d.
    \end{equation*}

    Finally, we substitute in (*), to obtain
    \begin{equation*}
      \begin{split}
        \mu(y,w)q^d &\dsim \mu(s_1y,s_1w)q^d + \mu(s_3s_2y, s_2s_1w)q^d + \mu(s_4s_3y, s_2s_1w)q^d -Cq^d\\
        &- \sum_{s_1z < z} \mu(y,z)\mu(z,s_1w)q^d.
      \end{split}
    \end{equation*}

    This yields
    \begin{equation*}
      \begin{split}
        \mu(y,w) &\leq \mu(s_1y,s_1w) + \mu(s_3s_2y, s_2s_1w) + \mu(s_4s_3y, s_2s_1w)\\
        &- \sum_{s_1z < z} \mu(y,z)\mu(z,s_1w)
      \end{split}
      \tag{***}
    \end{equation*}

    Now, let's work with the sum portion of the inequality.
    Since by hypothesis, $y \notin \WDFour w$, we have
    \begin{equation*}
      \sum_{\substack{z \in \WDFour y\\ s_1z < z}}\mu(y,z)\mu(z,s_1w) + \sum_{\substack{z \in \WDFour w\\ s_1z < z}} \mu(y,z)\mu(z,s_1w) \leq  \sum_{s_1z < z} \mu(y,z)\mu(z,s_1w)
      \tag{****}
    \end{equation*}

    Using \myrefT{thm:nonNegative}, we can substitute the left-hand side of (****) for the right-hand side of (****) in (***).
    This yields the inequality of the proposition.
  \end{proof}

  \begin{proposition}
    \label{prop:CTopAlt}
    Let $y, w \in W$ with $l(w) - l(y)$ odd and $y \notin \WDFour w$.
    Suppose $y$ is of type $\scC$ and suppose $\tDFour(w) = \bigRefSet$, $\tDFour(s_1w) = \{s_2, s_4\}$, and $\tDFour(s_4s_1w) = \{s_3, s_2\}$.
    Then
    \begin{equation*}
      \begin{split}
        &\mu(y,w) + \sum_{\substack{z \in \WDFour y\\ s_1z < z}}\mu(y,z)\mu(z,s_1w) \leq\\ \mu(s_1y,s_1w) + &\mu(s_3s_4y, s_4s_1w)
        + \mu(s_2s_3y, s_4s_1w)
         - \sum_{\substack{z \in \WDFour w\\ s_1z < z}} \mu(y,z)\mu(z,s_1w)
      \end{split}
    \end{equation*}
  \end{proposition}

  \begin{proof}
    This is just \myrefP{prop:CTop}, with the roles of $s_2$ and $s_4$ interchanged.
  \end{proof}

  \begin{proposition}
    \label{prop:ATopAlt}
    Let $y, w \in W$ with $l(w) - l(y)$ odd and $y \notin \WDFour w$.
    Suppose $y$ is of type $\scA_1$, with $s_4y$ of type $\scB_2$ and $\tDFour(s_2s_4y) = \bigRefSet$ and suppose $\tDFour(w) = \{s_1, s_3\}$ and $s_1(w)$ is of type $\scD$.
    Then
    \begin{equation*}
      \begin{split}
        &\mu(y,w) + \sum_{\substack{z \in \WDFour y\\ s_1z < z}}\mu(y,z)\mu(z,s_1w) \le \\ \mu(s_1y,s_1w) +  &\mu(s_4s_3y, s_3s_1w) + \mu(s_2s_4y, s_3s_1w)
         - \sum_{\substack{z \in \WDFour w\\ s_1z < z}} \mu(y,z)\mu(z,s_1w)
      \end{split}
    \end{equation*}
  \end{proposition}

  \begin{figure}[!ht]
    \begin{center}
      \begin{tikzpicture}[scale=0.7,  every node/.style={transform shape}]
        \node (Left) at (0,0) {
        \begin{tikzpicture}[scale=0.8,  every node/.style={transform shape}]
          \node (CNode) {\dynkinLabelWeyl{\TypeCPicture}{ }};
          \node (BTwoNode) [below = \verticalToNode{CNode}] {\dynkinLabelWeyl{\TypeBPicture{2}}{ }};
          \node (AOneNode) [below = \verticalToNode{BTwoNode}] {\dynkinLabelWeyl{\TypeAPicture{1}}{y}};
          \node (DNode) [below = \verticalToNode{AOneNode}] {\dynkinLabelWeyl{\TypeDPicture}{ }};

          \draw[-, ultra thick, green, dotted] (BTwoNode.north) -- (CNode.south);
          \upLine{goldenpoppy}{solid}{AOneNode}{BTwoNode}
          \draw[-, very thick, magenta, dashed] (DNode.north) -- (AOneNode.south);
        \end{tikzpicture}
        };
        \node (Right) at (8.5,-4) {
        \begin{tikzpicture}[scale=0.8,  every node/.style={transform shape}]
          \node (AOneNode)  {\dynkinLabelWeyl{\TypeAPicture{1}}{w}};
          \node (DNode) [below = \verticalToNode{AOneNode}] {\dynkinLabelWeyl{\TypeDPicture}{ }};
          \node (CNode) [below = \verticalToNode{DNode}] {\dynkinLabelWeyl{\TypeCPicture}{ }};

          \upLine{darkblue}{solid}{CNode}{DNode}
          \draw[-, very thick, magenta, dashed] (DNode.north) -- (AOneNode.south);
        \end{tikzpicture}
        };

        \matrix [draw, below left, every node/.style={scale=.7}]
        at (current bounding box.north east)
        {
          \node [legendLine, draw=magenta,label=right:{$s_1$}] {}; \\
          \node [legendLine, draw=green,label=right:{$s_2$}] {}; \\
          \node [legendLine, draw=darkblue,label=right:{$s_3$}] {}; \\
          \node [legendLine, draw=goldenpoppy,label=right:{$s_4$}] {}; \\
        };
      \end{tikzpicture}
    \end{center}
    \caption{\myrefP{prop:ATopAlt}}
    \label{fig:ATopAlt}
  \end{figure}

  \begin{proof}
    Refer to \autoref{fig:ATopAlt}.
    Just as in the proof of \myrefP{prop:CTop}, we need to evaluate the term $qP_{y,s_1w}$.
    Once that's done, the rest of the proof of this proposition will be the same as the proof of \myrefP{prop:CTop}.

    Let $d = d(y, w) = (l(w) - l(y) - 1) / 2$.
    We will show that
    \begin{equation*}
        qP_{y, s_1w} \dsim (\mu(s_4s_3y, s_3s_1w) + \mu(s_2s_4y, s_3s_1w) - C)q^d
      \tag{*}
    \end{equation*}
    where $C$ is a non-negative integer.
    To see this, let $x = y$, $x' = s_1w$, $s = s_3$, and $t = s_4$.
    As before, $(l(x') - l(x) - 2)/2 = d(y, w) - 1$.
    Since $s_2 \notin \tau(tx)$ and $s_2 \in \tau(sx')$, we have $tx \ne sx'$.
    If we apply \myrefP{prop:lowerStar} (the $d$ of that proposition is then one less than the $d$ of this proposition)
    the result, after multiplying both sides by $q$, is
    \begin{equation*}
      qP_{y,s_1w} \dsim qP_{s_3y, s_3s_1w} + q^2P_{s_4y, s_3s_1w} - Cq^d
      \tag{**}
    \end{equation*}
    for some non-negative integer $C$.
    Since $s_3 \notin \tau(s_4y)$, we know that $s_4 \notin \tau(s_3y)$.
    By hypothesis $s_3 \in \tau(s_3s_1w)$, so
    we can apply \myrefP{prop:KL23g} to obtain $P_{s_3y, s_3s_1w} = P_{s_4s_3y, s_3s_1w}$.
    Since $s_2 \in \tau(s_3s_1w)$ and $s_2 \notin \tau(s_4y)$,
    we can apply \myrefP{prop:KL23g} to obtain $P_{s_4y, s_3s_1w} = P_{s_2s_4y, s_3s_1w}$.
    Now, $l(s_4s_3y) = l(y)$ and $l(s_3s_1w) - l(w) - 2$, so $d(s_4s_3y, s_3s_1w) = d(y, w) - 1$.
    Thus
    \begin{equation*}
      qP_{s_3y, s_3s_1w} = qP_{s_4s_3y, s_3s_1w} \dsim \mu(s_4s_3y, s_3s_1w)q^d.
    \end{equation*}
    Also, $l(s_2s_4y) = l(y) - 1$, so $d(s_2s_4y, s_3s_1w) = d(y, w) - 2$.
    Thus
    \begin{equation*}
      q^2P_{s_4y, s_3s_1w} = q^2P_{s_2s_4y, s_3s_1w} \dsim \mu(s_2s_4y, s_3s_1w)q^d.
    \end{equation*}
    If we put these last two formulas into (**), we get (*), as desired.
  \end{proof}

  In the third stage, we improve on our understanding of each side of the inequalities derived in the previous three propositions.  The left-hand side of the inequalities in these propositions is the same, so we'll treat that first, in the cases which we'll need later.
  After that, we have three more lemmas, one for each of the three right-hand sides.

  \begin{lemma}
    \label{lem:LHS}
    With all notation as in \myrefT{thm:mainA}, assume that $i = 1$, and that if $\abs{C'} = 14$, then $C' \subset X(14, a, j)$ or $C' \subset X(14, b, j)$  with $j \neq 1$.
    Choose $L', U'$ with $l(L') < l(U')$.
    If $\abs{C'} = 10$, let $H' = s_1s_3U'$.

    Let
    \begin{equation*}
        LHS = \mu(y,w) + \sum_{\substack{z \in \WDFour w\\ s_1z < z}} \mu(y,z)\mu(z,s_1w)
    \end{equation*}
    \begin{enumerate}
      \item If $w = H'$ then $\mu(y, M') \leq LHS$.
      \item If $w = U'$ then $\mu(y, U') + \mu(y, L') \leq LHS$.
      \item If $w = M'$ then $\mu(y, M') \leq LHS$.
    \end{enumerate}
  \end{lemma}

  \begin{proof}
    For statement 1, we need to note that $s_1M' < M'$ and  $\mu(M', s_1H') = 1$.
    Thus $\mu(y, M')$ occurs in the sum portion of the equation.
    For statement 2, we need to note that $s_1L' < L'$ and $\mu(L', s_1U') = 1$.
    These both can be seen by inspection of the relevant diagram, and then applying \myrefP{prop:KLPolysParabolic}.
    The rest is obvious (given \myrefT{thm:nonNegative}).
  \end{proof}

  In the proofs of the next three lemmas, we will use expanded diagrams of our clumps, showing relevant nearby elements.
  It is easy to verify that the additional elements have the displayed $\tDFour$ values, either by looking at the explicit elements in $\WDFour$ or by using \myrefP{prop:parabolic2}, applied to parabolic subgroups of type $A_1 \times A_1$ and $A_2$.
  Also, we can see that solid lines correspond to $T_{i,j}$ maps just by looking at the $\tDFour$ values of the elements which they connect.

  \begin{lemma}
    \label{lem:Cs2}
    With all notation as in \myrefT{thm:mainA}, let $y \in \{L, M, U\}$, and suppose that $y$ is type $\scC$.
    Assume in addition that $i = 1$, and that if $\abs C = 14$ then $C \subset X(14, a, 4)$ or $C \subset X(14, b, 2)$.
    Let $w \in W$ with $\tDFour(w) = \bigRefSet$, and suppose $s_1w$ is type $\scB_1$, and $s_2s_1w$ is type $\scA_4$.
    Let
    \begin{equation*}
      \begin{split}
        &RHS = \\ &\mu(s_1y,s_1w) + \mu(s_3s_2y, s_2s_1w) + \mu(s_4s_3y, s_2s_1w)
        - \sum_{\substack{z \in \WDFour y\\ s_1z < z}}\mu(y,z)\mu(z,s_1w)
      \end{split}
    \end{equation*}
    Let $K = T_{3,4}T_{1,3}T_{3,2}(s_1w)$ (so $K$ is type $\scA_1$).
    Then we have the following:
    \begin{enumerate}
      \item If $y = U$ then $RHS \leq \mu(M, K)$.
      \item If $y = M$ then $RHS \leq \mu(L, K) + \mu(U, K)$.
      \item If $y = L$ then $RHS \leq \mu(M, K)$.
    \end{enumerate}
  \end{lemma}

  \begin{figure}[!ht]
    \begin{center}
      \begin{tikzpicture}[scale=0.8,  every node/.style={transform shape}]
        \node (CNode) at (0 * \xSLT, 0 * \ySLT) {\dynkinLabelWeyl{\TypeCPicture}{w}};
        \node (B1Node) at (0 * \xSLT, -1 * \ySLT) {\dynkinLabelWeyl{\TypeBPicture{1}}{ }};
        \node (A4Node) at (0 * \xSLT, -2 * \ySLT) {\dynkinLabelWeyl{\TypeAPicture{4}}{ }};

        \draw[-, very thick, shorten <=0.4cm, shorten >=0.4cm, magenta, dotted] (0 * \xSLT, -1 * \ySLT) -- (0 * \xSLT, -1 * \ySLT + \ySLT);
        \upLineT{green}{solid}{0}{-2}
      \end{tikzpicture}
    \end{center}
    \caption{$w$ for \myrefL{lem:Cs2}}
    \label{fig:wCs2}
  \end{figure}

  \begin{proof}
    As in \myrefR{rem:chooseLU}, we can
    choose $L, U$ with $l(L) < l(U)$.
    For $w$, refer to \autoref{fig:wCs2}.
    Note that $T_{3,2}(s_1w) = s_2s_1w$.
    To prove statement 1, note that if $y = U$ then $\abs C = 10$. Refer to \autoref{fig:type10Extra}.

    Using the figures, we see that $T_{3,4}T_{1,3}T_{3,2}(s_1U) = M$, so $\mu(s_1y, s_1w) = \mu(M, K)$.
    Now let $z_1 = T_{2,3}(s_3s_2y) = s_2s_3s_2y$.
    Then $\mu(s_3s_2y, s_2s_1w) = \mu(z_1, s_1w)$.
    Note that $s_1z_1 < z_1$ and $\mu(y, z_1) = 1$.
    Similarly, let $z_2 = T_{2,3}(s_4s_3y) = s_3s_4s_3y$.
    Then $\mu(s_4s_3y, s_2s_1w) = \mu(z_2, s_1w)$, $s_1z_2 < z_2$ and $\mu(y, z_2) = 1$.
    So, $z_1$ and $z_2$ occur in the sum portion of the definition of $RHS$.

    Putting this all together, we have
    \begin{equation*}
      \begin{split}
        &RHS = \\
        &\mu(s_1y,s_1w) + \mu(s_3s_2y, s_2s_1w) + \mu(s_4s_3y, s_2s_1w)
        - \sum_{\substack{z \in \WDFour y\\ s_1z < z}}\mu(y,z)\mu(z,s_1w)\\
        &= \mu(M,K) + \mu(z_1, s_1w) + \mu(z_2, s_1w) -  \mu(z_1, s_1w) - \mu(z_2, s_1w)\\
        &- \sum_{\substack{z \in \WDFour y\\ s_1z < z\\ z \notin \{z_1,z_2\}}}\mu(y,z)\mu(z,s_1w)\\
        &\leq \mu(M,K)
      \end{split}
    \end{equation*}
    This proves statement 1.

    We'll prove statement 3 next, since that also has $\abs C = 10$. So $y = L$ in \autoref{fig:type10Extra}.
    First we note that $\mu(s_1y, s_1w) = \mu(T_{3, 2}(s_1y), T_{3, 2}(s_1w)) = 0$, the latter equality by \myrefP{prop:KL23g} since $s_4 \in \tau(T_{3, 2}(s_1w))$, $s_4 \notin \tau(T_{3, 2}(s_1y))$.
    Similarly, $\mu(s_3s_2y, s_2s_1w) = 0$.
    Now
    \begin{equation*}
      \mu(s_4s_3y, s_2s_1w) = \mu(T_{3,4}T_{1,3}(s_4s_3y), T_{3,4}T_{1,3}(s_2s_1w)) = \mu(M, K).
    \end{equation*}
    Statement 3 follows easily from these.

    To prove statement 2, note that if $y = M$ then $\abs C = 14$. So, we'll have two cases.
    First, we assume that $C \subset X(14, a, 4)$.  Refer to \autoref{fig:type14aExtra}.
    Here, we see that $T_{3,4}T_{1,3}T_{3,2}(s_1M) = L$, so $\mu(s_1y, s_1w) = \mu(L, K)$.
    We have $\mu(s_3s_2y, s_2s_1w) = 0$ since $s_4 \in \tau(s_2s_1w)$, $s_4 \notin \tau(s_3s_2y)$.
    Now $\mu(s_4s_3y, s_2s_1w) = \mu(T_{3,4}T_{1,3}(s_4s_3y), T_{3,4}T_{1,3}(s_2s_1w)) = \mu(U, K)$.  Statement 2 in this case follows easily from these.

    Finally, we assume that $C \subset X(14, b, 2)$.  Refer to \autoref{fig:type14bExtraColorChange}.
    Here, we see again that $T_{3,4}T_{1,3}T_{3,2}(s_1M) = L$, so $\mu(s_1y, s_1w) = \mu(L, K)$.
    For $\mu(s_3s_2y, s_2s_1w)$, we argue as in the proof of statement 1.
    We let $z_1 = T_{2,3}(s_3s_2y) = s_2s_3s_2y$.
    Then $\mu(s_3s_2y, s_2s_1w) = \mu(z_1, s_1w)$.
    Note that $s_1z_1 < z_1$ and $\mu(y, z_1) = 1$.
    For the last term, we have
    \begin{equation*}
      \mu(s_4s_3y, s_2s_1w) = \mu(T_{3,4}T_{1,3}(s_4s_3y), T_{3,4}T_{1,3}(s_2s_1w)) = \mu(U, K).
    \end{equation*}
    Statement 2 in this case now follows from arguments which we have seen already in this proof.
    This completes the proof of the Lemma.
  \end{proof}

  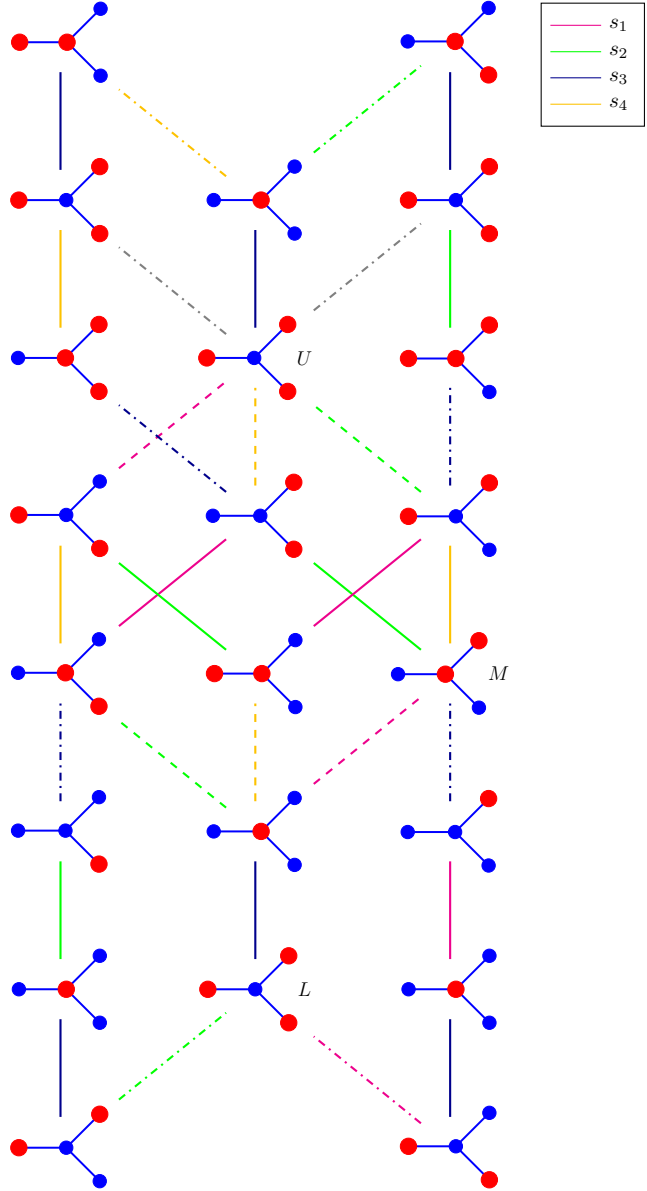
\begin{figure}[!ht]
    \begin{center}
      \begin{tikzpicture}[scale=0.7,  every node/.style={transform shape}]
        \node (CNode) at (0, 0) {\dynkinLabelWeyl{\TypeCPicture}{L}};
        \node (DNode) at (0, \ySLT) {\dynkinLabelWeyl{\TypeDPicture}{ }};
        \node (A2Node) at (-1 * \xSLT, 2 * \ySLT) {\dynkinLabelWeyl{\TypeAPicture{2}}{ }};
        \node (A4Node) at (0, 2 * \ySLT) {\dynkinLabelWeyl{\TypeAPicture{4}}{ }};
        \node (A1Node) at (1 * \xSLT, 2 * \ySLT) {\dynkinLabelWeyl{\TypeAPicture{1}}{M}};
        \node (B1Node) at (-1 * \xSLT, 3 * \ySLT) {\dynkinLabelWeyl{\TypeBPicture{1}}{ }};
        \node (B4Node) at (0, 3 * \ySLT) {\dynkinLabelWeyl{\TypeBPicture{4}}{ }};
        \node (B2Node) at (1 * \xSLT, 3 * \ySLT) {\dynkinLabelWeyl{\TypeBPicture{2}}{ }};
        \node (CPNode) at (0, 4 * \ySLT) {\dynkinLabelWeyl{\TypeCPicture}{U}};
        \node (DPNode) at (0, 5 * \ySLT) {\dynkinLabelWeyl{\TypeDPicture}{ }};

        \node (F4Node) at (-1 * \xSLT, 4 * \ySLT) {\dynkinLabelWeyl{\TypeFPicture{4}}{ }};
        \node (GNode) at (-1 * \xSLT, 5 * \ySLT) {\dynkinLabelWeyl{\TypeCPicture}{ }};
        \node (E4Node) at (-1 * \xSLT, 6 * \ySLT) {\dynkinLabelWeyl{\TypeAPicture{4}}{ }};

        \node (F2Node) at (1 * \xSLT, 4 * \ySLT) {\dynkinLabelWeyl{\TypeFPicture{2}}{ }};
        \node (GPNode) at (1 * \xSLT, 5 * \ySLT) {\dynkinLabelWeyl{\TypeCPicture}{ }};
        \node (E2Node) at (1 * \xSLT, 6 * \ySLT) {\dynkinLabelWeyl{\TypeAPicture{2}}{ }};

        \node (I2Node) at (-1 * \xSLT, 1 * \ySLT) {\dynkinLabelWeyl{\TypeIPicture{2}}{ }};
        \node (KNode) at (-1 * \xSLT, 0 * \ySLT) {\dynkinLabelWeyl{\TypeDPicture}{ }};
        \node (J2Node) at (-1 * \xSLT, -1 * \ySLT) {\dynkinLabelWeyl{\TypeBPicture{2}}{ }};

        \node (I1Node) at (1 * \xSLT, 1 * \ySLT) {\dynkinLabelWeyl{\TypeIPicture{1}}{ }};
        \node (KPNode) at (1 * \xSLT, 0 * \ySLT) {\dynkinLabelWeyl{\TypeDPicture}{ }};
        \node (J1Node) at (1 * \xSLT, -1 * \ySLT) {\dynkinLabelWeyl{\TypeBPicture{1}}{ }};

        \upLineT{darkblue}{solid}{0}{0}
        \upLineT{goldenpoppy}{dashed}{0}{1}
        \rightLineT{magenta}{dashed}{0}{1}
        \leftLineT{green}{dashed}{0}{1}
        \rightLineT{magenta}{solid}{-1}{2}
        \upLineT{goldenpoppy}{solid}{-1}{2}
        \leftLineT{green}{solid}{0}{2}
        \leftLineT{green}{solid}{1}{2}
        \upLineT{goldenpoppy}{solid}{1}{2}
        \rightLineT{magenta}{solid}{0}{2}
        \leftLineT{green}{dashed}{1}{3}
        \rightLineT{magenta}{dashed}{-1}{3}
        \upLineT{goldenpoppy}{dashed}{0}{3}
        \upLineT{darkblue}{solid}{0}{4}

        \leftLineT{darkblue}{dashdotted}{0}{3}
        \upLineT{goldenpoppy}{solid}{-1}{4}
        \upLineT{darkblue}{solid}{-1}{5}
        \leftLineT{goldenpoppy}{dashdotted}{0}{5}

        \upLineT{darkblue}{dashdotted}{1}{3}
        \upLineT{green}{solid}{1}{4}
        \upLineT{darkblue}{solid}{1}{5}
        \rightLineT{green}{dashdotted}{0}{5}

        \upLineT{darkblue}{dashdotted}{-1}{1}
        \upLineT{green}{solid}{-1}{0}
        \upLineT{darkblue}{solid}{-1}{-1}
        \rightLineT{green}{dashdotted}{-1}{-1}

        \upLineT{darkblue}{dashdotted}{1}{1}
        \upLineT{magenta}{solid}{1}{0}
        \upLineT{darkblue}{solid}{1}{-1}
        \leftLineT{magenta}{dashdotted}{1}{-1}

        \leftLineT{gray}{dashdotted}{0}{4}
        \rightLineT{gray}{dashdotted}{0}{4}

        \matrix  [draw, below left, every node/.style={scale=.8}]
         at (2 * \xSLT, 6.25 * \ySLT)
        {
          \node [legendLine, draw=magenta,label=right:{$s_1$}] {}; \\
          \node [legendLine, draw=green,label=right:{$s_2$}] {}; \\
          \node [legendLine, draw=darkblue,label=right:{$s_3$}] {}; \\
          \node [legendLine, draw=goldenpoppy,label=right:{$s_4$}] {}; \\
        };

        \draw [color=white] (-2 * \xSLT, 5.25 * \ySLT) rectangle (-1.5 * \xSLT, 6.25 * \ySLT);
      \end{tikzpicture}
    \end{center}
    \caption{$C(10, a)$ or $C(10, b)$}
    \label{fig:type10Extra}
  \end{figure}

  \begin{figure}[!ht]
    \begin{center}
      \begin{scaletikzpicturetowidth}{\textwidth}
        \begin{tikzpicture}[scale=\tikzscale, every node/.style={transform shape}]
          \node (A4Node) at (0,0) {\dynkinLabelWeyl{\TypeAPicture{4}}{ }};
          \node (B2Node) at (-1 * \xSLF, 1 * \ySLF) {\dynkinLabelWeyl{\TypeBPicture{2}}{ }};
          \node (B1Node) at (1 * \xSLF, 1 * \ySLF) {\dynkinLabelWeyl{\TypeBPicture{1}}{ }};
          \node (A1Node) at (-2 * \xSLF, 2 * \ySLF) {\dynkinLabelWeyl{\TypeAPicture{1}}{L}};
          \node (CNode) at (0, 2 * \ySLF) {\dynkinLabelWeyl{\TypeCPicture}{M}};
          \node (A2Node) at (2 * \xSLF, 2 * \ySLF) {\dynkinLabelWeyl{\TypeAPicture{2}}{ }};
          \node (B4Node) at (-2 * \xSLF, 3 * \ySLF) {\dynkinLabelWeyl{\TypeBPicture{4}}{ }};
          \node (DNode) at (0, 3 * \ySLF) {\dynkinLabelWeyl{\TypeDPicture}{ }};
          \node (B4PNode) at (2 * \xSLF, 3 * \ySLF) {\dynkinLabelWeyl{\TypeBPicture{4}}{ }};
          \node (A2PNode) at (-1 * \xSLF, 4 * \ySLF) {\dynkinLabelWeyl{\TypeAPicture{2}}{ }};
          \node (A4PNode) at (0, 4 * \ySLF) {\dynkinLabelWeyl{\TypeAPicture{4}}{ }};
          \node (A1PNode) at (1 * \xSLF, 4 * \ySLF) {\dynkinLabelWeyl{\TypeAPicture{1}}{U}};
          \node (B1PNode) at (-1 * \xSLF, 5 * \ySLF) {\dynkinLabelWeyl{\TypeBPicture{1}}{ }};
          \node (B2PNode) at (1 * \xSLF, 5 * \ySLF) {\dynkinLabelWeyl{\TypeBPicture{2}}{ }};

          \node (F4Node) at (0 * \xSLF, 5 * \ySLF) {\dynkinLabelWeyl{\TypeFPicture{4}}{ }};
          \node (GNode) at (0 * \xSLF, 6 * \ySLF) {\dynkinLabelWeyl{\TypeCPicture}{ }};

          \node (GPNode) at (3 * \xSLF, 4 * \ySLF) {\dynkinLabelWeyl{\TypeCPicture}{ }};
          \node (E1Node) at (3 * \xSLF, 5 * \ySLF) {\dynkinLabelWeyl{\TypeAPicture{1}}{ }};
          \node (F2Node) at (2 * \xSLF, 6 * \ySLF) {\dynkinLabelWeyl{\TypeFPicture{2}}{ }};

          \node (I4Node) at (1 * \xSLF, 3 * \ySLF) {\dynkinLabelWeyl{\TypeIPicture{4}}{ }};
          \node (KNode) at (1 * \xSLF, 2 * \ySLF) {\dynkinLabelWeyl{\TypeDPicture}{ }};
          \node (J4Node) at (0 * \xSLF, 1 * \ySLF) {\dynkinLabelWeyl{\TypeBPicture{4}}{ }};

          \upLineF{darkblue}{solid}{0}{2}
          \upLineF{goldenpoppy}{dashed}{0}{3}
          \rightLineF{magenta}{dashed}{0}{3}
          \leftLineF{green}{dashed}{0}{3}
          \upLineF{goldenpoppy}{solid}{-1}{4}
          \leftLineF{green}{solid}{0}{4}
          \upLineF{goldenpoppy}{solid}{1}{4}
          \rightLineF{magenta}{solid}{0}{4}
          \rightLineF{darkblue}{solid}{-2}{3}
          \leftLineF{darkblue}{solid}{2}{3}
          \upLineF{green}{solid}{-2}{2}
          \upLineF{magenta}{solid}{2}{2}
          \leftLineF{darkblue}{solid}{-1}{1}
          \rightLineF{green}{dashed}{-1}{1}
          \leftLineF{magenta}{dashed}{1}{1}
          \rightLineF{darkblue}{solid}{1}{1}
          \leftLineF{magenta}{solid}{0}{0}
          \rightLineF{green}{solid}{0}{0}

          \rightLineF{magenta}{dashdotted}{-1}{4}
          \leftLineF{green}{dashdotted}{1}{4}
          \upLineF{goldenpoppy}{solid}{0}{5}
          \rightLineF{magenta}{dashdotted}{-1}{5}
          \leftLineF{green}{dashdotted}{1}{5}

          \rightLineF{goldenpoppy}{dashdotted}{2}{3}
          \upLineF{darkblue}{solid}{3}{4}
          \rightLineF{darkblue}{dashdotted}{1}{5}
          \leftLineF{goldenpoppy}{dashdotted}{3}{5}

          \leftLineF{darkblue}{dashdotted}{1}{3}
          \upLineF{goldenpoppy}{solid}{1}{2}
          \rightLineF{darkblue}{solid}{0}{1}
          \upLineF{goldenpoppy}{dashdotted}{0}{1}

          \draw[-, thick, shorten <=1cm, shorten >=0.5cm, gray, dashdotted] (1 * \xSLF, 4 * \ySLF) -- (3 * \xSLF, 5 * \ySLF);

          \matrix [draw, below right, every node/.style={scale=.9}]
          at (current bounding box.north west)
          {
            \node [legendLine, draw=magenta,label=right:{$s_1$}] {}; \\
            \node [legendLine, draw=green,label=right:{$s_2$}] {}; \\
            \node [legendLine, draw=darkblue,label=right:{$s_3$}] {}; \\
            \node [legendLine, draw=goldenpoppy,label=right:{$s_4$}] {}; \\
          };
        \end{tikzpicture}
      \end{scaletikzpicturetowidth}
    \end{center}
    \caption{$C(14, a, 4)$}
    \label{fig:type14aExtra}
  \end{figure}

  \begin{figure}[!ht]
    \begin{center}
        \begin{tikzpicture}[scale=.7, every node/.style={transform shape}]
          \node (A1Node) at (-1 * \xSLF, 0 * \ySLF)
          {\dynkinLabelWeyl{\TypeAPicture{1}}{L}};
          \node (A4Node) at (1 * \xSLF, 0 * \ySLF)
          {\dynkinLabelWeyl{\TypeAPicture{4}}{ }};
          \node (B4Node) at (-1 * \xSLF, 1 * \ySLF) {\dynkinLabelWeyl{\TypeBPicture{4}}{ }};
          \node (B2Node) at (0 * \xSLF, 1 * \ySLF) {\dynkinLabelWeyl{\TypeBPicture{2}}{ }};
          \node (B1Node) at (1 * \xSLF, 1 * \ySLF) {\dynkinLabelWeyl{\TypeBPicture{1}}{ }};
          \node (A2Node) at (-2 * \xSLF, 2 * \ySLF) {\dynkinLabelWeyl{\TypeAPicture{2}}{ }};
          \node (CNode) at (0 * \xSLF, 2 * \ySLF) {\dynkinLabelWeyl{\TypeCPicture}{M}};
          \node (A2PNode) at (2 * \xSLF, 2 * \ySLF) {\dynkinLabelWeyl{\TypeAPicture{2}}{ }};
          \node (B1PNode) at (-2 * \xSLF, 3 * \ySLF) {\dynkinLabelWeyl{\TypeBPicture{1}}{ }};
          \node (DNode)  at (0 * \xSLF, 3 * \ySLF) {\dynkinLabelWeyl{\TypeDPicture}{ }};
          \node (B4PNode)  at (2 * \xSLF, 3 * \ySLF) {\dynkinLabelWeyl{\TypeBPicture{4}}{ }};
          \node (A4PNode) at (-1 * \xSLF, 4 * \ySLF) {\dynkinLabelWeyl{\TypeAPicture{4}}{ }};
          \node (A1PNode) at (1 * \xSLF, 4 * \ySLF) {\dynkinLabelWeyl{\TypeAPicture{1}}{U}};
          \node (B2PNode) at (0 * \xSLF, 5 * \ySLF) {\dynkinLabelWeyl{\TypeBPicture{2}}{ }};

          \node (E2Node) at (0 * \xSLF, 4 * \ySLF) {\dynkinLabelWeyl{\TypeAPicture{2}}{ }};
          \node (F1Node) at (-1 * \xSLF, 5 * \ySLF) {\dynkinLabelWeyl{\TypeFPicture{1}}{ }};
          \node (F4Node) at (1 * \xSLF, 5 * \ySLF) {\dynkinLabelWeyl{\TypeFPicture{4}}{ }};
          \node (GNode) at (0 * \xSLF, 6 * \ySLF) {\dynkinLabelWeyl{\TypeCPicture}{ }};

          \node (GPNode) at (2 * \xSLF, 4 * \ySLF) {\dynkinLabelWeyl{\TypeCPicture}{ }};
          \node (E1Node) at (2 * \xSLF, 5 * \ySLF) {\dynkinLabelWeyl{\TypeAPicture{1}}{ }};
          \node (F2Node) at (1 * \xSLF, 6 * \ySLF) {\dynkinLabelWeyl{\TypeFPicture{2}}{ }};

          \node (F2PNode) at (1 * \xSLF, 2 * \ySLF) {\dynkinLabelWeyl{\TypeFPicture{2}}{ }};
          \node (GPPNode) at (1 * \xSLF, 3 * \ySLF) {\dynkinLabelWeyl{\TypeCPicture}{ }};

          \upLineF{darkblue}{solid}{0}{2}
          \rightLineF{magenta}{dashed}{0}{3}
          \leftLineF{goldenpoppy}{dashed}{0}{3}
          \rightLineF{magenta}{solid}{-1}{4}
          \leftLineF{goldenpoppy}{solid}{1}{4}
          \upLineF{green}{dashed}{0}{1}
          \rightLineF{goldenpoppy}{solid}{-1}{0}
          \leftLineF{magenta}{solid}{1}{0}
          \leftLineF{magenta}{dashed}{1}{1}
          \rightLineF{goldenpoppy}{dashed}{-1}{1}
          \upLineF{green}{solid}{-1}{0}
          \upLineF{green}{solid}{1}{0}
          \leftLineF{darkblue}{solid}{-1}{1}
          \upLineF{goldenpoppy}{solid}{-2}{2}
          \rightLineF{darkblue}{solid}{-2}{3}
          \rightLineF{darkblue}{solid}{1}{1}
          \upLineF{magenta}{solid}{2}{2}
          \leftLineF{darkblue}{solid}{2}{3}

          \upLineF{green}{dashdotted}{0}{3}
          \upLineF{green}{dashdotted}{-1}{4}
          \leftLineF{goldenpoppy}{dashdotted}{0}{4}
          \upLineF{green}{dashdotted}{1}{4}
          \rightLineF{magenta}{dashdotted}{0}{4}
          \upLineF{green}{dashdotted}{0}{5}
          \rightLineF{magenta}{solid}{-1}{5}
          \leftLineF{goldenpoppy}{solid}{1}{5}

          \upLineF{goldenpoppy}{dashdotted}{2}{3}
          \upLineF{darkblue}{solid}{2}{4}
          \rightLineF{darkblue}{dashdotted}{0}{5}
          \leftLineF{goldenpoppy}{dashdotted}{2}{5}

          \rightLineF{darkblue}{dashdotted}{0}{1}
          \upLineF{green}{solid}{1}{2}
          \leftLineF{darkblue}{solid}{1}{3}

          \rightLineF{gray}{dashdotted}{0}{2}
          \rightLineF{gray}{dashdotted}{1}{4}

          \matrix [draw, below right, every node/.style={scale=.9}]
          at (current bounding box.north west)
          {
            \node [legendLine, draw=magenta,label=right:{$s_1$}] {}; \\
            \node [legendLine, draw=green,label=right:{$s_2$}] {}; \\
            \node [legendLine, draw=darkblue,label=right:{$s_3$}] {}; \\
            \node [legendLine, draw=goldenpoppy,label=right:{$s_4$}] {}; \\
          };
        \end{tikzpicture}
    \end{center}
    \caption{$C(14, b, 2)$}
    \label{fig:type14bExtraColorChange}
  \end{figure}
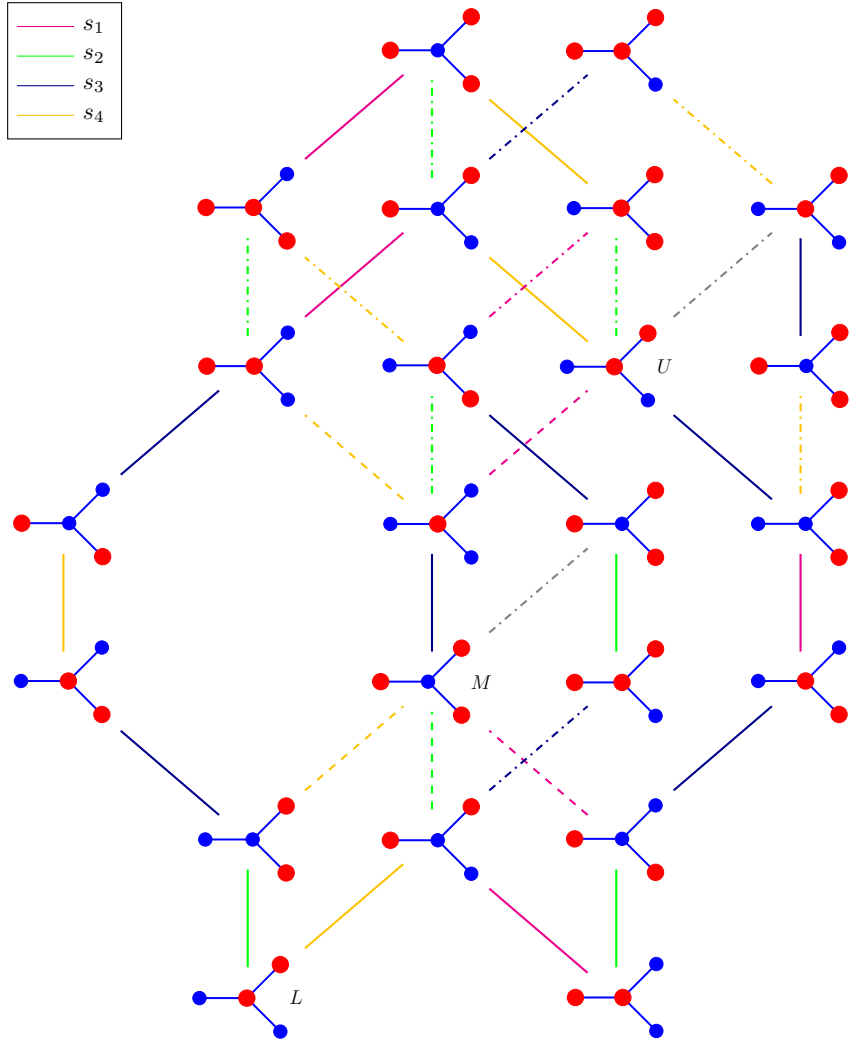

  \begin{lemma}
    \label{lem:Cs4}
    With all notation as in \myrefT{thm:mainA}, let  $y \in \{L, M, U\}$.
    Assume in addition that $i = 1$, and that if $\abs C = 14$ then $C \subset X(14, a, 4)$ or $C \subset X(14, b, 2)$.
    Let
    \begin{equation*}
      \begin{split}
        &RHS = \\ &\mu(s_1y,s_1w) + \mu(s_3s_4y, s_4s_1w) + \mu(s_2s_3y, s_4s_1w)
        - \sum_{\substack{z \in \WDFour y\\ s_1z < z}}\mu(y,z)\mu(z,s_1w)
      \end{split}
    \end{equation*}
    where $y$ is type $\scC$, and $w$ satisfies $\tDFour(w) = \bigRefSet$, $s_1w$ is type $\scB_1$ and $s_4s_1w$ is type $\scA_2$.
    Let $K = T_{3,2}T_{1,3}T_{3,4}(s_1w)$ (so $K$ is type $\scA_1$).
    \begin{enumerate}
      \item If $y = U$ then $RHS \leq \mu(M, K)$.
      \item If $y = M$ then $RHS \leq \mu(L, K) + \mu(U, K)$.
      \item If $y = L$ then $RHS \leq \mu(M, K)$.
    \end{enumerate}
  \end{lemma}

  \begin{figure}[!ht]
    \centering
    \begin{tikzpicture}[scale=0.8,  every node/.style={transform shape}]
      \node (CNode) at (0 * \xSLT, 0 * \ySLT) {\dynkinLabelWeyl{\TypeCPicture}{w}};
      \node (B1Node) at (0 * \xSLT, -1 * \ySLT) {\dynkinLabelWeyl{\TypeBPicture{1}}{ }};
      \node (A2Node) at (0 * \xSLT, -2 * \ySLT) {\dynkinLabelWeyl{\TypeAPicture{2}}{ }};

      \draw[-, very thick, shorten <=0.4cm, shorten >=0.4cm, magenta, dotted] (0 * \xSLT, -1 * \ySLT) -- (0 * \xSLT, -1 * \ySLT + \ySLT);
      \upLineT{goldenpoppy}{solid}{0}{-2}
    \end{tikzpicture}
    \caption{$w$ for \myrefL{lem:Cs4}}
    \label{fig:wCs4}
  \end{figure}

  \begin{proof}
    This lemma is the same as the previous lemma, just with $s_2$ and $s_4$ interchanged.
    For $w$, refer to \autoref{fig:wCs4}.
    Choose $L, U$ with $l(L) < l(U)$.
    Since any clump of size 10 is symmetric in $s_2$ and $s_4$, we only need to revisit statement 2, that is, the two cases where $\abs C = 14$ and $y = M$.
    First, we assume that $C \subset X(14, a, 4)$.  Refer to \autoref{fig:type14aExtra}.
    Here, we see that $T_{3,2}T_{1,3}T_{3,4}(s_1M) = U$, so $\mu(s_1y, s_1w) = \mu(U, K)$.
    We have $\mu(s_3s_4y, s_4s_1w) = 0$ since $s_2 \in \tau(s_4s_1w)$, $s_2 \notin \tau(s_3s_4y)$.
    Now
    \begin{equation*}
      \mu(s_2s_3y, s_4s_1w) = \mu(T_{3,2}T_{1,3}(s_2s_3y), T_{3,2}T_{1,3}(s_4s_1w)) = \mu(L, K).
    \end{equation*}
    Statement 2 in this case follows easily from these.

    Finally, we assume that $C \subset X(14, b, 2)$.  Refer to \autoref{fig:type14bExtraColorChange}.
    Here, we see again that $T_{3,2}T_{1,3}T_{3,4}(s_1M) = U$, so $\mu(s_1y, s_1w) = \mu(U, K)$.
    We have
    \begin{equation*}
      \mu(s_3s_4y, s_4s_1w) = \mu(T_{3,2}T_{1,3}(s_3s_4y), T_{3,2}T_{1,3}(s_4s_1w)) = \mu(L, K).
    \end{equation*}
    For $\mu(s_2s_3y, s_4s_1w)$, we argue as we have done before.
    We let $z_1 = T_{2,3}(s_2s_3y) = s_3s_2s_3y$.
    Then $\mu(s_2s_3y, s_4s_1w) = \mu(z_1, s_1w)$.
    Note that $s_1z_1 < z_1$ and $\mu(y, z_1) = 1$.
    Statement 2 in this case now follows from arguments which we have seen already in these proofs.
    This completes the proof of the Lemma.
  \end{proof}

  \begin{lemma}
    \label{lem:As4}
    With all notation as in \myrefT{thm:mainA}, let  $y \in \{L, M, U\}$.
    Assume in addition that $i = 1$, and that if $\abs C = 14$ then $C \subset X(14, a, 4)$ or $C \subset X(14, b, 2)$.
    Choose $L, U$ with $l(L) < l(U)$.
    Let
    \begin{equation*}
      \begin{split}
        &RHS = \\ &\mu(s_1y,s_1w) + \mu(s_4s_3y, s_3s_1w) + \mu(s_2s_4y, s_3s_1w)
        - \sum_{\substack{z \in \WDFour y\\ s_1z < z}}\mu(y,z)\mu(z,s_1w)
      \end{split}
    \end{equation*}
    where $y$ is type $\scA_1$, $s_4(y)$ is of type $\scB_2$, and $w$ satisfies $\tDFour(w) = \{s_1, s_3\}$ and $s_1w$ is type $\scD$.
    Let $K = s_3s_1w = T_{i,3}(s_1w)$ for $i \in \refSet$ (so $K$ is type $\scC$).
    \begin{enumerate}
      \item If $y = U$ then $RHS \leq \mu(M, K)$.
      \item If $y = M$ then $RHS \leq \mu(L, K) + \mu(U, K)$.
    \end{enumerate}
  \end{lemma}

  \begin{proof}
    To prove statement 1, note that if $y = U$ then $\abs C = 14$. So, we'll have two cases.
    First, we assume that $C \subset X(14, a, 4)$.  Refer to \autoref{fig:type14aExtra}.

    Using the figure, we see that $T_{1,3}(s_1U) = M$, so $\mu(s_1y, s_1w) = \mu(M, K)$.
    Now let $z_1 = T_{3,4}(s_4s_3y) = s_3s_4s_3y$.
    So $\mu(s_4s_3y, s_3s_1w) = \mu(z_1, s_1w)$.
    Note that $s_1z_1 < z_1$ and $\mu(y, z_1) = 1$.
    Similarly, let $z_2 = T_{3,4}(s_2s_4y) = s_2y$.
    Then $\mu(s_4s_3y, s_2s_1w) = \mu(z_2, s_1w)$, $s_1z_2 < z_2$ and $\mu(y, z_2) = 1$.
    So, $z_1$ and $z_2$ occur in the sum portion of the definition of $RHS$.
    As usual, this proves statement 1 in this case.
    The proof of statement 1 when $C \subset X(14, b, 2)$ is the same.  (Refer to \autoref{fig:type14bExtraColorChange}.)

    For statement 2, we have $\abs C = 10$.
    Refer to \autoref{fig:type10Extra}.
    Using the figure, we see that $T_{1,3}(s_1M) = L$, so $\mu(s_1y, s_1w) = \mu(L, K)$.
    For $\mu(s_4s_3y, s_3s_1w)$, we note that $s_2 \notin \tau(s_3M)$, and thus $s_2 \notin \tau(s_4s_3M)$.
    Since $s_2 \in \tau(s_3s_1w)$, we have $\mu(s_4s_3y, s_3s_1w) = 0$, by \myrefP{prop:KL23g}.
    We have $s_2s_4M = U$, so $\mu(s_2s_4y, s_3s_1w) = \mu(U, K)$.
    As usual, these combine to prove statement 2.
  \end{proof}

  With these lemmas in hand, we can now prove statements 1 and 2 of \myrefT{thm:mainA}.
  For convenience, we'll break this proof into four lemmas, one for each of the cases, based on $\abs C$ and $\abs{C'}$.

  \begin{lemma}
    \label{lem:mainA1010}
    \myrefT{thm:mainA} holds when $\abs C = 10$ and $\abs{C'} = 10$.
  \end{lemma}

  \begin{proof}
    First, we'll prove the lemma under the additional assumption that $i = 1$.
    We can apply \myrefP{prop:CTop}, \myrefL{lem:LHS}, and \myrefL{lem:Cs2} to two $y, w$ pairs.
    They, with the resulting inequalities, are as follows:
    \begin{equation*}
      \begin{array}{| c | c | c |}
        \hline
        y & w & \text{Inequality}\\ \hline
        U & U' & \mu(U,U') + \mu(U, L') \leq \mu(M, M')\\ \hline
        L & U' & \mu(L,U') + \mu(L, L') \leq \mu(M, M')\\ \hline
      \end{array}
    \end{equation*}

    Similarly, we can apply \myrefP{prop:ATopAlt}, \myrefL{lem:LHS}, and \myrefL{lem:As4} to two $y, w$ pairs.
    They, with the resulting inequalities, are as follows:

    \begin{equation*}
      \begin{array}{| c | c | c |}
        \hline
        y & w & \text{Inequality}\\ \hline
        M & M' & \mu(M, M') \leq \mu(L, L') + \mu(U, L')\\ \hline
        M & H' & \mu(M, M') \leq \mu(L, U') + \mu(U, U')\\ \hline
      \end{array}
    \end{equation*}

    It's easy to go from these inequalities to the equalities of \myrefT{thm:mainA}.

    So, the lemma is proved when $i = 1$.
    Now, for $j \in \refSet$, write $M(j)$ for the element of type $\scA_j$ in $C$, and similarly $M'(j)$ in $C'$.

    Using this notation, what we've proved so far is that
    \begin{align*}
      &\mu(M(1), M'(1)) = \mu(L, L') + \mu(U, L')\\
      &\mu(U, U') = \mu(L, L')\\
      &\mu(U, L') = \mu(L, U')
    \end{align*}

    We want to show that $\mu(M(2), M'(2)) = \mu(M(1), M'(1))$, and similarly with 4 in place of 2.
    Now $M(2) = T_{3,1}T_{2,3}(M(1))$ and $M'(2) = T_{3,1}T_{2,3}(M'(1))$.
    So, by \myrefT{thm:talb}, we have
    \begin{equation*}
      \mu(M(2), M'(2)) = \mu(M(1), M'(1))
    \end{equation*}
    Similarly, $\mu(M(4), M'(4)) = \mu(M(1), M'(1))$.
    So, we have now proved the lemma in all cases.
  \end{proof}

  We'll use the following proposition in the rest of the cases.

  \begin{proposition}
    \label{prop:semicircle}
    Let $C$ be a clump and $w \in C$ an element of type $\scA_1$.
    Let $T = T_{3,4}T_{1,3}T_{3,2}T_{4,3}T_{3,1}T_{2,3}$.
    Then $T(w)$ is defined and is also of type $\scA_1$.
    If $\abs C = 10$ then $T(w) = w$.
    If $\abs C = 14$ then $T(w) \neq w$.
  \end{proposition}

  \begin{proof}
    This can be seen by inspection.  To go from $\WDFour$ to $W$ we use \myrefP{prop:clumpDiagram}.
  \end{proof}

  \begin{lemma}
    \label{lem:mainA1014}
    \myrefT{thm:mainA} holds when $\abs C = 10$,  $\abs{C'} = 14$, and $C' \subset X(14, a, 4)$ or $C' \subset X(14, b, 2)$.
  \end{lemma}

  \begin{proof}
    As before, we'll first prove the lemma when $i = 1$.
    Using \myrefT{thm:talb} and \myrefP{prop:semicircle}, we have the equality:
    \begin{equation*}
      \mu(M, U') = \mu(M, L')
    \end{equation*}

    We will need three inequalities in addition.

    We can apply \myrefP{prop:CTop}, \myrefL{lem:LHS}, and \myrefL{lem:Cs2} to two $y, w$ pairs.
    They, with the resulting inequalities, are as follows:
    \begin{equation*}
      \begin{array}{| c | c | c |}
        \hline
        y & w & \text{Inequality}\\ \hline
        U & M' & \mu(U,M') \leq \mu(M, L')\\ \hline
        L & M' & \mu(L,M') \leq \mu(M, L')\\ \hline
      \end{array}
    \end{equation*}

    We can apply \myrefP{prop:ATopAlt}, \myrefL{lem:LHS}, and \myrefL{lem:As4} with $y = M$ and $w = U'$ to obtain the inequality
    \begin{equation*}
        \mu(M, U') + \mu(M,L') \leq \mu(L, M') + \mu(U, M')
    \end{equation*}

    The equalities of \myrefT{thm:mainA} now follow trivially.

    This proves the lemma when $i = 1$.
    Now, for $j \in \refSet$, write $M(j)$ for the element of type $\scA_j$ in $C$.
    For $j \in \refSet$, let $L'(j), U'(j)$ be the two elements of type $\scA_j$ in $C'$, with $l(L'(j)) \leq l(U'(j))$.
    Using this notation, what we've proved so far is
    \begin{equation*}
        \mu(U,M') = \mu(M(1), L'(1)) = \mu(L,M') = \mu(M(1), U'(1))
    \end{equation*}

    We want to have the same equations with 2 in place of 1, and similarly with 4 in place of 1.
    We have $M(2) = T_{3,1}T_{2,3}(M(1))$ and $M(4) = T_{3,1}T_{4,3}(M(1))$.
    We also have $T_{3,1}T_{2,3}(\{L'(1), U'(1)\}) = \{L'(2), U'(2)\}$, and similarly with 4 in place of 2.
    So, clearly, we also have the desired equations.
  \end{proof}

  \begin{lemma}
    \label{lem:mainA1410}
    \myrefT{thm:mainA} holds when $\abs C = 14$, $\abs{C'} = 10$, and $C \subset X(14, a, 4)$ or $C \subset X(14, b, 2)$.
  \end{lemma}

  \begin{proof}
    Assume first $i = 1$.
    Using \myrefT{thm:talb} and \myrefP{prop:semicircle}, we have the equality:
    \begin{equation*}
      \mu(U, M') = \mu(L, M')
    \end{equation*}

    We will need three inequalities in addition.

    We can apply \myrefP{prop:CTop}, \myrefL{lem:LHS}, and \myrefL{lem:Cs2} with $y = M$ and $w = U'$ to obtain the inequality
    \begin{equation*}
      \mu(M, U') + \mu(M,L') \leq \mu(L, M') + \mu(U, M')
    \end{equation*}

    Similarly, we can apply \myrefP{prop:ATopAlt}, \myrefL{lem:LHS}, and \myrefL{lem:As4} to two $y, w$ pairs.
    They, with the resulting inequalities, are as follows:
    \begin{equation*}
      \begin{array}{| c | c | c |}
        \hline
        y & w & \text{Inequality}\\ \hline
        U & H' & \mu(U,M') \leq \mu(M, U')\\ \hline
        U & M' & \mu(U,M') \leq \mu(M, L')\\ \hline
      \end{array}
    \end{equation*}

    The equalities of \myrefT{thm:mainA} now follow trivially.
    This proves the lemma when $i = 1$.  The arguments in the previous lemma for the cases where $i \neq 1$ work here as well.
  \end{proof}

  \begin{lemma}
    \label{lem:mainA1414}
    \myrefT{thm:mainA} holds when $\abs C = 14$,  $\abs{C'} = 14$, $C \subset X(14, a, 4)$ or $C \subset X(14, b, 2)$, and $C' \subset X(14, a, j)$ or $C' \subset X(14, b, j)$  with $j \neq 1$.
  \end{lemma}

  \begin{proof}
    Assume first that $i = 1$.
    Using \myrefT{thm:talb} and \myrefP{prop:semicircle}, we have two equalities:
    \begin{enumerate}
      \item $\mu(U, U') = \mu(L, L')$
      \item $\mu(U, L') = \mu(L, U')$
    \end{enumerate}

    These are statements 2 and 3 of the theorem.
    We'll need two inequalities to supplement them.

    If $C' \subset X(14, a, 4)$ or $C' \subset X(14, b, 2)$, we can apply \myrefP{prop:CTop}, \myrefL{lem:LHS}, and \myrefL{lem:Cs2} with $y = M$ and $w = M'$ to obtain the inequality
    \begin{equation*}
        \mu(M, M') \leq \mu(L, L') + \mu(U, L')
    \end{equation*}
    If instead $C' \subset X(14, a, 2)$ or $C' \subset X(14, b, 4)$, we can apply \myrefP{prop:CTopAlt}, \myrefL{lem:LHS}, and \myrefL{lem:Cs4} with $y = M$ and $w = M'$ to obtain the same inequality.

    We can apply \myrefP{prop:ATopAlt}, \myrefL{lem:LHS}, and \myrefL{lem:As4} with $y = U$ and $w = U'$ to obtain the inequality
    \begin{equation*}
        \mu(U,U') + \mu(U, L') \leq \mu(M, M')
    \end{equation*}
    The equalities of \myrefT{thm:mainA} now follow trivially.
    This proves the lemma when $i = 1$.

    Now we'll do the other cases.  For $j \in \refSet$, let $L(j), U(j)$ be the two elements of type $\scA_j$ in $C$, with $l(L(j)) \leq l(U(j))$.
    Let $L'(j), U'(j)$ be the two elements of type $\scA_j$ in $C'$, with $l(L'(j)) \leq l(U'(j))$.

    Using this notation, what we've proved so far is that
    \begin{align}
      \label{eqs:thmA1mod}
      \begin{split}
        &\mu(M, M') = \mu(L(1), L'(1)) + \mu(U(1), L'(1))\\
        &\mu(U(1), U'(1)) = \mu(L(1), L'(1))\\
        &\mu(U(1), L'(1)) = \mu(L(1), U'(1))
      \end{split}
    \end{align}

    We want to prove the same equations with 2 in place of 1 and with 4 in place of 1.
    Now $U(4) = T_{3,1}T_{4,3}(U(1))$ and $L(4) = T_{3,1}T_{4,3}(L(1))$.
    Also, $U(2) = T_{3,1}T_{2,3}(U(4))$ when $C \in X(14, a, 4)$.
    When $C \in X(14, b, 2)$, the two elements of type $A_2$ in $C$ have the same length, so we can choose $U(2)$ so that $U(2) = T_{3,1}T_{2,3}(U(4))$ in this case.
    Then in both cases, $L(2) = T_{3,1}T_{2,3}(L(4))$.
    In general (as can easily be seen) $T_{3,1}T_{4,3}(\{L'(1), U'(1)\}) = \{L'(4), U'(4)\}$ and  $T_{3,1}T_{2,3}(\{L'(4), U'(4)\}) = \{L'(2), U'(2)\}$.

    So, if we apply $T_{3,1}T_{4,3}$ to all the terms in the above three equations,  \myref{Equations}{eqs:thmA1mod}, and use \myrefT{thm:talb}, we obtain one of the following
    sets of equations:
    \begin{align}
      \label{eqs:thmA1modTransform1}
      \begin{split}
        &\mu(M, M') = \mu(L(4), L'(4)) + \mu(U(4), L'(4))\\
        &\mu(U(4), U'(4)) = \mu(L(4), L'(4))\\
        &\mu(U(4), L'(4)) = \mu(L(4), U'(4))
      \end{split}
    \end{align}
    or
    \begin{align}
      \label{eqs:thmA1modTransform2}
      \begin{split}
        &\mu(M, M') = \mu(L(4), U'(4)) + \mu(U(4), U'(4))\\
        &\mu(U(4), L'(4)) = \mu(L(4), U'(4))\\
        &\mu(U(4), U'(4)) = \mu(L(4), L'(4))
      \end{split}
    \end{align}

    \myref{Equations}{eqs:thmA1modTransform1} is the desired outcome, and the three equations in \myref{Equations}{eqs:thmA1modTransform2} are easily seen to be equivalent to those in   \myref{Equations}{eqs:thmA1modTransform1}.

    In the same way, we can go from \myref{Equations}{eqs:thmA1modTransform1} to
    \begin{align}
      \label{eqs:thmA1modTransform3}
      \begin{split}
        &\mu(M, M') = \mu(L(2), L'(2)) + \mu(U(2), L'(2))\\
        &\mu(U(2), U'(2)) = \mu(L(2), L'(2))\\
        &\mu(U(2), L'(2)) = \mu(L(2), U'(2))
      \end{split}
    \end{align}

    Finally, I need to address the fact that, when $C \subset X(14, b, 2)$, we made a choice of $L(2)$ and $U(2)$.
    If we make the other choice, this will interchange $L(2)$ and $U(2)$ in \myref{Equations}{eqs:thmA1modTransform3}, which is easily seen to result in an equivalent family of equations.

    This completes the proof of the lemma.
  \end{proof}

  \begin{proof}[Proof of \myrefT{thm:mainA}, Statements 1 and 2]
    This just combines \myrefL{lem:mainA1010}, \myrefL{lem:mainA1414}, \myrefL{lem:mainA1014}, and \myrefL{lem:mainA1410}, given \myrefR{rem:relabel}.
  \end{proof}

  Now we can go on to prove statement 3 of \myrefT{thm:mainA}.
  More precisely, we have the following.

  \begin{proposition}
    \label{prop:mainCase3}
    Let $C$ and $C'$ be clumps.
    Fix $i \in \refSet$.
    We define elements $L, M, U \in C$ as follows:
    if $\abs C = 10$ then $L$ and $U$ are the two elements of type $\scC$ in $C$, and $M$ is the one element of type $\scA_i$ in $C$.
    If instead $\abs C = 14$, then $L$ and $U$ are the two elements of type $\scA_i$ in $C$,
    and $M$ is the one element of type $\scC$ in $C$.
    We define similarly $L', M', U' \in C'$.
    Assume further that $C' \subset \WDFour C$.
    Let $y \in C$ and $w \in C'$, with $y$ and $w$ of the same type.
    We have the following:
    \begin{enumerate}
      \item If $C \subseteq X(10, a)$ and $C' \subseteq X(10, b)$ then $\mutilde(y, w) = 1$ for the edges shown as dotted gray lines in \autoref{fig:mainCase3Case1}, plus the seven edges obtained from those edges using \myrefT{thm:talbTilde}.
      For any other $y$ and $w$ as above, we have $\mutilde(y, w) = 0$.
      \item If $C \subseteq X(14, a, 1)$ and $C' \subseteq X(14, b, 2)$ then $\mutilde(y, w) = 1$ for the edges shown as dotted gray lines in \autoref{fig:mainCase3Case2}, plus the eleven edges obtained from those edges using \myrefT{thm:talbTilde}.
      For any other $y$ and $w$ as above, we have $\mutilde(y, w) = 0$.
      We have analogous statements for $C \subseteq X(14, a, j)$ and $C' \subseteq X(14, b, k)$ for $j, k \in \refSet$ with $j \neq k$.
      $C \subseteq X(14, a, 1)$ and $C' \subseteq X(14, b, 2)$
      \item If $C \subseteq X(10, a)$ and $C' \subseteq X(14, b, j)$ for $j \in \refSet$ then $\mutilde(y, w) = 1$ for any  $y$ and $w$ of the same type. (See  \autoref{fig:mainCase3Case3}.)
      \item If $C \subseteq X(14, a, j)$ and $C' \subseteq X(10, b)$ for $j \in \refSet$ then $\mutilde(y, w) = 1$ for any  $y$ and $w$ of the same type.
      \item For any pairs $C$ and $C'$ not listed in the previous cases, we have $\mutilde(y, w) = 0$.
    \end{enumerate}
  \end{proposition}

  \begin{figure}[!ht]
    \centering
      \begin{tikzpicture}[scale=0.4,  every node/.style={transform shape}]
        \node at (0 * \xSLT, 0 * \ySLT) {
          \begin{tikzpicture}
            \upLineT{darkblue}{solid}{0}{0}
            \upLineT{goldenpoppy}{dashed}{0}{1}
            \rightLineT{magenta}{dashed}{0}{1}
            \leftLineT{green}{dashed}{0}{1}
            \rightLineT{magenta}{solid}{-1}{2}
            \upLineT{goldenpoppy}{solid}{-1}{2}
            \leftLineT{green}{solid}{0}{2}
            \leftLineT{green}{solid}{1}{2}
            \upLineT{goldenpoppy}{solid}{1}{2}
            \rightLineT{magenta}{solid}{0}{2}
            \leftLineT{green}{dashed}{1}{3}
            \rightLineT{magenta}{dashed}{-1}{3}
            \upLineT{goldenpoppy}{dashed}{0}{3}
            \upLineT{darkblue}{solid}{0}{4}

            \node (CNode) at (0, 0) {\dynkinLabelWeylHuge{\TypeCPicture}{L}};
            \node (DNode) at (0, \ySLT) {\dynkinLabelWeylHuge{\TypeDPicture}{ }};
            \node (A2Node) at (-1 * \xSLT, 2 * \ySLT) {\dynkinLabelWeylHuge{\TypeAPicture{2}}{ }};
            \node (A4Node) at (0, 2 * \ySLT) {\dynkinLabelWeylHuge{\TypeAPicture{4}}{M}};
            \node (A1Node) at (1 * \xSLT, 2 * \ySLT) {\dynkinLabelWeylHuge{\TypeAPicture{1}}{ }};
            \node (B1Node) at (-1 * \xSLT, 3 * \ySLT) {\dynkinLabelWeylHuge{\TypeBPicture{1}}{ }};
            \node (B4Node) at (0, 3 * \ySLT) {\dynkinLabelWeylHuge{\TypeBPicture{4}}{ }};
            \node (B2Node) at (1 * \xSLT, 3 * \ySLT) {\dynkinLabelWeylHuge{\TypeBPicture{2}}{ }};
            \node (CPNode) at (0, 4 * \ySLT) {\dynkinLabelWeylHuge{\TypeCPicture}{U}};
            \node (DPNode) at (0, 5 * \ySLT) {\dynkinLabelWeylHuge{\TypeDPicture}{ }};

          \end{tikzpicture}
        };
        \node at (3 * \xSLT, 1 * \ySLT) {
          \begin{tikzpicture}
            \upLineT{darkblue}{solid}{0}{0}
            \upLineT{goldenpoppy}{dashed}{0}{1}
            \rightLineT{magenta}{dashed}{0}{1}
            \leftLineT{green}{dashed}{0}{1}
            \rightLineT{magenta}{solid}{-1}{2}
            \upLineT{goldenpoppy}{solid}{-1}{2}
            \leftLineT{green}{solid}{0}{2}
            \leftLineT{green}{solid}{1}{2}
            \upLineT{goldenpoppy}{solid}{1}{2}
            \rightLineT{magenta}{solid}{0}{2}
            \leftLineT{green}{dashed}{1}{3}
            \rightLineT{magenta}{dashed}{-1}{3}
            \upLineT{goldenpoppy}{dashed}{0}{3}
            \upLineT{darkblue}{solid}{0}{4}

            \node (CNode) at (0, 0) {\dynkinLabelWeylHuge{\TypeCPicture}{L'}};
            \node (DNode) at (0, \ySLT) {\dynkinLabelWeylHuge{\TypeDPicture}{ }};
            \node (A2Node) at (-1 * \xSLT, 2 * \ySLT) {\dynkinLabelWeylHuge{\TypeAPicture{2}}{ }};
            \node (A4Node) at (0, 2 * \ySLT) {\dynkinLabelWeylHuge{\TypeAPicture{4}}{M'}};
            \node (A1Node) at (1 * \xSLT, 2 * \ySLT) {\dynkinLabelWeylHuge{\TypeAPicture{1}}{ }};
            \node (B1Node) at (-1 * \xSLT, 3 * \ySLT) {\dynkinLabelWeylHuge{\TypeBPicture{1}}{ }};
            \node (B4Node) at (0, 3 * \ySLT) {\dynkinLabelWeylHuge{\TypeBPicture{4}}{ }};
            \node (B2Node) at (1 * \xSLT, 3 * \ySLT) {\dynkinLabelWeylHuge{\TypeBPicture{2}}{ }};
            \node (CPNode) at (0, 4 * \ySLT) {\dynkinLabelWeylHuge{\TypeCPicture}{U'}};
            \node (DPNode) at (0, 5 * \ySLT) {\dynkinLabelWeylHuge{\TypeDPicture}{ }};

          \end{tikzpicture}
        };

        \draw[-, gray, dotted, thick, shorten <=0.7cm, shorten >=0.7cm] (0 * \xSLT, -2.5 * \ySLT) -- (3 * \xSLT, -1.5 * \ySLT);

        \draw[-, gray, dotted, thick, shorten <=0.7cm, shorten >=0.7cm] (0 * \xSLT, -.5 * \ySLT) -- (3 * \xSLT, .5 * \ySLT);

        \draw[-, gray, dotted, thick, shorten <=0.7cm, shorten >=0.7cm] (0 * \xSLT, 1.5 * \ySLT) -- (3 * \xSLT, 2.5 * \ySLT);
      \end{tikzpicture}
    \caption{\myrefP{prop:mainCase3} Case 1}
    \label{fig:mainCase3Case1}
  \end{figure}

  \begin{figure}[!ht]
    \centering
    \begin{scaletikzpicturetowidth}{\textwidth}
      \begin{tikzpicture}[scale=\tikzscale, every node/.style={transform shape}]
        \node at (0 * \xSLT, 0 * \ySLT) {
          \begin{tikzpicture}
            \upLineF{darkblue}{solid}{0}{2}
            \upLineF{magenta}{dashed}{0}{3}
            \rightLineF{green}{dashed}{0}{3}
            \leftLineF{goldenpoppy}{dashed}{0}{3}
            \upLineF{magenta}{solid}{-1}{4}
            \leftLineF{goldenpoppy}{solid}{0}{4}
            \upLineF{magenta}{solid}{1}{4}
            \rightLineF{green}{solid}{0}{4}
            \rightLineF{darkblue}{solid}{-2}{3}
            \leftLineF{darkblue}{solid}{2}{3}
            \upLineF{goldenpoppy}{solid}{-2}{2}
            \upLineF{green}{solid}{2}{2}
            \leftLineF{darkblue}{solid}{-1}{1}
            \rightLineF{goldenpoppy}{dashed}{-1}{1}
            \leftLineF{green}{dashed}{1}{1}
            \rightLineF{darkblue}{solid}{1}{1}
            \leftLineF{green}{solid}{0}{0}
            \rightLineF{goldenpoppy}{solid}{0}{0}

            \node (A1Node) at (0,0) {\dynkinLabelWeylHuge{\TypeAPicture{1}}{L}};
            \node (B4Node) at (-1 * \xSLF, 1 * \ySLF) {\dynkinLabelWeylHuge{\TypeBPicture{4}}{ }};
            \node (B2Node) at (1 * \xSLF, 1 * \ySLF) {\dynkinLabelWeylHuge{\TypeBPicture{2}}{ }};
            \node (A2Node) at (-2 * \xSLF, 2 * \ySLF) {\dynkinLabelWeylHuge{\TypeAPicture{2}}{ }};
            \node (CNode) at (0, 2 * \ySLF) {\dynkinLabelWeylHuge{\TypeCPicture}{M}};
            \node (A4Node) at (2 * \xSLF, 2 * \ySLF) {\dynkinLabelWeylHuge{\TypeAPicture{4}}{ }};
            \node (B1Node) at (-2 * \xSLF, 3 * \ySLF) {\dynkinLabelWeylHuge{\TypeBPicture{1}}{ }};
            \node (DNode) at (0, 3 * \ySLF) {\dynkinLabelWeylHuge{\TypeDPicture}{ }};
            \node (B1PNode) at (2 * \xSLF, 3 * \ySLF) {\dynkinLabelWeylHuge{\TypeBPicture{1}}{ }};
            \node (A4PNode) at (-1 * \xSLF, 4 * \ySLF) {\dynkinLabelWeylHuge{\TypeAPicture{4}}{ }};
            \node (A1PNode) at (0, 4 * \ySLF) {\dynkinLabelWeylHuge{\TypeAPicture{1}}{U}};
            \node (A2PNode) at (1 * \xSLF, 4 * \ySLF) {\dynkinLabelWeylHuge{\TypeAPicture{2}}{ }};
            \node (B2PNode) at (-1 * \xSLF, 5 * \ySLF) {\dynkinLabelWeylHuge{\TypeBPicture{2}}{ }};
            \node (B4PNode) at (1 * \xSLF, 5 * \ySLF) {\dynkinLabelWeylHuge{\TypeBPicture{4}}{ }};

          \end{tikzpicture}
        };
        \node at (5 * \xSLT, 1 * \ySLT) {
          \begin{tikzpicture}
            \upLineF{darkblue}{solid}{0}{2}
            \rightLineF{magenta}{dashed}{0}{3}
            \leftLineF{goldenpoppy}{dashed}{0}{3}
            \rightLineF{magenta}{solid}{-1}{4}
            \leftLineF{goldenpoppy}{solid}{1}{4}
            \upLineF{green}{dashed}{0}{1}
            \rightLineF{goldenpoppy}{solid}{-1}{0}
            \leftLineF{magenta}{solid}{1}{0}
            \leftLineF{magenta}{dashed}{1}{1}
            \rightLineF{goldenpoppy}{dashed}{-1}{1}
            \upLineF{green}{solid}{-1}{0}
            \upLineF{green}{solid}{1}{0}
            \leftLineF{darkblue}{solid}{-1}{1}
            \upLineF{goldenpoppy}{solid}{-2}{2}
            \rightLineF{darkblue}{solid}{-2}{3}
            \rightLineF{darkblue}{solid}{1}{1}
            \upLineF{magenta}{solid}{2}{2}
            \leftLineF{darkblue}{solid}{2}{3}

            \node (A1Node) at (-1 * \xSLF, 0 * \ySLF)
            {\dynkinLabelWeylHuge{\TypeAPicture{1}}{L'}};
            \node (A4Node) at (1 * \xSLF, 0 * \ySLF)
            {\dynkinLabelWeylHuge{\TypeAPicture{4}}{ }};
            \node (B4Node) at (-1 * \xSLF, 1 * \ySLF) {\dynkinLabelWeylHuge{\TypeBPicture{4}}{ }};
            \node (B2Node) at (0 * \xSLF, 1 * \ySLF) {\dynkinLabelWeylHuge{\TypeBPicture{2}}{ }};
            \node (B1Node) at (1 * \xSLF, 1 * \ySLF) {\dynkinLabelWeylHuge{\TypeBPicture{1}}{ }};
            \node (A2Node) at (-2 * \xSLF, 2 * \ySLF) {\dynkinLabelWeylHuge{\TypeAPicture{2}}{ }};
            \node (CNode) at (0 * \xSLF, 2 * \ySLF) {\dynkinLabelWeylHuge{\TypeCPicture}{M'}};
            \node (A2PNode) at (2 * \xSLF, 2 * \ySLF) {\dynkinLabelWeylHuge{\TypeAPicture{2}}{ }};
            \node (B1PNode) at (-2 * \xSLF, 3 * \ySLF) {\dynkinLabelWeylHuge{\TypeBPicture{1}}{ }};
            \node (DNode)  at (0 * \xSLF, 3 * \ySLF) {\dynkinLabelWeylHuge{\TypeDPicture}{ }};
            \node (B4PNode)  at (2 * \xSLF, 3 * \ySLF) {\dynkinLabelWeylHuge{\TypeBPicture{4}}{ }};
            \node (A4PNode) at (-1 * \xSLF, 4 * \ySLF) {\dynkinLabelWeylHuge{\TypeAPicture{4}}{ }};
            \node (A1PNode) at (1 * \xSLF, 4 * \ySLF) {\dynkinLabelWeylHuge{\TypeAPicture{1}}{U'}};
            \node (B2PNode) at (0 * \xSLF, 5 * \ySLF) {\dynkinLabelWeylHuge{\TypeBPicture{2}}{ }};

          \end{tikzpicture}
        };

        \draw[-, gray, dotted, thick, shorten <=0.5cm, shorten >=0.7cm] (0 * \xSLT, -2.5 * \ySLT) -- (4 * \xSLT, -1.5 * \ySLT);

        \draw[-, gray, dotted, thick, shorten <=0.5cm, shorten >=0.7cm] (0 * \xSLT, -.5 * \ySLT) -- (5 * \xSLT, .5 * \ySLT);

        \draw[-, gray, dotted, thick, shorten <=0.5cm, shorten >=0.7cm] (0 * \xSLT, 1.5 * \ySLT) -- (6 * \xSLT, 2.5 * \ySLT);
      \end{tikzpicture}
    \end{scaletikzpicturetowidth}
    \caption{\myrefP{prop:mainCase3} Case 2}
    \label{fig:mainCase3Case2}
  \end{figure}

  \begin{figure}
    \centering
      \begin{tikzpicture}[scale=0.4,  every node/.style={transform shape}]
        \node at (0 * \xSLT, 0 * \ySLT) {
          \begin{tikzpicture}
            \upLineT{darkblue}{solid}{0}{0}
            \upLineT{goldenpoppy}{dashed}{0}{1}
            \rightLineT{magenta}{dashed}{0}{1}
            \leftLineT{green}{dashed}{0}{1}
            \rightLineT{magenta}{solid}{-1}{2}
            \upLineT{goldenpoppy}{solid}{-1}{2}
            \leftLineT{green}{solid}{0}{2}
            \leftLineT{green}{solid}{1}{2}
            \upLineT{goldenpoppy}{solid}{1}{2}
            \rightLineT{magenta}{solid}{0}{2}
            \leftLineT{green}{dashed}{1}{3}
            \rightLineT{magenta}{dashed}{-1}{3}
            \upLineT{goldenpoppy}{dashed}{0}{3}
            \upLineT{darkblue}{solid}{0}{4}

            \node (CNode) at (0, 0) {\dynkinLabelWeylHuge{\TypeCPicture}{L}};
            \node (DNode) at (0, \ySLT) {\dynkinLabelWeylHuge{\TypeDPicture}{ }};
            \node (A2Node) at (-1 * \xSLT, 2 * \ySLT) {\dynkinLabelWeylHuge{\TypeAPicture{2}}{ }};
            \node (A4Node) at (0, 2 * \ySLT) {\dynkinLabelWeylHuge{\TypeAPicture{4}}{ }};
            \node (A1Node) at (1 * \xSLT, 2 * \ySLT) {\dynkinLabelWeylHuge{\TypeAPicture{1}}{M}};
            \node (B1Node) at (-1 * \xSLT, 3 * \ySLT) {\dynkinLabelWeylHuge{\TypeBPicture{1}}{ }};
            \node (B4Node) at (0, 3 * \ySLT) {\dynkinLabelWeylHuge{\TypeBPicture{4}}{ }};
            \node (B2Node) at (1 * \xSLT, 3 * \ySLT) {\dynkinLabelWeylHuge{\TypeBPicture{2}}{ }};
            \node (CPNode) at (0, 4 * \ySLT) {\dynkinLabelWeylHuge{\TypeCPicture}{U}};
            \node (DPNode) at (0, 5 * \ySLT) {\dynkinLabelWeylHuge{\TypeDPicture}{ }};

          \end{tikzpicture}
        };
        \node at (4 * \xSLT, 1 * \ySLT) {
          \begin{tikzpicture}
            \upLineF{darkblue}{solid}{0}{2}
            \rightLineF{magenta}{dashed}{0}{3}
            \leftLineF{goldenpoppy}{dashed}{0}{3}
            \rightLineF{magenta}{solid}{-1}{4}
            \leftLineF{goldenpoppy}{solid}{1}{4}
            \upLineF{green}{dashed}{0}{1}
            \rightLineF{goldenpoppy}{solid}{-1}{0}
            \leftLineF{magenta}{solid}{1}{0}
            \leftLineF{magenta}{dashed}{1}{1}
            \rightLineF{goldenpoppy}{dashed}{-1}{1}
            \upLineF{green}{solid}{-1}{0}
            \upLineF{green}{solid}{1}{0}
            \leftLineF{darkblue}{solid}{-1}{1}
            \upLineF{goldenpoppy}{solid}{-2}{2}
            \rightLineF{darkblue}{solid}{-2}{3}
            \rightLineF{darkblue}{solid}{1}{1}
            \upLineF{magenta}{solid}{2}{2}
            \leftLineF{darkblue}{solid}{2}{3}

            \node (A1Node) at (-1 * \xSLF, 0 * \ySLF)
            {\dynkinLabelWeylHuge{\TypeAPicture{1}}{L'}};
            \node (A4Node) at (1 * \xSLF, 0 * \ySLF)
            {\dynkinLabelWeylHuge{\TypeAPicture{4}}{ }};
            \node (B4Node) at (-1 * \xSLF, 1 * \ySLF) {\dynkinLabelWeylHuge{\TypeBPicture{4}}{ }};
            \node (B2Node) at (0 * \xSLF, 1 * \ySLF) {\dynkinLabelWeylHuge{\TypeBPicture{2}}{ }};
            \node (B1Node) at (1 * \xSLF, 1 * \ySLF) {\dynkinLabelWeylHuge{\TypeBPicture{1}}{ }};
            \node (A2Node) at (-2 * \xSLF, 2 * \ySLF) {\dynkinLabelWeylHuge{\TypeAPicture{2}}{ }};
            \node (CNode) at (0 * \xSLF, 2 * \ySLF) {\dynkinLabelWeylHuge{\TypeCPicture}{M'}};
            \node (A2PNode) at (2 * \xSLF, 2 * \ySLF) {\dynkinLabelWeylHuge{\TypeAPicture{2}}{ }};
            \node (B1PNode) at (-2 * \xSLF, 3 * \ySLF) {\dynkinLabelWeylHuge{\TypeBPicture{1}}{ }};
            \node (DNode)  at (0 * \xSLF, 3 * \ySLF) {\dynkinLabelWeylHuge{\TypeDPicture}{ }};
            \node (B4PNode)  at (2 * \xSLF, 3 * \ySLF) {\dynkinLabelWeylHuge{\TypeBPicture{4}}{ }};
            \node (A4PNode) at (-1 * \xSLF, 4 * \ySLF) {\dynkinLabelWeylHuge{\TypeAPicture{4}}{ }};
            \node (A1PNode) at (1 * \xSLF, 4 * \ySLF) {\dynkinLabelWeylHuge{\TypeAPicture{1}}{U'}};
            \node (B2PNode) at (0 * \xSLF, 5 * \ySLF) {\dynkinLabelWeylHuge{\TypeBPicture{2}}{ }};

          \end{tikzpicture}
        };

        \draw[-, gray, dotted, thick, shorten <=0.7cm, shorten >=0.7cm] (0 * \xSLT, -2.5 * \ySLT) -- (4 * \xSLT, .5 * \ySLT);

        \draw[-, gray, dotted, thick, shorten <=0.7cm, shorten >=0.7cm] (1 * \xSLT, -.5 * \ySLT) -- (3 * \xSLT, -1.5 * \ySLT);

        \draw[-, gray, dotted, thick, shorten <=0.7cm, shorten >=0.5cm] (1 * \xSLT, -.5 * \ySLT) -- (5 * \xSLT, 2.5 * \ySLT);

        \draw[-, gray, dotted, thick, shorten <=0.7cm, shorten >=0.7cm] (0 * \xSLT, 1.5 * \ySLT) -- (4 * \xSLT, .5 * \ySLT);
      \end{tikzpicture}
    \caption{\myrefP{prop:mainCase3} Case 3}
    \label{fig:mainCase3Case3}
  \end{figure}

  \begin{proof}
    We'll prove that the theorem holds for $\WDFour$. Given that, the proof of the theorem in general follows directly from \myrefP{prop:KLPolysParabolic}.

    Looking at the cells in $\WDFour$, we have many more cases than for parabolic subgroups of type $A_2$ or $B_2$, but in fact the situation, in large, is the same.
    Elements in different left cells are connected by edges when they are in the same right cell.

    Let's start by looking at $C(10, a)$.
    I claim that all elements of $C(10, a)$ are of right type $\scC$.
    To see that, note that the bottom element of that cell, $s_1s_2s_4$, is an involution, and thus has the same left type and right type.
    Now, we can appeal to \myrefP{prop:D4OpRightCells}, which says that elements which are in the same left cell have the same right type (if any).
    Similarly, elements of $C(10, b)$ are of right type $\scD$, elements of $C(14, a, j)$ are of right type $\scA_j$, and elements of $C(14, b, j)$ are of right type $\scB_j$.

    Turning to \autoref{fig:mainCase3Case1}, we see that the bottom and top dotted lines correspond to the two edges connecting elements of type $\scC$ to elements of type $\scD$ in \autoref{fig:type10Full}, if we consider \autoref{fig:type10Full} as showing connections on the right instead of on the left.

    Similarly, the middle dotted line in \autoref{fig:mainCase3Case1} corresponds to the line connecting the element of type $\scC$ to the element of type $\scD$ in \autoref{fig:type14aFull}.

    The other cases are analogous.
  \end{proof}

  \begin{remark}
    We can also use the example shown in \autoref{fig:mainCase3Case2} to illustrate the three cases of the first edge transport theorem, \myrefT{thm:stStringsA}.
    We'll be transporting the edges using Knuth maps on the right.
    Starting with the left cell $C(14, a, 1)$, we see from \autoref{fig:type14aFull} that there are 25 edges connecting elements of this left cell.
    Though the elements of the cell have different left \ti s,
    as per \myrefP{prop:tauCell}, all the elements of this cell have the same right \ti, namely $\{s_1, s_3\}$.
    In particular, $C(14, a, 1) \subseteq D^R_{s_4, s_3}(W)$.
    The map $T^R_{s_4, s_3}$ takes the left cell $C(14, a, 1)$ to the left cell $C(14, b, 2)$.
    So, it will transport each of the 25 edges connecting two elements of $C(14, a, 1)$ to an edge connecting the corresponding two elements of $C(14, b, 2)$.
    Most of those transports fall under case 1 of \myrefT{thm:stStringsA}.
    However, there are two instances of case 2 and four instances of case 3 of the theorem.
    We'll show an example of each.
    Refer to \autoref{fig:cases13StStringsA} and \autoref{fig:case2StStringsA}.
    In those figures, elements are labeled $L, U, L', U'$ as in \myrefT{thm:stStringsA}, but with a subscript 1, 2, or 3, to show which case they belong in.
    Lines labeled $s_3$ or $s_4$ refer to multiplication on the right by that element.
    \begin{figure}[!ht]
      \centering
      \begin{scaletikzpicturetowidth}{\textwidth}
        \begin{tikzpicture}[scale=\tikzscale, every node/.style={transform shape}]
          \node at (0 * \xSLT, 0 * \ySLT) {
            \begin{tikzpicture}
              \upLineF{darkblue}{solid}{0}{2}
              \upLineF{magenta}{dashed}{0}{3}
              \rightLineF{green}{dashed}{0}{3}
              \leftLineF{goldenpoppy}{dashed}{0}{3}
              \upLineF{magenta}{solid}{-1}{4}
              \leftLineF{goldenpoppy}{solid}{0}{4}
              \upLineF{magenta}{solid}{1}{4}
              \rightLineF{green}{solid}{0}{4}
              \rightLineF{darkblue}{solid}{-2}{3}
              \leftLineF{darkblue}{solid}{2}{3}
              \upLineF{goldenpoppy}{solid}{-2}{2}
              \upLineF{green}{solid}{2}{2}
              \leftLineF{darkblue}{solid}{-1}{1}
              \rightLineF{goldenpoppy}{dashed}{-1}{1}
              \leftLineF{green}{dashed}{1}{1}
              \rightLineF{darkblue}{solid}{1}{1}
              \leftLineF{green}{solid}{0}{0}
              \rightLineF{goldenpoppy}{solid}{0}{0}

              \node (A1Node) at (0,0) {\dynkinLabelWeylHuge{\TypeAPicture{1}}{}};
              \node (B4Node) at (-1 * \xSLF, 1 * \ySLF) {\dynkinLabelWeylHuge{\TypeBPicture{4}}{ }};
              \node (B2Node) at (1 * \xSLF, 1 * \ySLF) {\dynkinLabelWeylHuge{\TypeBPicture{2}}{L_3}};
              \node (A2Node) at (-2 * \xSLF, 2 * \ySLF) {\dynkinLabelWeylHuge{\TypeAPicture{2}}{ }};
              \node (CNode) at (0, 2 * \ySLF) {\dynkinLabelWeylHuge{\TypeCPicture}{}};
              \node (A4Node) at (2 * \xSLF, 2 * \ySLF) {\dynkinLabelWeylHuge{\TypeAPicture{4}}{U'_3}};
              \node (B1Node) at (-2 * \xSLF, 3 * \ySLF) {\dynkinLabelWeylHuge{\TypeBPicture{1}}{ }};
              \node (DNode) at (0, 3 * \ySLF) {\dynkinLabelWeylHuge{\TypeDPicture}{}};
              \node (B1PNode) at (2 * \xSLF, 3 * \ySLF) {\dynkinLabelWeylHuge{\TypeBPicture{1}}{ }};
              \node (A4PNode) at (-1 * \xSLF, 4 * \ySLF) {\dynkinLabelWeylHuge{\TypeAPicture{4}}{L_1}};
              \node (A1PNode) at (0, 4 * \ySLF) {\dynkinLabelWeylHuge{\TypeAPicture{1}}{}};
              \node (A2PNode) at (1 * \xSLF, 4 * \ySLF) {\dynkinLabelWeylHuge{\TypeAPicture{2}}{ }};
              \node (B2PNode) at (-1 * \xSLF, 5 * \ySLF) {\dynkinLabelWeylHuge{\TypeBPicture{2}}{L_1'}};
              \node (B4PNode) at (1 * \xSLF, 5 * \ySLF) {\dynkinLabelWeylHuge{\TypeBPicture{4}}{}};

            \end{tikzpicture}
          };
          \node at (5 * \xSLT, 1 * \ySLT) {
            \begin{tikzpicture}
              \upLineF{darkblue}{solid}{0}{2}
              \rightLineF{magenta}{dashed}{0}{3}
              \leftLineF{goldenpoppy}{dashed}{0}{3}
              \rightLineF{magenta}{solid}{-1}{4}
              \leftLineF{goldenpoppy}{solid}{1}{4}
              \upLineF{green}{dashed}{0}{1}
              \rightLineF{goldenpoppy}{solid}{-1}{0}
              \leftLineF{magenta}{solid}{1}{0}
              \leftLineF{magenta}{dashed}{1}{1}
              \rightLineF{goldenpoppy}{dashed}{-1}{1}
              \upLineF{green}{solid}{-1}{0}
              \upLineF{green}{solid}{1}{0}
              \leftLineF{darkblue}{solid}{-1}{1}
              \upLineF{goldenpoppy}{solid}{-2}{2}
              \rightLineF{darkblue}{solid}{-2}{3}
              \rightLineF{darkblue}{solid}{1}{1}
              \upLineF{magenta}{solid}{2}{2}
              \leftLineF{darkblue}{solid}{2}{3}

              \node (A1Node) at (-1 * \xSLF, 0 * \ySLF)
              {\dynkinLabelWeylHuge{\TypeAPicture{1}}{}};
              \node (A4Node) at (1 * \xSLF, 0 * \ySLF)
              {\dynkinLabelWeylHuge{\TypeAPicture{4}}{L'_3}};
              \node (B4Node) at (-1 * \xSLF, 1 * \ySLF) {\dynkinLabelWeylHuge{\TypeBPicture{4}}{ }};
              \node (B2Node) at (0 * \xSLF, 1 * \ySLF) {\dynkinLabelWeylHuge{\TypeBPicture{2}}{U_3}};
              \node (B1Node) at (1 * \xSLF, 1 * \ySLF) {\dynkinLabelWeylHuge{\TypeBPicture{1}}{ }};
              \node (A2Node) at (-2 * \xSLF, 2 * \ySLF) {\dynkinLabelWeylHuge{\TypeAPicture{2}}{ }};
              \node (CNode) at (0 * \xSLF, 2 * \ySLF) {\dynkinLabelWeylHuge{\TypeCPicture}{}};
              \node (A2PNode) at (2 * \xSLF, 2 * \ySLF) {\dynkinLabelWeylHuge{\TypeAPicture{2}}{ }};
              \node (B1PNode) at (-2 * \xSLF, 3 * \ySLF) {\dynkinLabelWeylHuge{\TypeBPicture{1}}{ }};
              \node (DNode)  at (0 * \xSLF, 3 * \ySLF) {\dynkinLabelWeylHuge{\TypeDPicture}{}};
              \node (B4PNode)  at (2 * \xSLF, 3 * \ySLF) {\dynkinLabelWeylHuge{\TypeBPicture{4}}{}};
              \node (A4PNode) at (-1 * \xSLF, 4 * \ySLF) {\dynkinLabelWeylHuge{\TypeAPicture{4}}{U_1}};
              \node (A1PNode) at (1 * \xSLF, 4 * \ySLF) {\dynkinLabelWeylHuge{\TypeAPicture{1}}{}};
              \node (B2PNode) at (0 * \xSLF, 5 * \ySLF) {\dynkinLabelWeylHuge{\TypeBPicture{2}}{U_1'}};

            \end{tikzpicture}
          };

          \draw[-, brown, dotted, thick, shorten <=0.4cm, shorten >=0.7cm] (2 * \xSLT, -.5 * \ySLT) --node[right, black]{\LARGE{$s_3$}} (6 * \xSLT, -1.5 * \ySLT);

          \draw[-, brown, dotted, thick, shorten <=0.5cm, shorten >=0.7cm] (1 * \xSLT, -1.5 * \ySLT) -- node[below, black]{\LARGE{$s_4$}}(5 * \xSLT, -.5 * \ySLT);

          \draw[-, brown, dotted, thick, shorten <=0.5cm, shorten >=0.7cm] (-1 * \xSLT, 1.5 * \ySLT) --node[above, black]{\LARGE{$s_4$}} (4 * \xSLT, 2.5 * \ySLT);

          \draw[-, brown, dotted, thick, shorten <=0.5cm, shorten >=0.7cm] (-1 * \xSLT, 2.5 * \ySLT) --node[above, black]{\LARGE{$s_4$}} (5 * \xSLT, 3.5 * \ySLT);
        \end{tikzpicture}
      \end{scaletikzpicturetowidth}
      \caption{Cases 1 and 3 of  \myrefT{thm:stStringsA}}
      \label{fig:cases13StStringsA}
    \end{figure}

    \begin{figure}[!ht]
      \centering
      \begin{scaletikzpicturetowidth}{\textwidth}
        \begin{tikzpicture}[scale=\tikzscale, every node/.style={transform shape}]
          \node at (0 * \xSLT, 0 * \ySLT) {
            \begin{tikzpicture}
              \upLineF{darkblue}{solid}{0}{2}
              \upLineF{magenta}{dashed}{0}{3}
              \rightLineF{green}{dashed}{0}{3}
              \leftLineF{goldenpoppy}{dashed}{0}{3}
              \upLineF{magenta}{solid}{-1}{4}
              \leftLineF{goldenpoppy}{solid}{0}{4}
              \upLineF{magenta}{solid}{1}{4}
              \rightLineF{green}{solid}{0}{4}
              \rightLineF{darkblue}{solid}{-2}{3}
              \leftLineF{darkblue}{solid}{2}{3}
              \upLineF{goldenpoppy}{solid}{-2}{2}
              \upLineF{green}{solid}{2}{2}
              \leftLineF{darkblue}{solid}{-1}{1}
              \rightLineF{goldenpoppy}{dashed}{-1}{1}
              \leftLineF{green}{dashed}{1}{1}
              \rightLineF{darkblue}{solid}{1}{1}
              \leftLineF{green}{solid}{0}{0}
              \rightLineF{goldenpoppy}{solid}{0}{0}

              \node (A1Node) at (0,0) {\dynkinLabelWeylHuge{\TypeAPicture{1}}{}};
              \node (B4Node) at (-1 * \xSLF, 1 * \ySLF) {\dynkinLabelWeylHuge{\TypeBPicture{4}}{ }};
              \node (B2Node) at (1 * \xSLF, 1 * \ySLF) {\dynkinLabelWeylHuge{\TypeBPicture{2}}{ }};
              \node (A2Node) at (-2 * \xSLF, 2 * \ySLF) {\dynkinLabelWeylHuge{\TypeAPicture{2}}{ }};
              \node (CNode) at (0, 2 * \ySLF) {\dynkinLabelWeylHuge{\TypeCPicture}{L_2}};
              \node (A4Node) at (2 * \xSLF, 2 * \ySLF) {\dynkinLabelWeylHuge{\TypeAPicture{4}}{}};
              \node (B1Node) at (-2 * \xSLF, 3 * \ySLF) {\dynkinLabelWeylHuge{\TypeBPicture{1}}{ }};
              \node (DNode) at (0, 3 * \ySLF) {\dynkinLabelWeylHuge{\TypeDPicture}{}};
              \node (B1PNode) at (2 * \xSLF, 3 * \ySLF) {\dynkinLabelWeylHuge{\TypeBPicture{1}}{ }};
              \node (A4PNode) at (-1 * \xSLF, 4 * \ySLF) {\dynkinLabelWeylHuge{\TypeAPicture{4}}{ }};
              \node (A1PNode) at (0, 4 * \ySLF) {\dynkinLabelWeylHuge{\TypeAPicture{1}}{}};
              \node (A2PNode) at (1 * \xSLF, 4 * \ySLF) {\dynkinLabelWeylHuge{\TypeAPicture{2}}{ }};
              \node (B2PNode) at (-1 * \xSLF, 5 * \ySLF) {\dynkinLabelWeylHuge{\TypeBPicture{2}}{ }};
              \node (B4PNode) at (1 * \xSLF, 5 * \ySLF) {\dynkinLabelWeylHuge{\TypeBPicture{4}}{U'_2}};

              \draw[bend right = 54, gray, dashed] (CNode) to (B4PNode);
            \end{tikzpicture}
          };
          \node at (5 * \xSLT, 1 * \ySLT) {
            \begin{tikzpicture}
              \upLineF{darkblue}{solid}{0}{2}
              \rightLineF{magenta}{dashed}{0}{3}
              \leftLineF{goldenpoppy}{dashed}{0}{3}
              \rightLineF{magenta}{solid}{-1}{4}
              \leftLineF{goldenpoppy}{solid}{1}{4}
              \upLineF{green}{dashed}{0}{1}
              \rightLineF{goldenpoppy}{solid}{-1}{0}
              \leftLineF{magenta}{solid}{1}{0}
              \leftLineF{magenta}{dashed}{1}{1}
              \rightLineF{goldenpoppy}{dashed}{-1}{1}
              \upLineF{green}{solid}{-1}{0}
              \upLineF{green}{solid}{1}{0}
              \leftLineF{darkblue}{solid}{-1}{1}
              \upLineF{goldenpoppy}{solid}{-2}{2}
              \rightLineF{darkblue}{solid}{-2}{3}
              \rightLineF{darkblue}{solid}{1}{1}
              \upLineF{magenta}{solid}{2}{2}
              \leftLineF{darkblue}{solid}{2}{3}

              \node (A1Node) at (-1 * \xSLF, 0 * \ySLF)
              {\dynkinLabelWeylHuge{\TypeAPicture{1}}{}};
              \node (A4Node) at (1 * \xSLF, 0 * \ySLF)
              {\dynkinLabelWeylHuge{\TypeAPicture{4}}{}};
              \node (B4Node) at (-1 * \xSLF, 1 * \ySLF) {\dynkinLabelWeylHuge{\TypeBPicture{4}}{ }};
              \node (B2Node) at (0 * \xSLF, 1 * \ySLF) {\dynkinLabelWeylHuge{\TypeBPicture{2}}{ }};
              \node (B1Node) at (1 * \xSLF, 1 * \ySLF) {\dynkinLabelWeylHuge{\TypeBPicture{1}}{ }};
              \node (A2Node) at (-2 * \xSLF, 2 * \ySLF) {\dynkinLabelWeylHuge{\TypeAPicture{2}}{ }};
              \node (CNode) at (0 * \xSLF, 2 * \ySLF) {\dynkinLabelWeylHuge{\TypeCPicture}{U_2}};
              \node (A2PNode) at (2 * \xSLF, 2 * \ySLF) {\dynkinLabelWeylHuge{\TypeAPicture{2}}{ }};
              \node (B1PNode) at (-2 * \xSLF, 3 * \ySLF) {\dynkinLabelWeylHuge{\TypeBPicture{1}}{ }};
              \node (DNode)  at (0 * \xSLF, 3 * \ySLF) {\dynkinLabelWeylHuge{\TypeDPicture}{}};
              \node (B4PNode)  at (2 * \xSLF, 3 * \ySLF) {\dynkinLabelWeylHuge{\TypeBPicture{4}}{L'_2}};
              \node (A4PNode) at (-1 * \xSLF, 4 * \ySLF) {\dynkinLabelWeylHuge{\TypeAPicture{4}}{ }};
              \node (A1PNode) at (1 * \xSLF, 4 * \ySLF) {\dynkinLabelWeylHuge{\TypeAPicture{1}}{}};
              \node (B2PNode) at (0 * \xSLF, 5 * \ySLF) {\dynkinLabelWeylHuge{\TypeBPicture{2}}{ }};

              \draw[-, thick, gray, dashed] (CNode.east) -- (B4PNode.west);
            \end{tikzpicture}
          };

          \draw[-, brown, dotted, thick, shorten <=0.5cm, shorten >=0.7cm] (1 * \xSLT, 2.5 * \ySLT) --node[above, black]{\LARGE{$s_3$}} (7 * \xSLT, 1.5 * \ySLT);

          \draw[-, brown, dotted, thick, shorten <=0.5cm, shorten >=0.7cm] (0 * \xSLT, -.5 * \ySLT) -- node[above, black]{\LARGE{$s_4$}}(5 * \xSLT, .5 * \ySLT);

        \end{tikzpicture}
      \end{scaletikzpicturetowidth}
      \caption{Case 2 of \myrefT{thm:stStringsA}}
      \label{fig:case2StStringsA}
    \end{figure}
  \end{remark}

  \section{Edge Transport Functions, Part 1}
  \label{sec:edgeTransportFunctions}
  The edge transport theorems, \myrefT{thm:stStrings}, \myrefT{thm:stsStrings}, and \myrefT{thm:main}, are associated with maps.
  In the case of \myrefT{thm:stStrings} and \myrefT{thm:stsStrings} we have already defined the maps.
  In the case of \myrefT{thm:main} we will define the maps in \myrefS{sec:D4Maps}.

  These maps all have additional properties, which, in conjunction with the edge transport theorems, will allow us to define the generalized \ti\ using them, and prove that it is a weaker equivalence relation than that of being in the same left cell.

  We'll describe these properties next.
  The first one is a property of the domain of the functions.

  \begin{definition}
    \label{def:KLIntervalSet}
    Let $D \subset W$.
    We say $D$ is a left KL interval set if, whenever $x, y \in D$ and $w \in W$ with $x \leqL w \leqL y$, then $w \in D$.
    We define similarly right KL interval set.
  \end{definition}

  \begin{proposition}
    \label{prop:KLIntervalEquiv}
    If $D \subset W$ is a left (resp.\ right) KL interval set and $x, y \in W$ with $x \equivL y$ (resp.\ $x \equivR y$) then $x \in D$ if and only if $y \in D$.
  \end{proposition}

  \begin{proof}
    This is clear.
  \end{proof}

  The second property concerns the image of the function.
  We'll define it first for functions such as the Knuth maps.

  \begin{definition}
    \label{def:tauPreserving}
    A function $T: D \longrightarrow W$ with $D \subset W$ is left \ti\ preserving (or left descent set preserving) if $\tau_L(T(w)) = \tau_L(w)$.
    We define similarly right \ti\ preserving.
  \end{definition}

  For use with the generalized \ti, the above property is all we need.
  Recall, however, that we also want to use the same maps to define an equivalence relation which is stronger than that of being in the same (left or right) cell.
  For that we'll need the following definition.

  \begin{definition}
    \label{def:KLCellFunction}
    A function $T: D \longrightarrow W$ with $D \subset W$ is a left KL cell function if $T(w) \equivL w$ for all $w \in D$.
    We define similarly right KL cell function.
  \end{definition}

  \begin{remarkNumbered}
    \label{rem:KLcellTauPreserving}
    A left (resp.\ right) KL cell function is right (resp.\ left) \ti\ preserving by \myrefP{prop:tauCell}.
  \end{remarkNumbered}

  Now, let's see that the Knuth maps have these properties.

  \begin{proposition}
    \label{prop:KnuthAKLSet}
    Suppose $s, t \in S$ with $st$ of order 3.
    Then $\DstL$ is a right KL interval set, and $\DstR$ is a left KL interval set.
  \end{proposition}

  \begin{proof}
  This follows from \myrefP{prop:tauInterval}.
  \end{proof}

  \begin{proposition}
    \label{prop:KnuthAKLCellFunction}
    Suppose $s, t \in S$ with $st$ of order 3.
    Then $\TstL$ is a left KL cell function and
    $\TstR$ is a right KL cell function.
  \end{proposition}

  \begin{proof}
    This is clear from the definitions.
    That is, if $x \in \DstL$ and $y = \TstL(x)$, then either $y = sx$ or $y = tx$, so in either case $\mutilde(x, y) = 1$.
    Also, since $y \in D^L_{t,s}(W)$, we have $\tauL(x) \not\subset \tauL(y)$ and $\tauL(y) \not\subset \tauL(x)$.
  \end{proof}

  Finally, let's encapsulate the edge transport theorem in a definition which we can apply to Knuth maps immediately, and then to other families of maps in \myrefS{sec:otherMaps}.

  \begin{definition}
    \label{def:edgeTransportFunction}
    Let $T: D \longrightarrow W$, where $D \subset W$.
    The function $T$ is called an edge transport function if it is an injection and if $\mutilde(T(x), T(y)) = \mutilde(x, y)$ for all $x, y \in D$.
  \end{definition}

  \begin{proposition}
    \label{prop:KnuthAEdgeTransportFunction}
    Let $s, t \in S$ with $st$ of order 3.
    Then $\TstL$ (resp.\ $\TstR$) is an edge transport function.
  \end{proposition}

  \begin{proof}
    That $\TstL$ is an edge transport function is \myrefT{thm:talbTilde}.
  \end{proof}

  To use these functions with the generalized \ti, we'll need this property.

  \begin{definition}
    A function $T: D \longrightarrow W$ with $D \subset W$ is left KL order preserving if for $x, y \in D$ with $x \leqL y$ we have $T(x) \leqL T(y)$.
    We define similarly right  KL order preserving.
  \end{definition}

  In the following proposition and corollary, we reproduce the argument of Corollary 4.3 and part of Section 5 of \citeKL, in our more general context.

  \begin{proposition}
    \label{prop:edgeTransportKLOrderA}
    Let $T: D \longrightarrow W$ be an edge transport function.
    Assume in addition that $D$ is a right KL interval set and that $T$ is right \ti\ preserving.
    Then $T$ is right KL order preserving.
    Similarly, with left and right interchanged.
  \end{proposition}

  \begin{proof}
    Let $x, y \in D$ with $x \leqR y$.
    Then there is a sequence $w_1,\dots,w_n$ of elements of $W$
    with $w_1 = x$ and $w_n = y$  such that $\mutilde(w_i,w_{i+1}) > 0$ and $\tau_R(w_i) \not\subset \tau_R(w_{i+1})$ for $1 \le i \le n-1$.
    Since $D$ is a right KL interval set, and
    since clearly $x \leqR w_i \leqR y$ for $2 \leq i \leq y$, we see that $w_i \in D$ for $2 \leq i \leq y$.

    Now, applying $T$ to the sequence $w_1,\dots,w_n$, we obtain a new sequence $w'_1,\dots,w'_n$, with $w'_i = T(w_i)$.
    Since $T$ is an edge transport function, we have $\mutilde(w'_i,w'_{i+1}) > 0$ for $1 \le i \le n-1$.
    Since $T$ is right \ti\ preserving, we have $\tauR(w'_i) = \tauR(w_i)$, and so $\tauR(w'_i) \not\subset \tauR(w'_{i+1})$ for $1 \le i \le n-1$.
    Thus $w'_1 \leqR w'_n$, that is, $T(x) \leqR T(y)$.
  \end{proof}

  \begin{corollary}
    \label{cor:edgeTransportEquiv}
    Let $T: D \longrightarrow W$ be an edge transport function.
    Let $\bar D$ be its image.
    Assume in addition that both $D$ and $\bar D$ are right KL interval sets and that $T$ and $T^{-1}$ are right \ti\ preserving.
    Let  $x, y \in D$.
    Then $x \leqR y$ if and only if $T(x) \leqR T(y)$.
    In particular, $x \equivR y$ if and only if $T(x) \equivR T(y)$.

    Let $C \subseteq D$ be a right cell.
    Then $T(C)$ is also a right cell, and $T$ gives an isomorphism from the $W$ graph of $C$ to the $W$ graph of $T(C)$.

    Similarly, with left and right interchanged.
  \end{corollary}

  For ease of future reference, we'll note here that the previous proposition and corollary apply to the Knuth maps.

  \begin{proposition}
    \label{prop:KnuthAKLOrder}
    Let $s, t \in S$ with $st$ of order 3.
    Then $\TstL$ (resp.\ $\TstR$) is right (resp.\ left) KL order preserving.
  \end{proposition}

  \begin{proof}
    This combines \myrefP{prop:KnuthAKLSet}, \myrefP{prop:KnuthAKLCellFunction}, \myrefR{rem:KLcellTauPreserving}, \myrefP{prop:KnuthAEdgeTransportFunction}, and \myrefP{prop:edgeTransportKLOrderA}.
  \end{proof}

  \begin{proposition}
    \label{prop:KnuthALeq}
    Suppose $s, t \in S$ with $st$ of order 3.
    Suppose $x, y \in \DstL$.  Then $x \leqR y$ if and only if $\TstL(x) \leqR \TstL(y)$.
    In particular, if $x, y \in \DstL$ then $x \equivR y$ if and only if $\TstL(x) \equivR \TstL(y)$.

    Let $C \subseteq \DstL$ be a right cell.
    Then $\TstL(C)$ is also a right cell, and $\TstL$ gives an isomorphism from the $W$ graph of $C$ to the $W$ graph of $\TstL(C)$.

    Similarly, with left and right interchanged.
  \end{proposition}

  \begin{proof}
    This combines \myrefP{prop:KnuthAKLSet}, \myrefP{prop:KnuthAKLCellFunction}, \myrefR{rem:KLcellTauPreserving}, \myrefP{prop:KnuthAEdgeTransportFunction}, and \myrefC{cor:edgeTransportEquiv}, after noting that $T_{s,t}^{-1} = T_{t,s}$.
  \end{proof}

  \begin{remark}
    The first part of this corollary is related to Corollary 3.6 of \cite{vogan_1979}.
    That is, Corollary 3.6 of \cite{vogan_1979} is the primitive ideal version of \myrefP{prop:KnuthALeq}.
    The second part of \myrefP{prop:KnuthALeq} is Corollary 4.3 of \citeKL, plus some of Section 5 of \citeKL.
  \end{remark}

  \section{The Generalized \texorpdfstring{$\tau$-invariant}{tau-invariant}, Part 1}
  \label{sec:genTauA}

  In this section, we'll present the easy version of the generalized \ti, the one which appears in \citeKL.
  The generalized \ti\ was first defined in \cite{vogan_1979}, Definition 3.10.

  The generalized \ti\ can be defined with respect to any set of maps each of which has domain a subset of $W$ and range $W$.
  Though, to be useful, the maps need to be (left or right) KL order preserving.
  But first, the definition.

  \begin{definition}
    \label{def:genTauA}
    Let $\euscr F$ be a set of functions each of which has domain a subset of $W$ and range $W$.
    We define the left generalized \ti\ with respect to $\euscr F$ as follows.
    Let $w_1, w_2 \in W$. We say $w_1$ and $w_2$ are equivalent to order 0 if $\tau_L(w_1) = \tau_L(w_2)$.
    For $n \geq 1$, we say $w_1$ and $w_2$ are equivalent to order $n$, $w_1 \underset{n}{\approx} w_2$, if the following two conditions hold.
    \begin{enumerate}
      \item $w_1 \underset{n - 1}{\approx} w_2$.
      \item For every $T \in \euscr F$ with $w_1$ in the domain of $T$, we have that $w_2$ is in the domain of $T$ and $T(w_1) \underset{n - 1}{\approx} T(w_2)$, and similarly with $w_2$ in place of $w_1$.
    \end{enumerate}

    We say that $w_1$ and $w_2$ are left equivalent to infinite order, or that $w_1$ and $w_2$ have the same left generalized \ti, with respect to the set $\euscr F$, if $w_1 \underset{n}{\approx} w_2$ for every non-negative integer $n$.
    We'll write this as $w_1 \equivGTF w_2$.
    Alternate notations are $w_1 \equivUnder{GT} w_2$ or $w_1 \equivUnder{GTL} w_2$, when $\euscr F$ is understood.

    We define analogously the right generalized \ti.
  \end{definition}

  The main property that we need about the left (resp.\ right) generalized \ti\ is that, if defined with respect to an appropriate set $\euscr F$, it is a weaker equivalence relation than that of being in the same right (resp.\ left) cell.

  \begin{theorem}
    \label{thm:genTauA}
    Let $\euscr F$ be a set of right KL order preserving functions and suppose the domain of every $T \in \euscr F$ is a right KL interval set.
    Let $w_1, w_2 \in W$.  If $w_1 \equivR w_2$ then $w_1$ and $w_2$ have the same left generalized $\tau$-invariant with respect to $\euscr F$.
    Similarly, interchanging left and right.
  \end{theorem}

  \begin{proof}
      We will prove by induction on $n$ that $w_1 \equivR w_2$ implies that $w_1 \underset{n}{\approx} w_2$ for all integers $n \geq 0$.
      When $n = 0$, this is true by \myrefP{prop:tauCell}.
      So, assume now that $n \geq 1$ and that $y \equivR w$ implies that $y \underset{k}{\approx} w$ for $0 \leq k \leq n - 1$.
      In particular, this says that $w_1 \underset{n -1}{\approx} w_2$, which is the first condition to be satisfied.

      For condition 2, let $T \in \euscr F$ with domain $D$. \myrefP{prop:KLIntervalEquiv} says that $w_1 \in D$ if and only if $w_2 \in D$.
      If $w_1 \in D$  then
      \myrefP{prop:edgeTransportKLOrderA} says that $T(w_1) \equivR T(w_2)$ and so by induction $T(w_1) \underset{n - 1}{\approx} T(w_2)$.
  \end{proof}

  \begin{theorem}
    \label{thm:genTauKnuthConclusion}
    In the context of \myrefD{def:genTauA}, let $\euscr F$ be a set of right Knuth maps.
    Suppose $y, w \in W$ with $y \equivL w$.
    Then $y \equivGTF w$.
    Similarly, interchanging left and right.
  \end{theorem}

  \begin{proof}
    This combines \myrefT{thm:genTauA} with \myrefP{prop:KnuthAKLOrder} and \myrefP{prop:KnuthAKLSet}.
  \end{proof}

  \begin{remark}
    The last theorem is used (though not stated separately) in Section 5 of \citeKL.
  \end{remark}

  Some examples will probably make the definition of the generalized \ti\ clearer.
  We'll work in the Weyl group of type $D_4$, since that will also allow us to complete the proof of \myrefP{prop:D4LeftCells}.
  Our set $\euscr F$ will be the set of left Knuth maps.
  The generalized \ti, as defined, is an equivalence relation, not, for example, a set, such as the \ti.
  However, we can, in small cases, given an element $w \in W$, draw a picture which has all the information necessary to understand the generalized \ti\ of $w$.
  So, let's look at that.

  Let $W$ be the Weyl group of type $D_4$, with the elements of $S$ labeled as in \autoref{fig:D4}.
  First let $w = s_4$.
  Its generalized \ti\ is pictured in \autoref{fig:genTauForS4}.
  We have $\tau(w) = \{s_4\}$.
  There is only one Knuth map which we can apply to $w$, namely $T_{3,4}$.
  The \ti\ of $T_{3,4}(w) = s_3s_4$ is $\{s_3\}$, as pictured.
  Next, we apply $T_{1,3}$ to $T_{3,4}(w)$.
  We have $s_2 \notin T_{1,3}(T_{3,4}(w)) = s_1s_3s_4$.
  Adding $T_{2,3}(T_{3,4}(w))$ to the picture completes the picture of the generalized \ti\ of $w$.
  There are no more Knuth maps which we can apply, except inverses of the ones already applied.

  \begin{figure}[!ht]
    \centering
    \begin{tikzpicture}[scale=.7, every node/.style = {transform shape}]
      \node (Node4) at (0 * \xSLT, 0 * \ySLT) {\TypeIPicture{4}};
      \node (Node3) at (1.5 * \xSLT, 0 * \ySLT) {\TypeDPicture};
      \node (Node1) at (3 * \xSLT, .7 * \ySLT) {\TypeIPicture{1}};
      \node (Node2) at (3 * \xSLT, -.7 * \ySLT) {\TypeIPicture{2}};

      \draw[->, thick, brown, shorten >= .1cm] (Node4.east) --  node[above] {$T_{3,4}$} (Node3.west);
      \draw[->, thick, brown, shorten <= .2cm, shorten >= .1cm] (Node3.east) --  node[above] {$T_{1,3}$} (Node1.west);
      \draw[->, thick, brown, shorten <= .2cm, shorten >= .1cm] (Node3.east) --  node[above] {$T_{2,3}$} (Node2.west);
    \end{tikzpicture}
    \caption{Generalized \ti\ for $w = s_4$}
    \label{fig:genTauForS4}
  \end{figure}

  Now let $y = s_1s_2$.
  Its generalized \ti\ is pictured in \autoref{fig:genTauForS1s2}.
  We have $\tau(y) = \{s_1, s_2\}$.
  We have $T_{3, 1}(y) = T_{3,2}(y) = s_3s_1s_2$, with $\tau(s_3s_1s_2) = \{s_3\}$, as pictured.
  Finally, let $z = T_{4, 3}(s_3s_1s_2) = s_4s_3s_1s_2$.
  We have $\tau(z) = \{s_4\}$.

  \begin{figure}[!ht]
    \centering
    \begin{tikzpicture}[scale=.7, every node/.style = {transform shape}]
      \node (Node12) at (0 * \xSLT, 0 * \ySLT) {\TypeBPicture{4}};
      \node (Node3) at (1.5 * \xSLT, 0 * \ySLT) {\TypeDPicture};
      \node (Node4) at (3 * \xSLT, 0 * \ySLT) {\TypeIPicture{4}};

      \draw[<-, thick, brown, shorten <= .1cm] (Node4.west) --  node[above] {$T_{4,3}$} (Node3.east);
      \draw[<-, thick, brown, shorten <= .1cm] (Node3.west) --  node[above] {$T_{3,1}$} node[below] {$T_{3,2}$} (Node12.east);
    \end{tikzpicture}
    \caption{Generalized \ti\ for $y = s_1s_2$}
    \label{fig:genTauForS1s2}
  \end{figure}

  Now that we've seen the generalized \ti\ of two elements, let's see what more we can get from this.
  For one, we can get generalized \ti\ pictures of the other elements which we have seen so far just by reversing some arrows (and relabeling them with the inverse function).  For example, if we want to see the generalized \ti\ of $z$, which is the element on the right in \autoref{fig:genTauForS1s2}, we can just reverse the two arrows.
  See \autoref{fig:genTauForS4s3s1s2}.

  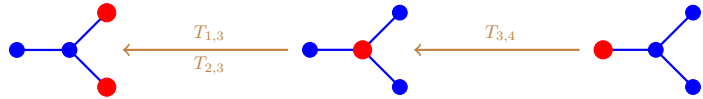
\begin{figure}[!ht]
    \centering
    \begin{tikzpicture}[scale=.7, every node/.style = {transform shape}]
      \node (Node12) at (0 * \xSLT, 0 * \ySLT) {\TypeBPicture{4}};
      \node (Node3) at (1.5 * \xSLT, 0 * \ySLT) {\TypeDPicture};
      \node (Node4) at (3 * \xSLT, 0 * \ySLT) {\TypeIPicture{4}};

      \draw[->, thick, brown, shorten <= .1cm] (Node4.west) --  node[above] {$T_{3,4}$} (Node3.east);
      \draw[->, thick, brown, shorten <= .1cm] (Node3.west) --  node[above] {$T_{1,3}$} node[below] {$T_{2,3}$} (Node12.east);
    \end{tikzpicture}
    \caption{Generalized \ti\ for $z = s_4s_3s_1s_2$}
    \label{fig:genTauForS4s3s1s2}
  \end{figure}

  Let's look at the examples which we have so far to see how the generalized \ti\ separates points.
  Consider $w$ and $z$.
  They both have the same \ti, so $z \underset{0}{\approx} w$.
  The only $T_{s, t}$ defined on $w$ or $z$ is $T_{3, 4}$.
  Set $w_1 = T_{3, 4}(w)$ and $z_1 = T_{3, 4}(z)$.
  Since $\tau(w_1) = \tau(z_1) = \{s_3\}$, we have $z \underset{1}{\approx} w$.
  Now, however, when we apply $T_{1,3}$ to $w_1$ and to $z_1$, the resulting \ti s do not agree.
  So $z_1 \underset{1}{\not\approx} w_1$, and so $z \underset{2}{\not\approx} w$.
  Thus $z$ and $w$ do not have the same generalized \ti.

  Next, let's count the elements of $W$ which have generalized \ti s which we've seen so far, and simple variations of them.
  First, as per \myrefT{thm:genTauA}, any two elements in the same right cell have the same left generalized \ti\ with respect to the left Knuth maps.
  Also, by \myrefP{prop:KnuthAKLCellFunction}, any element which can be obtained from a given element by a sequence of Knuth maps acting on the right is in the same right cell as the starting element.
  Starting with $w = s_4$, we can obtain $s_4s_3$, $s_4s_3s_1$ and $s_4s_3s_2$ using Knuth maps on the right.
  So, the right cell containing $w$ (call it $C^R(w)$) has a least four elements, and all those elements have the generalized \ti\ pictured in \autoref{fig:genTauForS4}.

  Now let's apply $T_{3,4}$ to $C^R(w)$.
  By \myrefC{prop:KnuthALeq}, the result will also be a right cell, specifically $C^R(s_3s_4)$.
  Again by \myrefT{thm:genTauA}, all the elements of $C^R(s_3s_4)$ will have the generalized \ti\ pictured by modifying \autoref{fig:genTauForS4} to reverse the left-most arrow.
  This accounts for (at least) another four elements of $W$.
  We obtain eight more elements by applying $T_{1, 3}$ and $T_{2,3}$ to $C^R(s_3s_4)$.
  Again, the pictures of their generalized \ti s are obtained by reversing arrows in \autoref{fig:genTauForS4}.

  By a similar argument, we obtain nine elements whose generalized \ti\ is pictured in \autoref{fig:genTauForS1s2}, or variations thereof.
  We can obtain another eighteen elements by starting with $s_1s_4$ or $s_2s_4$ instead of $s_1s_2$.
  So, in \autoref{fig:genTauForS4}, \autoref{fig:genTauForS1s2}, and variations, we have seen the generalized \ti\ of 43 elements of $W$.

  We can see another 43 elements of $W$ by multiplying by the long element, $w_0$, on the right.
  By \myrefP{prop:longElement}, this will operate on a picture of the generalized \ti\ by inverting the \ti s and replacing each $T_{s,t}$ with $T_{t,s}$, with the arrow going in the same direction.
  \autoref{fig:genTauForS4w0} and \autoref{fig:genTauForS1s2w0} show the results of applying this operation to
  \autoref{fig:genTauForS4} and \autoref{fig:genTauForS1s2}, respectively.
  \begin{figure}[!ht]
    \centering
    \begin{tikzpicture}[scale=.7, every node/.style = {transform shape}]
      \node (Node4) at (0 * \xSLT, 0 * \ySLT) {\TypeFPicture{4}};
      \node (Node3) at (1.5 * \xSLT, 0 * \ySLT) {\TypeCPicture};
      \node (Node1) at (3 * \xSLT, .7 * \ySLT) {\TypeFPicture{1}};
      \node (Node2) at (3 * \xSLT, -.7 * \ySLT) {\TypeFPicture{2}};

      \draw[->, thick, brown, shorten >= .1cm] (Node4.east) --  node[above] {$T_{4,3}$} (Node3.west);
      \draw[->, thick, brown, shorten <= .2cm, shorten >= .1cm] (Node3.east) --  node[above] {$T_{3,1}$} (Node1.west);
      \draw[->, thick, brown, shorten <= .2cm, shorten >= .1cm] (Node3.east) --  node[above] {$T_{3,1}$} (Node2.west);
    \end{tikzpicture}
    \caption{Generalized \ti\ for $s_4w_0$}
    \label{fig:genTauForS4w0}
  \end{figure}

  \begin{figure}[!ht]
    \centering
    \begin{tikzpicture}[scale=.7, every node/.style = {transform shape}]
      \node (Node12) at (0 * \xSLT, 0 * \ySLT) {\TypeAPicture{4}};
      \node (Node3) at (1.5 * \xSLT, 0 * \ySLT) {\TypeCPicture};
      \node (Node4) at (3 * \xSLT, 0 * \ySLT) {\TypeFPicture{4}};

      \draw[<-, thick, brown, shorten <= .1cm] (Node4.west) --  node[above] {$T_{3,4}$} (Node3.east);
      \draw[<-, thick, brown, shorten <= .1cm] (Node3.west) --  node[above] {$T_{1,3}$} node[below] {$T_{2,3}$} (Node12.east);
    \end{tikzpicture}
    \caption{Generalized \ti\ for $s_1s_2w_0$}
    \label{fig:genTauForS1s2w0}
  \end{figure}

  So, we've now seen the generalized \ti\ of 86 elements of $W$.
  Two more elements are the identity element and the long element, with $\tau(e) = \varnothing$ and $\tau(w_0) = \{s_1, s_2, s_3, s_4\}$ (and thus no $T_{s,t}$ maps are defined on either), for a total of 88 elements.
  The elements studied in \myrefS{sec:D4cells} comprise 104 elements.
  Since $W$ has 192 elements, we will have seen all the generalized \ti\ pictures once we draw those for the elements from \myrefS{sec:D4cells}.
  So let's do that.

  For a type $\scC$ element, it's simple.
  See \autoref{fig:genTauForS1s2s4}.
  For type $\scD$, reverse the arrow.

  \begin{figure}[!ht]
    \centering
    \begin{tikzpicture}[scale=.7, every node/.style = {transform shape}]
      \node (NodeC) at (1.5 * \xSLT, 0 * \ySLT) {\TypeCPicture};
      \node (NodeD) at (3 * \xSLT, 0 * \ySLT) {\TypeDPicture};

      \draw[->, thick, brown, shorten >= .1cm] (NodeC.east) --  node[above] {$T_{3,1}\quad T_{3,2}$} node[below] {$T_{3,4}$} (NodeD.west);
    \end{tikzpicture}
    \caption{Generalized \ti\ for $s_1s_2s_4$}
    \label{fig:genTauForS1s2s4}
  \end{figure}

  The generalized \ti\ picture for an element of type $\scA_1$ is an infinite chain.
  It does not circle around to its start after six or twelve \ti s have been seen, as the actual elements in the cells do.
  It is just a record of \ti s as we apply any applicable Knuth maps.
  See \autoref{fig:genTauForA1}.
  Thus, the elements of type $\scA_1$ in the three figures \refDFourFigures\ all have the same generalized \ti.
  For any type $\scA$ or type $\scB$ element, we can use the same picture after reversing some of the arrows.

  \begin{figure}[!ht]
    \centering
    \begin{tikzpicture}[scale=.64, every node/.style = {transform shape}]
      \node at (-1.5 * \xSLT, 0 * \ySLT) {\Large{$\dots$}};
      \node (NodeA2) at (-1 * \xSLT, 0 * \ySLT) {\TypeAPicture{2}};
      \node (NodeB4) at (0 * \xSLT, 0 * \ySLT) {\TypeBPicture{4}};
      \node (NodeA1) at (1 * \xSLT, 0 * \ySLT) {\TypeAPicture{1}};
      \node (NodeB2) at (2 * \xSLT, 0 * \ySLT) {\TypeBPicture{2}};
      \node (NodeA4) at (3 * \xSLT, 0 * \ySLT) {\TypeAPicture{4}};
      \node at (3.5 * \xSLT, 0 * \ySLT) {\Large{$\dots$}};

      \draw[<-, thick, brown, shorten >= .1cm] (NodeA2.east) --  node[above] {$T_{3,1}$} (NodeB4.west);
      \draw[<-, thick, brown, shorten >= .1cm] (NodeB4.east) --  node[above] {$T_{2,3}$} (NodeA1.west);
      \draw[->, thick, brown, shorten >= .1cm] (NodeA1.east) --  node[above] {$T_{4,3}$} (NodeB2.west);
      \draw[->, thick, brown, shorten >= .1cm] (NodeB2.east) --  node[above] {$T_{3,1}$} (NodeA4.west);
    \end{tikzpicture}
    \caption{Generalized \ti\ for Type $\scA_1$}
    \label{fig:genTauForA1}
  \end{figure}

  Now that we have seen all the possible generalized \ti s, we see that, though there are some elements in $W$ not in the clumps which have the same \ti\ as a type $\scC$ element, none of them have the same generalized \ti\ as a type $\scC$ element.
  Similarly for the other types.

  We can use the above discussion to complete the proof of \myrefP{prop:D4LeftCells}.

  \begin{proof}[Second half of the proof of \myrefP{prop:D4LeftCells}]
    As before, we'll show this for $C(10, a)$.
    The others are similar.
    We have already seen that the elements of $C(10, a)$ are in the same left cell.

    Now we have to see that the elements of $C(10, a)$ are not in the same left cell as any other elements of $\WDFour$.
    We'll do this using the generalized \ti\ with respect to the set of right Knuth maps.
    The converse of \myrefT{thm:genTauKnuthConclusion} says that two elemnts which do not have the same generalized \ti\ with respect to the set of right Knuth maps are not in the same left cell.

    So, basically, we just need to switch sides from what we did above.
    Above, we computed left generalized \ti s, which are constant on right cells.
    Now, instead, we are considering $C(10, a)$, which we want to show is a left cell.
    To do that, we need to compute the right generalized \ti\ of its elements, and of the other elements in $\WDFour$, with respect to the right Knuth maps.

    This is no different than what we have done above.
    We just need to reverse the order in which the reduced expressions of the elements in question are written.
    For example, the bottom element of $C(10, a)$ is $s_1s_2s_4$, an involution.
    So, it is right type $\scC$ as well as left type $\scC$.
    Both its left and right generalized \ti s are illustrated by \autoref{fig:genTauForS1s2s4}.

    Similarly, the bottom element of $C(14, a, 4)$ is $s_4s_3s_4$, also an involution.
    So, it is of right type $\scA_4$ as well as being of left type $\scA_4$.
    In $C(10, b)$, the lower type $\scD$ element is $s_3s_1s_2s_4s_3$, also an involution and so of right type $\scD$ as well as left type $\scD$.
    In $C(14, b, 4)$, the lower type $\scB_4$ element is $s_1s_2s_3s_2s_1$, also an involution and so of right type $\scB_4$ as well as left type $\scB_4$.

    So, as above, we see that no other elements have the same right generalized \ti\ as the elements of $C(10, a)$.
    We can therefore conclude that $C(10, a)$ is a left cell.
  \end{proof}

  \begin{remark}
    As we've just seen, for the Weyl group of type $D_4$, the generalized $\tau$-invariant defined using Knuth maps alone is enough to separate the left cells.
    Once we get to $D_6$, that's no longer true.
    In the Weyl group of type $D_6$, there are different left cells with the same generalized \ti\ if only Knuth maps are used.
    In \myrefS{sec:D4Maps}, we'll add the $D_4$ maps to the generalized \ti.
    As we'll see in a later paper, that will be enough to separate the left cells in Weyl groups of type $D_n$.
  \end{remark}

  Before we leave this section, let's recall some of the motivation for the definitions which we've made.

  \begin{definition}
    \label{def:equivalenceRelationF}
    Let $\euscr F$ be a set of maps such that each $T \in \euscr F$ has domain a subset of $W$ and range $W$.
    Write $\equivF$ for the equivalence relation on $W$ generated by $\euscr F$.
    That is, we have $w \equivF T(w)$ for every $T \in \euscr F$ and $w$ in the domain of $T$.
  \end{definition}

  By definition, if $\euscr F$ is a set of left (resp.\ right) KL cell functions, then $\equivF$ is a stronger equivalence relation than $\equivL$ (resp.\ $\equivR$).
  By \myrefT{thm:genTauA}, if $\euscr F'$ is a set of right (resp.\ left) KL order preserving functions whose domains are right (resp.\ left) KL interval set, then $\equivL$ (resp.\ $\equivR$) is a stronger equivalence relation than $\equivGTFPrime$.
  Ideally, we would like to find a set $\euscr F$ of left KL cell functions and a set $\euscr F'$ of right KL order preserving functions such that $\equivF$ and $\equivGTFPrime$ coincide.
  In that case, both will coincide with $\equivL$.

  As described in Section 5 of \citeKL\, this ideal situation is achieved for the Weyl group of type $A_n$, where $\euscr F$ (resp.\ $\euscr F'$) is the set of Knuth maps acting on the left (resp.\ right).

  \section{Edge Transport Functions, Part 2}
  \label{sec:edgeTransportPairs}
  In this section, we'll extend the definitions and results of \myrefS{sec:edgeTransportFunctions} to the maps defined in \myrefS{sec:B2Maps} and the maps which we'll define in \myrefS{sec:D4Maps}.

  \myrefD{def:KLIntervalSet} still works for this situation, but we need to modify
  \myrefD{def:tauPreserving} and \myrefD{def:KLCellFunction} a little.

  \begin{definition}
    \label{def:tauPreservingB}
    A function $T: D \longrightarrow \mathcal P(W)$, where $D \subset W$, is left \ti\ preserving if for all $w \in D$ and $w' \in T(w)$ we have $\tau(w') = \tau(w)$.
    We define similarly right \ti\ preserving.
  \end{definition}

  \begin{definition}
    \label{def:KLCellFunctionB}
    A function $T: D \longrightarrow \mathcal P(W)$, where $D \subset W$, is a left KL cell function if for all $w \in D$ and $w' \in T(w)$ we have $w' \equivL w$.
    We define similarly right KL cell function.
  \end{definition}

  \begin{remarkNumbered}
    As before, a left (resp.\ right) KL cell function is right (resp.\ left) \ti\ preserving by \myrefP{prop:tauCell}.
  \end{remarkNumbered}

  The next two propositions have the same proof as when $st$ is of order 3.  (See \myrefP{prop:KnuthAKLSet} and \myrefP{prop:KnuthAKLCellFunction}.)

  \begin{proposition}
  \label{prop:BTwoKLSet}
  Suppose $s, t \in S$ with $st$ of order 4.
  Then $\DstL$ is a right KL interval set, and $\DstR$ is a left KL interval set.
  \end{proposition}

  \begin{proposition}
  \label{prop:BTwoKLCellFunction}
  Suppose $s, t \in S$ with $st$ of order 4.
  Then $\TstL$ is a left KL cell function and
  $\TstR$ is a right KL cell function.
  \end{proposition}

  In our current situation, the definition of an edge transport function is rather more complicated than in \myrefS{sec:edgeTransportFunctions}.
  We'll use a definition which encompasses the maps of this paper ($B_2$ maps and $D_4$ maps) and also hopefully those which might be defined and/or studied in the future.
  Specifically, there should be an edge transport theorem coming from a parabolic subgroup of type $E_6$.
  In addition, as seen in Section 10 of \cite{lusztig_1985}, there are (more complicated) edge transport theorems for parabolic subgroups generated by $s, t \in S$ where the order of $st$ is larger than 4.
  These also can be used to define maps which seem like they will fit into \myrefD{def:edgeTransportPair}.

  \begin{definition}
    \label{def:pairFunction}
    Let $T: D \longrightarrow \mathcal P(W)$, where $D \subset W$.
    Let $\bar D = \cup_{w \in D}T(w)$.
    Let $\pair(T): \bar D \longrightarrow \mathcal P(W)$ be defined by $\pair(T)(\bar w) = \{w \in D \mid T(w) = \bar w \}$.
    We'll call $\pair(T)$ the pair function to $T$.
    Note that $\pair(\pair(T)) = T$.
  \end{definition}

  \begin{definition}
    \label{def:edgeTransportPair}
    Let $T: D \longrightarrow \mathcal P(W)$, where $D \subset W$.
    Let $\bar T = \pair(T)$.
    We'll call $T$ a type 2 edge transport function if it satisfies the following conditions.
    Suppose $y, w \in D$
    with $\mutilde(y,w) \neq 0$.
    \begin{enumerate}
      \item If $\abs{T(y)} = \abs{T(w)} = k$ then we can write $T(y) = \{y_1,\dots,w_k\}$ and $T(w) = \{w_1,\dots,w_k\}$ so that $\mutilde(y_i, w_i) \neq 0$ for all $1 \leq i \leq k$.
      \item If $\abs{T(y)} \neq \abs{T(w)}$ then for every $y' \in T(y)$ and every $w' \in T(w)$ we have $\mutilde(y', w') \neq 0$.
      \item $\bar T$ also satisfies the above conditions.
    \end{enumerate}
  \end{definition}

  \begin{remarkNumbered}
    \label{rem:typeIToTypeII}
    A type 1 edge transport function $T$ can be turned into a type 2 edge transport function $T'$ simply by setting $T'(w) = \{T(w)\}$.
    If $T$ is left \ti\ preserving (resp.\ a left KL cell function), then so is $T'$, and similarly with right in place of left.
    If $T$ is a Knuth map, we will also call $T'$ a Knuth map.
  \end{remarkNumbered}

  Now let's look at the edge transport theorems which are the subject of this paper and see that the maps arising from them ($B_2$ maps and $D_4$ maps) are type 2 edge transport functions.
  Since the edge transport theorems have the same form, we can prove both at once.

  \begin{proposition}
    \label{prop:edgeTransportPair}
    Let $T: D \longrightarrow \mathcal P(W)$, where $D \subset W$.
    Let $\bar T = \pair(T)$.
    Assume that we have the following:
    \begin{enumerate}
      \item For $w \in D$, $\abs{T(w)}$ is 1 or 2, and similarly for $\bar T$.
      \item If $T(y) = \{y', y''\}$  with $y' \neq y''$ then $\bar T(y') = \bar T(y'') = \{y\}$.
      \item If $T(y) = \{y'\}$ then $\bar T(y') = \{y, y^*\}$ with $y \neq y^*$ and $T(y^*) = \{y'\}$ .
      \item Suppose $T(y) = \{y', y''\}$ and $T(w) = \{w', w''\}$ with $y' \neq y''$ and $w' \neq w''$. Then we have
      \begin{enumerate}
        \item $\mutilde(y, w) = \mutilde(y', w') + \mutilde(y', w'')$
        \item $\mutilde(y', w') = \mutilde(y'', w'')$
        \item $\mutilde(y', w'') = \mutilde(y'', w')$
      \end{enumerate}
      \item Suppose $T(y) = \{y', y''\}$  with $y' \neq y''$ and $T(w) = \{w'\}$.
      Let $w^* \in W$ be such that $w^* \neq w$ and $T(w^*) = \{w'\}$.
      Then we have
      \begin{equation*}
        \mutilde(y, w) = \mutilde(y, w^*) = \mutilde(y', w') = \mutilde(y'', w')
      \end{equation*}
      \item Conditions 2--5 also hold with $T$ and $\bar T$ interchanged.
    \end{enumerate}

    Then $T$ and $\bar T$ are type 2 edge transport functions.
  \end{proposition}

  \begin{proof}
    It suffices to prove this for $T$, since our assumptions are symmetric in $T$ and $\bar T$, and since $\pair(\pair(T)) = T$.
    To verify statement 1 of \myrefD{def:edgeTransportPair}, we'll split into two cases.
    Let $y, w \in D$ with $\mutilde(y,w) \neq 0$ and assume first that $\abs{T(y)} = \abs{T(w)} = 2$.
    For this case we'll use assumption 4 of our proposition.
    From equation (a) of assumption 4, we have that either $\mutilde(y', w') \neq 0$ or $\mutilde(y', w'') \neq 0$.
    Without loss of generality we can assume the former.
    Then set $y_1 = y'$, $y_2 = y''$, $w_1 = w'$, and $w_2 = w''$.
    From equation (b) of assumption 4, we have $\mutilde(y'', w'') = \mutilde(y', w')$.
    So this gives statement 1 of the proposition.

    Next assume that $\abs{T(y)} = \abs{T(w)} = 1$.
    Our assumption 6 says that have assumption 5 with $\bar T$ in place of $T$.
    Let's introduce some new letters to avoid the overlap.
    Set $x = T(y)$ and $z = T(w)$.
    Then, let $\bar T(x) = \{x', x''\}$ with $x' = y$ and $\bar T(z) = \{z', z''\}$ with $z' = w$.
    We have $\mutilde(x, z) = \mutilde(x', z') + \mutilde(x', z'')$.
    Given \myrefT{thm:nonNegative} and our hypothesis that $\mutilde(x', z') \neq 0$, we conclude that
    $\mutilde(x, z) \neq 0$, that is, $\mutilde(T(y), T(w)) \neq 0$, as was to have been shown.

    To verify statement 2 of \myrefD{def:edgeTransportPair}, since $\mutilde(y, w) = \mutilde(w, y)$, we can without loss of generality assume that  $\abs{T(y)} = 2$ and $\abs{T(w)} = 1$.
    Then assumption 5 yields the desired conclusion.
  \end{proof}

  \begin{proposition}
    \label{prop:KnuthBEdgeTP}
    Let $s, t \in S$ with $st$ of order 4.
    Then $T_{s, t}$ is a type 2 edge transport function, with pair function $T_{t, s}$.
  \end{proposition}

  \begin{proof}
    That $\pair(T_{s,t}) = T_{t,s}$ is clear from the definition.
    We'll use \myrefP{prop:edgeTransportPair}.
    The first three conditions follow from the definition of the maps.
    The next two conditions are \myrefT{thm:stsStrings}.
  \end{proof}

  Now, let's prove the analogue of \myrefP{prop:edgeTransportKLOrderA} for these maps.
  Our goal is \myrefP{prop:edgeTransportLeqB} and the propositions following it.

  \begin{definition}
    A function $T: D \longrightarrow \mathcal P(W)$, with $D \subset W$, is left KL order preserving if for $x, y \in D$ with $x \leqL y$, we have the following:
    \begin{enumerate}
      \item If $\abs{T(y)} = \abs{T(w)} = k$ then we can write $T(y) = \{y_1,\dots,y_k\}$ and $T(w) = \{w_1,\dots,w_k\}$ so that $y_i \leqL w_i$ for all $1 \leq i \leq k$.
      \item If $\abs{T(y)} \neq \abs{T(w)}$ then for every $y' \in T(y)$ and every $w' \in T(w)$ we have $y' \leqL w'$.
    \end{enumerate}
    We define similarly right KL order preserving.
  \end{definition}

  \begin{proposition}
    \label{prop:edgeTransportLeqB}
    Let $T$ be a type 2 edge transport function, with $D$ the domain of $T$.
    Assume in addition that $D$ is a right KL interval set and that $T$ is right $\tau$-invariant preserving.
    Then $T$ is right KL order preserving.
    Similarly, with left and right interchanged.
  \end{proposition}

  \begin{proof}
    Let  $y, w \in D$ with $y \leqR w$.
    By definition, there is a sequence $x_1,\dots,x_n$ of elements of $W$
    with $x_1 = y$ and $x_n = w$ such that $\mutilde(x_i,x_{i+1}) \neq 0$ and $\tau_R(x_i) \not\subset \tau_R(x_{i+1})$ for $1 \le i \le n-1$.
    Then $y \leqR x_i \leqR w$ for all $2 \le i \le n-1$, and so, since $D$ is a right KL interval set, $x_i \in D$ for all $2 \le i \le n-1$.

    Suppose first that $\abs{T(x_i)} = k$ for some $k$ and all $1 \leq i \leq n$.
    Then we can find $k$ sequences $x_1^j,\dots,x_n^j$ for $1 \leq j \leq k$ such that $\mutilde(x_i^j,x_{i+1}^j) \neq 0$ for all $1 \leq j \leq k$ and such that $T(x_i) =  \{x_i^1,\dots,x_i^k\}$ for all $1 \le i \leq n$.
    This follows easily from the statement 1 of \myrefD{def:edgeTransportPair}, using induction on $n$.
    Since $T$ is right \ti\ preserving, we have that $\tauR(x_i^j) = \tauR(x_i)$ for all applicable $i$ and $j$.
    Thus, each sequence $x_1^j,\dots,x_n^j$ demonstrates that $x_1^j \leqR x_n^j$, as was to have been shown.

    Now suppose that $\abs{T(x_i)} \neq \abs{T(x_{i'})}$ for some $1 \leq i, i' \leq n$.
    Then given $y' \in T(y)$ and $w' \in T(w)$, there is a sequence $x_1',\dots,x_n'$ such that $y' = x_1'$, $w' = x_n', $ $x_i' \in T(x_i)$ for $1 \leq i \leq n$ and $\mutilde(x_i', x_{i + 1}') \neq 0$ for $1 \leq i \leq n - 1$.
    This follows easily from the previous case and statement 2 of \myrefD{def:edgeTransportPair}, using induction on $n$.
    As in the previous case, this sequence demonstrates that $y' \leqR w'$.
  \end{proof}

  \begin{proposition}
    \label{prop:edgeTransportEquivalence}
    Suppose $T: D \longrightarrow \mathcal P(W)$, with $D \subset W$, is left KL order preserving.
    Suppose we have $x, y \in D$ with $x \equivL y$.
    Then we have the following:
    \begin{enumerate}
      \item If $\abs{T(y)} = \abs{T(w)} = k$ then we can write $T(y) = \{y_1,\dots,y_k\}$ and $T(w) = \{w_1,\dots,w_k\}$ so that $y_i \equivL w_i$ for all $1 \leq i \leq k$.
      \item If $\abs{T(y)} \neq \abs{T(w)}$ then for every $y' \in T(y)$ and every $w' \in T(w)$ we have $y' \equivL w'$.
    \end{enumerate}
  \end{proposition}
  Similarly, interchanging left and right.

  \begin{proof}
    Statement 2 is clear.  So assume that $\abs{T(y)} = \abs{T(w)} = k$.
    Since $y \leqL w$, we can write $T(y) = \{y_1,\dots,y_k\}$ and $T(w) = \{w_1,\dots,w_k\}$ so that $y_i \leqL w_i$ for all $1 \leq i \leq k$.
    Since $w \leq y$, there is a permutation $\sigma$ of $1,\dots,k$ such that $w_i \leqL y_{\sigma(i)}$ for all $1 \leq i \leq k$.
    If $\sigma$ is the identity, we are done.
    If not, still, some power of $\sigma$ is the identity, and we can use that to conclude that $y_{\sigma(i)} \leqL y_i$, and thus reach the desired conclusion.
  \end{proof}

  \begin{proposition}
    \label{prop:cellsToCells}
    Let $T$ be a type 2 edge transport function, with $D$ and $\bar D$ as in \myrefD{def:pairFunction}, and $\bar T = \pair(T)$.
    Assume in addition that both $D$ and $\bar D$ are right KL interval sets and that $T$ and $\bar T$ are right \ti\ preserving.

    Let $C$ be a right cell contained in $D$.
    \begin{enumerate}
      \item If $\abs{T(w)} = k$ for some $k$ and all $w \in C$ then $T(C)$ is a union of at most $k$ right cells.
      \item If $\abs{T(y)} \neq \abs{T(w)}$ for some $y, w \in C$ then $T(C)$ is a right cell.
    \end{enumerate}

    Similarly, interchanging left and right.
  \end{proposition}

  \begin{proof}
    Assume first that $\abs{T(w)} = k$ for some $k$ and all $w \in C$.
    Fix $x \in C$ and write $T(x) = \{x_1,\dots,x_k\}$.
    Let $C_i$ be the right cell containing $x_i$ for $1 \leq i \leq k$.
    Let $y \in C$.
    By \myrefP{prop:edgeTransportEquivalence}, we can write $T(y) = \{y_1,\dots,y_k\}$ so that $y_i \in C_i$.
    That is, $T(C) \subset \cup_{i=1}^k C_i$.

    Now suppose $y_i \in C_i$ for some $i$, that is, $y_i \equivR x_i$.
    Then, again by the previous proposition, this time applied to $\bar T$, since $x \in \bar T(x_i)$, there is a $y \in \bar T(y_i)$ such that $y \equivR x$ .
    Now $y_i \in T(y)$, that is, $y_i \in T(C)$.  Thus $C_i \subset T(C)$.
    So $T(C) = \cup_{i=1}^k C_i$

    Now suppose $\abs{T(y)} \neq \abs{T(w)}$ for some $y, w \in C$.
    Choose $y' \in T(y)$ and let $C'$ be the left cell containing $y'$.
    From statement 2 of \myrefP{prop:edgeTransportEquivalence}, we see that $T(y)$ and $T(w)$ are contained in $C'$.
    For any other $z \in C$, we have either $\abs{T(z)} \neq \abs{T(y)}$ or $\abs{T(z)} \neq \abs{T(w)}$, and so similarly, we have $T(z) \subseteq C'$.
    Now suppose $z' \in C'$.
    Since $z' \equivR y'$, applying \myrefP{prop:edgeTransportEquivalence} to $\bar T$, $y'$, and $z'$, we can find a $z \in \bar T(z')$ with $z \equivR y$.
    Then $z' \in T(z)$.
    We conclude that $C' \in T(C)$, and so $C' = T(C)$.
  \end{proof}

  As a consequence, we have these results for the $B_2$ maps.

  \begin{corollary}
    \label{cor:BTwoConsequences}
    Let $s, t \in S$ with $st$ of order 4.  We have
    \begin{enumerate}
      \item $\TstL$ is right KL order preserving.
      \item Suppose $x,y \in \DstL$ with $x \equivR y$.
      Then we can write $\TstL(x) = \{x',x''\}$ and $\TstL(y) = \{y',y''\}$ (where possibly $x' = x''$ and/or $y' = y''$) so that $x' \equivR y'$ and $x'' \equivR y''$.
      \item Let $C$ be a right cell contained in $\DstL$.
      Then $\TstL(C)$ is either a right cell or a union of two right cells.
    \end{enumerate}
    Similarly, interchanging left and right.
  \end{corollary}

  \begin{proof}
    \myrefP{prop:KnuthBEdgeTP} says that both $\TstL$ and $\pair(\TstL) = T^L_{t,s}$ are type 2 edge transport functions.
    \myrefP{prop:BTwoKLCellFunction} implies that both are right \ti\ preserving.
    \myrefP{prop:BTwoKLSet} says that both domains are right KL interval sets.
    So, the conclusions of \myrefP{prop:edgeTransportLeqB}, \myrefP{prop:edgeTransportEquivalence}, and \myrefP{prop:cellsToCells} hold for $\TstL$.
  \end{proof}

  \section{The Generalized \texorpdfstring{$\tau$-invariant}{tau-invariant}, Part 2}
  \label{sec:genTauAB}

  In this section, we'll give the more elaborate definition of the generalized \ti, this time for type 2 edge transport functions.
  Again, this is based on the definition of \cite{vogan_1979}.

  \begin{definition}
    \label{def:genTau}
    Let $\euscr F$ be a set of functions each of which  has domain a subset of $W$ and range $\mathcal P(W)$.
    We define the left generalized \ti\ with respect to $\euscr F$ as follows.
    Let $w_1, w_2 \in W$. We say $w_1$ and $w_2$ are equivalent to order 0 if $\tauL(w_1) = \tauL(w_2)$.
    For $n \geq 1$, we say $w_1$ and $w_2$ are equivalent to order $n$, $w_1 \underset{n}{\approx} w_2$, if
    \begin{enumerate}
      \item $w_1 \underset{n - 1}{\approx} w_2$.
      \item For every $T \in \euscr F$ with $w_1$ in the domain of $T$ we have that $w_2$ is in the domain of $T$, and for every $y_1 \in T(w_1)$, there is a $y_2 \in T(w_2)$ such that $y_1 \underset{n - 1}{\approx} y_2$.
      Similarly, interchanging $w_1$ and $w_2$.
    \end{enumerate}

    We say that $w_1$ and $w_2$ are left equivalent to infinite order, or that $w_1$ and $w_2$ have the same left generalized $\tau$-invariant, with respect to the set $\euscr F$, if $w_1 \underset{n}{\approx} w_2$ for every non-negative integer $n$.
    We'll write this as $w_1 \equivGTF w_2$.
    Alternate notations are $w_1 \equivUnder{GT} w_2$ or $w_1 \equivUnder{GTL} w_2$, when $\euscr F$ is understood.
  \end{definition}

  Again, the main property that we need about the right generalized \ti\ is that (under the right conditions) it is a weaker equivalence relation than that of being in the same left cell.

  \begin{theorem}
    \label{thm:genTau}
    Let $\euscr F$ be a set of functions each of which  has domain a subset of $W$ and range $\mathcal P(W)$.
    Suppose that every $T \in \euscr F$ is a right KL order preserving function and suppose the domain of every $T \in \euscr F$ is a right KL interval set.
    Let $w_1, w_2 \in W$.  If $w_1 \equivR w_2$ then $w_1$ and $w_2$ have the same left generalized \ti\ with respect to $\euscr F$.
  \end{theorem}

  \begin{proof}
      We will prove by induction on $n$ that $w_1 \equivR w_2$ implies that $w_1 \underset{n}{\approx} w_2$ for all integers $n \geq 0$.
      When $n = 0$, this is true by \myrefP{prop:tauCell}.
      Assume now that $n \geq 1$ and that  $y \equivR w$ implies that $y \underset{r}{\approx} w$ for $0 \leq r \leq n - 1$.
      In particular, this says that $w_1 \underset{n -1}{\approx} w_2$, which is the first condition to be satisfied.
      For condition 2, let $T \in \euscr F$ with domain $D$.
      \myrefP{prop:KLIntervalEquiv} says that $w_1 \in D$ if and only if $w_2 \in D$.

      By \myrefP{prop:edgeTransportEquivalence},
      for every $y_1 \in T(w_1)$, there is a $y_2 \in T(w_2)$ such that $y_1 \equivR y_2$, and similarly interchanging $w_1$ and $w_2$.
      Then the desired conclusion follows by induction.
  \end{proof}

  \begin{remark}
    In general, such a $\euscr F$ will include type 1 edge transport functions which we regard as type 2 edge transport functions as in \myrefR{rem:typeIToTypeII}.
  \end{remark}

  \begin{theorem}
    \label{thm:genTauBTwoConclusion}
    In the context of \myrefD{def:genTau}, let $\euscr F$ be a set of right Knuth maps and right $B_2$. maps
    Suppose $y, w \in W$ with $y \equivL w$.
    Then $y \equivGTF w$.
    Similarly, interchanging left and right.
  \end{theorem}

  \begin{proof}
    For the Knuth maps, we have already seen this in the proof of \myrefT{thm:genTauKnuthConclusion}.
    For the $B_2$ maps, this combines \myrefT{thm:genTau} with \myrefC{cor:BTwoConsequences}--1 and \myrefP{prop:BTwoKLSet}.
  \end{proof}

  Let's continue looking at past theorems using the generalized $\tau$-invariant, as motivation.
  We need to modify \myrefD{def:equivalenceRelationF} a little.

  \begin{definition}
    \label{def:equivalenceRelationF2}
    Let $\euscr F$ be a set of functions each of which has domain a subset of $W$ and range $\mathcal P(W)$.
    We'll write $\equivF$ for the equivalence relation on $W$ generated by $\euscr F$.
    That is, we have $w \equivF y$ for every $T \in \euscr F$, $w$ in the domain of $T$, and $y \in T(w)$.
  \end{definition}

  Again, by definition, if $\euscr F$ is a set of left (resp.\ right) KL cell functions, then $\equivF$ is a stronger equivalence relation than $\equivL$ (resp.\ $\equivR$).
  By \myrefT{thm:genTau}, if $\euscr F$ is a set of right (resp.\ left) KL order preserving functions whose domains are right (resp.\ left) KL interval sets, then $\equivL$ (resp.\ $\equivR$) is a stronger equivalence relation than $\equivGTF$.

  Ideally, we would like to find a left transport set $\euscr F$ and a right transport set $\euscr F'$ such that $\equivF$ and $\equivGTFPrime$ coincide.
  In that case, both will coincide with $\equivL$.

  As with type $A_n$, the ideal situation of a set $\euscr F$ of left KL cell functions and a set $\euscr F'$ of right KL order preserving functions such that $\equivF$ and $\equivGTFPrime$ coincide is achieved for type $B_n$ and $C_n$.
  Here $\euscr F$ is the set of left Knuth maps and left $B_2$ maps, and similarly $\euscr F'$ is the set of right Knuth maps and right $B_2$ maps.

  This theorem is proved in \cite{garfinkle_1993} as Corollary 3.5.6.

  \section{\texorpdfstring{$D_4$}{D4} Maps}
  \label{sec:D4Maps}
  It remains now to define the maps associated to \myrefT{thm:main}, and to show that they have the properties described in \myrefS{sec:edgeTransportPairs}.
  We will then have the machinery necessary to carry out (in future papers) the program of classifying left (and right) cells in the Weyl group of type $D_n$.

  \myrefT{thm:main} is a theorem about the relationship between $\mutilde$ terms connecting elements of type $\scA$ and elements of type $\scC$.
  So, there are maps implicit in that theorem.
  Let's start by defining those maps.

  \begin{definition}
    \label{def:D4MainMap}
    Let $D_{\scC} \subset W$ be the set of elements of type $\scC$.
    Let $i \in \refSet$.
    Let $T_i$ be the map which associates to a $w \in D_{\scC}$ the set consisting of the one or two elements of type $\scA_i$ which are in the same clump as $w$.
    Let $\bar D_i \subset W$ be the set of elements of type $\scA_i$.
    Let $\bar T_i$ be the map which associates to a $w \in \bar D_i$ the set consisting of the one or two elements of type $\scA_i$ which are in the same clump as $w$.
    Note that $\bar T_i = \pair(T_i)$.
  \end{definition}

  For applications, variations on that map are more useful.
  Next, we'll define the variations (following \cite{garfinkle_vogan_1992}.)
  We'll call the maps $D_4$ maps.
  As usual, these maps have a left and a right version.
  We'll define the left version below, as usual omitting the superscript $L$.

  Please note, our terminology so far has been a shorthand.
  We've talked about type $\scC$, etc.
  This so far has been left type $\scC$.
  There are also right type $\scC$ elements.
  Similarly we have right type $\scA_1$ elements, etc.
  If the generators of the parabolic subgroup are not labeled 1, 2, 3, and 4, (as for example in the Weyl group of type $E_6$) then we'd use different subscripts for our type $\scA$ and $\scB$ elements.
  That will also be convenient in a situation where there are two (or more) parabolic subgroups of type $D_4$, as for example in affine $D_4$.

  \begin{definition}
    \label{def:D4Ops}
    Let $i \in \refSet$.
    \begin{enumerate}
      \item Let $T_{{\scD},i}$ be the map which associates to a $w \in W$ of type $\scA_i$ the set consisting of the one or two elements of type $\scD$ which are in the same clump as $w$.
      \item Let $T_{{\scC},i}$ be the map which associates to a $w \in W$ of type $\scB_i$ the set consisting of the one or two elements of type $\scC$ which are in the same clump as $w$.
      \item Let $T_{i,{\scC}}$ be the map which associates to a $w \in W$ of type $\scC$ the set consisting of the one or two elements of type $\scB_i$ which are in the same clump as $w$.
      \item Let $T_{i,{\scD}}$ be the map which associates to a $w \in W$ of type $\scD$ the set consisting of the one or two elements of type $\scA_i$ which are in the same clump as $w$.
    \end{enumerate}
  \end{definition}

  \begin{notation}
    We will write $D_{{\scD},i}(W)$ for the domain of the map $T_{{\scD},i}$, that is, $D_{{\scD},i}(W)$ is simply the set of elements of type $\scA_i$.
    We have an analogous notation for the domains of the other maps in \myrefD{def:D4Ops}.
  \end{notation}

  We have an alternate characterization of the maps in \myrefD{def:D4Ops}, analogous to that of \myrefP{prop:talbAltDef} and \myrefP{prop:KnuthBAltDef}.

  \begin{proposition}
    Let $i \in \refSet$. Let $j, k$ be such that $\refSet = \{i, j, k\}$.
    \begin{enumerate}
      \item  If $w \in D_{{\scD},i}(W)$ then
      \begin{equation*}
        T_{{\scD},i}(w) = D_{i,{\scD}}(W) \cap \{s_i w,s_3 s_j s_k w,s_3 s_j s_3 w, s_3 s_k s_3 w\}
      \end{equation*}

      \item  If $w \in D_{{\scC},i}(W)$ then
      \begin{equation*}
        T_{{\scC},i}(w) = D_{i,{\scC}}(W) \cap \{s_i w, s_3 s_j s_k w,s_3 s_j s_3 w,  s_3 s_k s_3 w\}
      \end{equation*}

      \item  If $w \in D_{i,{\scC}}(W)$ then
      \begin{equation*}
        T_{i,{\scC}}(w) = D_{{\scC},i}(W) \cap \{s_i w, s_j s_k s_3 w,s_3 s_j s_3 w,   s_3 s_k s_3 w\}
      \end{equation*}

      \item  If $w \in D_{i,{\scD}}(W)$ then
      \begin{equation*}
        T_{i,{\scD}}(w) = D_{{\scD},i}(W) \cap \{s_i w, s_j s_k s_3 w,s_3 s_j s_3 w, s_3 s_k s_3 w\}
      \end{equation*}
    \end{enumerate}
  \end{proposition}

  \begin{proof}
    Let's prove 3 first.  This is actually clear, by examination of \autoref{fig:type10Extra}, \autoref{fig:type14aExtra}, and \autoref{fig:type14bExtraColorChange}.
    If $w$ is of type $\scC$, we see that the one or two elements of type $\scA_i$ in each clump are in the set $\{s_i w, s_j s_k s_3 w,s_3 s_j s_3 w,
    s_3 s_k s_3 w\}$, and that the other elements of the set are not type $\scA_i$.

    For statement 2, if $w$ is of type $\scB_i$, the figures cited above show that the one or two elements of type $\scC$ in each clump are in the set $S_w = \{s_i w, s_3 s_j s_k w,s_3 s_j s_3 w, s_3 s_k s_3 w\}$.
    The figures don't display all the other elements of the set.
    However, notice that the four elements of $\WDFour$ which we're multiplying $w$ by to obtain $S_w$ are the inverses of the elements used in statement 3 to go from an element of type $\scC$ to an element of type $\scB_i$.
    So, if multiplying $w$ by of them led to an element of type $\scC$ which is not in the same clump as $w$, then the inverse would lead from that element back to our $w$ of type $\scB_i$.
    But we've already seen in the previous paragraph that from an element of type $\scC$, you only get to the elements of type $\scB_i$ which are in its clump.
    So, therefore, none of the other elements in $S_w$ are type $\scC$.

    The arguments for statements 4 and 1 are the same, starting with statement 4.
  \end{proof}

  here are some properties of the $D_4$ maps, which we'll need either in this paper or in future papers.

  \begin{proposition}
    \label{prop:d4OperatorsWOneTwo}
    Let $w \in W$. Let $i \in \refSet$.
    Let $j, k$ be such that $\refSet = \{i, j, k\}$.
    \begin{enumerate}
      \item Let $T$ be one of the maps of \myrefD{def:D4Ops}.  Let $w$ be in the domain of $T$.  Then
      $T(w)$, consists of one or two elements. Furthermore, we have the following.
      \begin{enumerate}
        \item If $T(w)=\{w'\}$ then
        $T(w')=\{w,w''\}$ with $w''\neq w$ and $T(w'')=\{w'\}$.
        \item If $T(w)=\{w_1,w_2\}$ with $w_1\neq w_2$ then
        $T(w_1)=T(w_2)=\{w\}$.
      \end{enumerate}

      \item If $w$ is of left type $\scC$ then
      \begin{equation*}
        T_{3,j}^L(T_{i,C}^L(w)) = T_{k,D}^L(T_{3,k}^L(w)).
      \end{equation*}
      In particular $\abs{T_{i,C}^L(w)}=\abs{T_{j,C}^L(w)}=\abs{T_{k,C}^L(w)}$, and similarly with $D$ in place of $C$.

      \item If $w$ is of left type $\scC$ and if $T_{i,C}^L(w)=\{w_1,w_2\}$ (where possibly $w_1 = w_2$) then
       \begin{equation*}
         w_2=(T_{k,3}^L\circ T_{3,i}^L\circ T_{j,3}^L\circ T_{3,k}^L\circ T_{i,3}^L\circ T_{3,j}^L)(w_1).
       \end{equation*}

      If $w$ is of left type $\scD$ and if $T_{i,D}^L(w)=\{w_1,w_2\}$ (where possibly $w_1 = w_2$) then
      \begin{equation*}
        w_2=(T_{3,k}^L\circ T_{i,3}^L\circ T_{3,j}^L
        \circ T_{k,3}^L\circ T_{3,i}^L\circ T_{j,3}^L)(w_1).
      \end{equation*}

      \item We have the corresponding statements with right in place of left.
    \end{enumerate}
  \end{proposition}

  \begin{proof}
    This can be seen by inspecting \refDFourFigures, taking into account \myrefP{prop:clumpDiagram}.
  \end{proof}

  Now, we'll connect our maps to the theorems of the last two sections.

  \begin{proposition}
    \label{prop:D4MainETP}
    The maps $T_i$ and $\bar T_i$ are type 2 edge transport functions.
  \end{proposition}

  \begin{proof}
    The first three conditions (for $T_i$ and $\bar T_i$) follow from the definition of the maps.
    The next two conditions (again for $T_i$ and $\bar T_i$) are \myrefT{thm:main}.
  \end{proof}

  To prove the same for the functions of \myrefD{def:D4Ops}, we'll relate those functions to the ones of \myrefD{def:D4MainMap}.

  \begin{proposition}
    \label{prop:D4MapRelations}
    Let $i \in \refSet$.
    Let $j, k$ be such that $\refSet = \{i, j, k\}$.
    We have:
    \begin{enumerate}
      \item $T_{i, \scC} = T_{k, 3} \circ T_j$
      \item $T_{\scC, i} = \bar T_k \circ T_{3, j}$
      \item $T_{i, \scD} = T_i \circ T_{i, 3}$
      \item $T_{\scD, i} = T_{3, i} \circ \bar T_i$
    \end{enumerate}
  \end{proposition}

  \begin{proof}
    This is clear from the definitions and the diagrams  \refDFourFigures, using also \myrefP{prop:clumpDiagram}.
  \end{proof}

  \begin{proposition}
    \label{prop:D4MapsETP}
    The functions of \myrefD{def:D4Ops} are type 2 edge transport functions.
  \end{proposition}

  \begin{proof}
    We'll use \myrefP{prop:edgeTransportPair}.
    Note that $T_{i, \scC}$ and $T_{\scC, i}$ are pair functions, as are $T_{i, \scD}$ and $T_{\scD, i}$.
    Conditions 1-3 of \myrefP{prop:edgeTransportPair} are stated in \myrefP{prop:d4OperatorsWOneTwo}--1.
    So, we just need to show that conditions 4 and 5 hold for all the maps $T$ listed in the proposition.
    This follows from \myrefP{prop:D4MainETP}, \myrefP{prop:D4MapRelations}, and \myrefT{thm:talbTilde}.
  \end{proof}

  Now, let's check the other properties of these functions.

  \begin{proposition}
    \label{prop:D4OpDomain}
    Let $D$ be the domain of one of the maps in \myrefD{def:D4Ops}.
    Then $D$ is a right KL interval set.
    Similarly, with left and right interchanged.
  \end{proposition}

  \begin{proof}
    This is just \myrefP{prop:D4OpRightCells}.
  \end{proof}

  \begin{proposition}
    \label{prop:D4KLCell}
    Let $T$ be one of the maps in \myrefD{def:D4Ops} or \myrefD{def:D4MainMap}.
    Then $T$ is a left KL cell map.
  \end{proposition}

  \begin{proof}
    This follows from \myrefP{prop:clumpSameCell}.
  \end{proof}

  As a consequence, we have these results for the $D_4$ maps.

  \begin{corollary}
    \label{cor:D4Conclusions}
    Let T be a let $D_4$ map with domain $D$.  We have
    \begin{enumerate}
      \item $T$ is right KL order preserving.
      \item Suppose $x,y \in D$ with $x \equivR y$.
      Then we can write $T(x) = \{x',x''\}$ and $T(y) = \{y',y''\}$ (where possibly $x' = x''$ and/or $y' = y''$) so that $x' \equivR y'$ and $x'' \equivR y''$.
      \item Let $C$ be a right cell contained in $D$.
      Then $T(C)$ is either a right cell or a union of two right cells.
    \end{enumerate}
    Similarly, interchanging left and right.
  \end{corollary}

  \begin{proof}
    \myrefP{prop:KnuthBEdgeTP} says that $T$ and $\pair(T) = T_{t,s}$ are type 2 edge transport functions.
    \myrefP{prop:BTwoKLCellFunction} implies that both are right \ti\ preserving.
    \myrefP{prop:BTwoKLSet} says that both domains are right KL interval sets.
    Then the conclusions of \myrefP{prop:edgeTransportLeqB}, \myrefP{prop:edgeTransportEquivalence}, and \myrefP{prop:cellsToCells} hold for $T$.
  \end{proof}

  Finally, we can conclude

  \begin{theorem}
    \label{thm:genTauConclusion}
    In the context of \myrefD{def:genTau}, let $\euscr F$ be a set of functions consisting of some combination of right Knuth maps, right $B_2$ maps, and right $D_4$ maps.
    Suppose $y, w \in W$ with $y \equivL w$.
    Then $y \equivGTF w$.
    Similarly, interchanging left and right.
  \end{theorem}

  \begin{proof}
    For the Knuth maps and $B_2$ maps, we have already seen this in the proof of \myrefT{thm:genTauBTwoConclusion}.
    For the $D_4$ maps, this combines \myrefT{thm:genTau} with \myrefC{cor:D4Conclusions}--1 and \myrefP{prop:D4OpDomain}.
  \end{proof}

  \section{Other Maps}
  \label{sec:otherMaps}
  There are other maps implicit in the situation of edge transport pairs.
  Though we won't make use of them, they appear elsewhere in the literature.
  Vogan defines maps called $S_{\alpha\beta}$ in Definition 4.6 of \cite{vogan_1980}.
  Lusztig in Section 10.6 of \cite{lusztig_1985} defines the analogous map as $w \mapsto \tilde w$.
  These maps are defined in relation to the $B_2$ maps.
  We'll make the definition here in the more general context of type 2 edge transport functions and show the maps' properties.
  The definition and properties will then also apply to the maps defined in \myrefS{sec:D4Maps}.

  \begin{definition}
    \label{def:derivedMap}
    Let $T$, $\bar T$,  and $D$ be as in \myrefP{prop:edgeTransportPair}.
    We define an associated map $U: D \longrightarrow D$ as follows.
    For $w \in D$, if $\abs{T(w)} = 2$ then $U(w) = w$.
    Otherwise, let $T(w) = \{w'\}$, and let $w^*$ be such that $\bar T = \{w, w^*\}$. Then we set $U(w) = \tilde w$.
  \end{definition}

  \begin{remark}
    With $U$ as in \myrefD{def:derivedMap}, we have $U^{-1} = U$.
    Also, Vogan uses the notation $S$ with a subscript for the associated maps.  For us, with that convention, we would start with a map $T_{s,t}$, with $st$ of order 4, and write $S_{s,t}$ for the associated map.
    Similarly, we can write $S_{\scC,i}$, etc., for the maps associated to $D_4$ maps.
  \end{remark}

  \begin{remark}
    Let $T$ and $U$ be as in \myrefD{def:derivedMap}.
    Since $U$ has the same domain as $T$, if we start with a map $T$ whose domain is a left (resp.\ right) KL interval set, then the domain of the derived function $U$ has the same property.
  \end{remark}

  \begin{proposition}
    \label{prop:derivedMap}
    Let $T$ and $U$ be as in \myrefD{def:derivedMap}.
    Then $U$ is a type 1 edge transport function.
    If $T$ is a left (resp.\ right) KL cell function, then so is $U$.
    If $T$ is left (resp.\ right) \ti\ preserving, then so is $U$.
  \end{proposition}

  \begin{proof}
    To see that $U$ is a type 1 edge transport function, we use conditions 4 and 5 of \myrefP{prop:edgeTransportPair}.
    That is, let $y, w \in D$.
    If $U(y) = y$ and $U(w) = w$, then clearly $\mutilde(U(y), U(w)) = \mutilde(y, w)$.
    If $U(y) \neq y$ and $U(w) \neq w$, then $\mutilde(U(y), U(w)) = \mutilde(y, w)$ follows from condition 4 applied to $\pair(T)$.
    If $U(y) = y$ and $U(w) \neq w$, then $\mutilde(U(y), U(w)) = \mutilde(y, w)$ follows from condition 5.
    Also, since $U^{-1} = U$, we know that $U$ is an injection.

    To see that $U$ is a left KL cell function when $T$ is, note that if $w \in D$ with $U(w) = w^* \neq w$, then $T(w) = T(w^*)$, so $w \equivL T(w) \equivL w^*$.
    We show similarly that $U$ is left \ti\ preserving when $T$ is.
  \end{proof}

  \begin{corollary}
    Let $T$, $U$, and $D$ be as in \myrefD{def:derivedMap}
    Assume in addition that $D$ is a right KL interval set and that $T$ is right \ti\ preserving.
    Then $U$ satisfies the hypotheses of \myrefP{prop:edgeTransportKLOrderA} and of \myrefC{cor:edgeTransportEquiv}.
    In particular, we have their conclusions for $U$.
  \end{corollary}

  \begin{proof}
    Mostly this follows from \myrefP{prop:derivedMap}. For \myrefC{cor:edgeTransportEquiv}, we also note that $U^{-1} = U$.
  \end{proof}

  \begin{remark}
    This replicates Theorem 4.8 of \cite{vogan_1980} when $W$ is a Weyl group, and part of Proposition 10.7 of \cite{lusztig_1985}.
  \end{remark}

  \section{Techniques of Strings and Clumps}
  \label{sec:techniques}

  In section 10.5 of \cite{lusztig_1985}, Lusztig describes the technique of strings.
  The technique of strings is just the application of \myrefT{thm:stStrings} or \myrefT{thm:stsStrings} to start with a known edge in the $W$ graph (for example one given by multiplication by an element of $S$) and deduce from it and the theorem the presence of a hitherto unknown edge.
  With good luck, this new edge will be one which contributes to the equivalence relation $\underset{L}{\leq}$.

  A simple example of this can be seen in \autoref{fig:type10Full}.
  Write $w$ for the element of type $\scC$ at the bottom of such a picture.
  Then $s_3w$ and $s_1s_3w$ are connected by multiplication by an element of $S$, namely $s_1$, so $\mu(s_3w, s_1s_3w) = 1$.
  Now, $w = T_{4,3}(s_3w)$ and $s_4s_1s_3w = T_{4,3}(s_1s_3w)$.
  From \myrefT{thm:stStrings}, we can then deduce that $\mu(w, s_4s_1s_3w) = 1$.
  This is one of the gray edges shown in \autoref{fig:type10Full}.
  This edge has the property that $\tau(s_4s_1s_3w) \subsetneq \tau(w)$, and thus shows that $s_4s_1s_3w \underset{L}{\leq} w$, and thus that the middle six elements are in the same left cell as the bottom two.

  Now let's look at an example of using the analogous ``technique of clumps''.  This example is in the Weyl group of type $E_6$.
  We'll use a standard numbering of the nodes of the Dynkin diagram, as shown below.

  
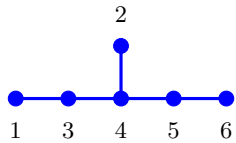
\begin{figure}[H]
    \centering
    \begin{tikzpicture}[scale=.7]
        \DynkinESixLabelBourbaki
    \end{tikzpicture}
  \end{figure}

  Let $w = s_1s_3s_1s_5s_6s_5s_2$.
  \autoref{fig:E6} shows part of the left cell containing $w$, with $w$ the element at the bottom of the diagram.

  \begin{figure}[!ht]
    \begin{center}
      \begin{tikzpicture}[scale=0.58, every node/.style={scale=.4}]
        \node (CENode) at (0, 0) {\TypeCPictureE};
        \node at (0, \ySE) {\ESixPicture1355};
        \node (DE1Node) at (-1 * \xSE, 2 * \ySE) {\TypeDOnePictureE};
        \node at (0, 2 * \ySE) {\ESixTopPicture135};
        \node (DE5Node) at (1 * \xSE, 2 * \ySE) {\TypeDFivePictureE};
        \node (AE1Node) at (-2 * \xSE, 3 * \ySE) {\TypeAOnePictureE};
        \node at (-1 * \xSE, 3 * \ySE) {\ESixTopPicture255};
        \node at (0, 3 * \ySE) {\ESixPicture2444};
        \node at (1 * \xSE, 3 * \ySE) {\ESixTopPicture144};
        \node (AE5Node) at (2 * \xSE, 3 * \ySE) {\TypeAFivePictureE};
        \node at (-2 * \xSE, 4 * \ySE) {\ESixPicture1244};
        \node at (-1 * \xSE, 4 * \ySE) {\ESixTopPicture125};
        \node at (0, 4 * \ySE) {\ESixTopPicture244};
        \node at (1 * \xSE, 4 * \ySE) {\ESixTopPicture145};
        \node at (2 * \xSE, 4 * \ySE) {\ESixPicture2455};
        \node at (0, 5 * \ySE) {\ESixPicture3333};
        \node (BE1Node) at (-1 * \xSE, 5 * \ySE) {\TypeBOnePictureE};
        \node (BE5Node) at (1 * \xSE, 5 * \ySE) {\TypeBFivePictureE};

        \upLineE{darkblue}{solid}{0}{0}
        \upLineE{goldenpoppy}{dashed}{0}{1}
        \rightLineE{green}{solid}{0}{1}
        \leftLineE{magenta}{solid}{0}{1}
        \rightLineE{green}{solid}{-1}{2}
        \upLineE{goldenpoppy}{solid}{-1}{2}
        \leftLineE{magenta}{solid}{0}{2}
        \leftLineE{magenta}{solid}{1}{2}
        \upLineE{goldenpoppy}{solid}{1}{2}
        \rightLineE{green}{solid}{0}{2}
        \leftLineE{magenta}{solid}{1}{3}
        \rightLineE{green}{solid}{-1}{3}
        \upLineE{goldenpoppy}{dashed}{0}{3}
        \upLineE{darkblue}{solid}{0}{4}
        \draw[-, very thick, shorten <=0.3cm, shorten >=.6cm, violet, dotted] (-1 * \xSE, 2 * \ySE) -- (-1 * \xSE - \xSE, 2 * \ySE + \ySE);
        \draw[-, very thick, shorten <=.3cm, shorten >=0.6cm, cyan, dotted] (1 * \xSE, 2 * \ySE) -- (1 * \xSE + \xSE, 2 * \ySE + \ySE);
        \upLineE{green}{solid}{-2}{3}
        \rightLineE{goldenpoppy}{solid}{-2}{3}
        \upLineE{violet}{dashdotted}{-1}{3}
        \upLineE{cyan}{dashdotted}{1}{3}
        \leftLineE{goldenpoppy}{solid}{2}{3}
        \upLineE{magenta}{solid}{2}{3}
        \rightLineE{goldenpoppy}{dashdotted}{-2}{4}
        \upLineE{green}{solid}{-1}{4}
        \upLineE{magenta}{solid}{1}{4}
        \leftLineE{goldenpoppy}{dashdotted}{2}{4}

        \draw[bend right = 60, thick, gray, dotted] (BE1Node) to (CENode);
        \draw[bend left = 60, thick, gray, dotted] (BE5Node) to (CENode);

        \matrix [draw, below right, every node/.style={scale=.7}]
        at (1.75 * \xSE, 6 * \ySE)
        {
          \node [legendLine, draw=violet,label=right:{$s_1$}] {}; \\
          \node [legendLine, draw=goldenpoppy,label=right:{$s_2$}] {}; \\
          \node [legendLine, draw=magenta,label=right:{$s_3$}] {}; \\
          \node [legendLine, draw=darkblue,label=right:{$s_4$}] {}; \\
          \node [legendLine, draw=green,label=right:{$s_5$}] {}; \\
          \node [legendLine, draw=cyan,label=right:{$s_6$}] {}; \\
        };
      \end{tikzpicture}
    \end{center}
    \caption{$E_6$ Example}
    \label{fig:E6}
  \end{figure}
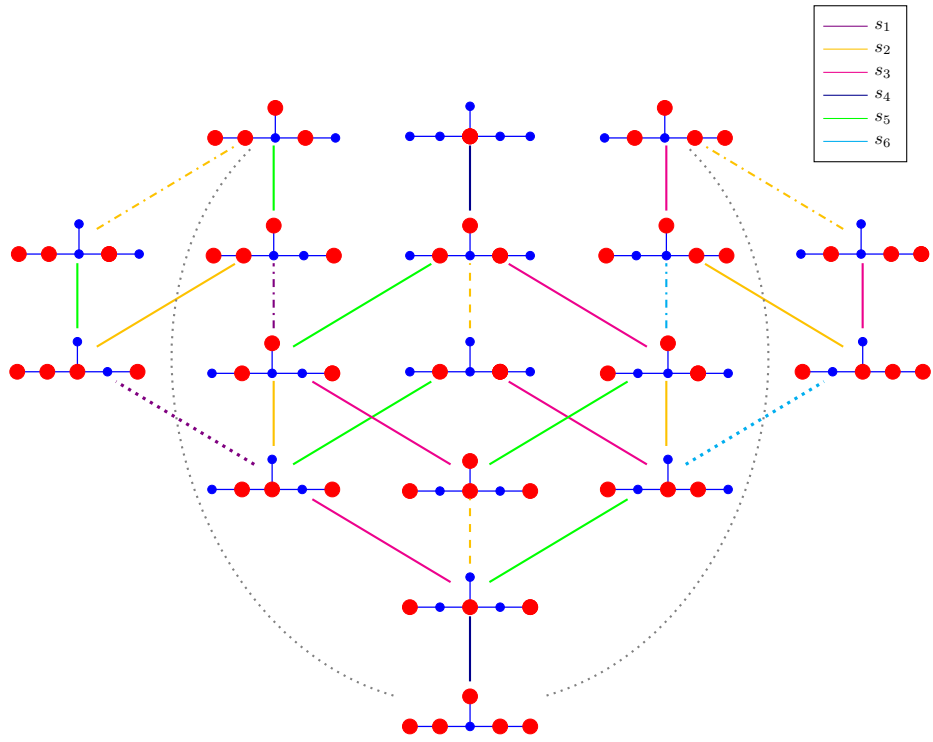

  The element $w$ is of type $\scC$ for the parabolic subgroup generated by $\{s_2, s_3, s_4, s_5\}$, and is part of a clump of size 10, all of which is shown in the diagram.
  Then $s_3s_4w$ is type $\scA_{s_3}$, as is $y = s_1s_3s_4w$.
  The element $y$ is in a clump of size 14, only four of whose elements is shown in the diagram.
  The element of type $\scC$ in the same clump as $y$ is $s_5s_2y$ and is shown in the diagram.
  We have obviously $\mu(s_3s4w, y) = 1$.
  Since the clump containing $s_3s_4w$ and the clump containing $y$ have different sizes, \myrefT{thm:mainA}--2 applies and says that $\mu(w, s_5s_2y) = 1$.
  (This is the curved edge shown in gray on the left in the diagram.)
  In particular, we can conclude that $w$ and $y$ are in the same left cell.

  The technique of strings is used for example in \cite{lusztig_1985} and \cite{bedard_1986} as part of their work computing left cells in certain low-rank affine Weyl groups.
  It's hoped that the edges transport theorem of this paper can have similar applications.

  The two-sided cell in the Weyl group of type $E_6$ containing the elements shown in \autoref{fig:E6} is in many ways analogous to the two-sided cell in $D_4$ which is the subject of this paper.
  Hopefully one can prove an edge transfer theorem for the $E_6$ cell as well.
  I think the methods of this paper should work there in principle.
  However, the $E_6$ cell is a lot larger than the $D_4$ cell, so the parts of the $D_4$ proof which go case by case would be harder to carry out in practice.

  \section*{Acknowledgements} I would like to thank the University of Pennsylvania for their kind hospitality while this paper was being written.
  I would like to thank David Harbater for much helpful advice about writing this paper.
  I would like to thank my son Christian Johnson, and my friend Leila Miller, for their support in writing this paper.
  I would like to thank my son Robert Johnson for writing programs to compute and draw cells in $E_6$ and $E_7$.
  I would like to thank Mike Chmutov and Joel Brewster Lewis for helpful comments about earlier versions of this paper.
  I would like to thank the creator and maintainers of the \LaTeX Tikz package, without which this research would not have been possible.
  Finally, I would like to thank Professor Ariki for his paper \cite{ariki_2000}, which was the inspiration for this one.

\end{document}